%% file: urnsfull_annals_arxiv.tex
\documentclass[aop]{imsart}
\usepackage[utf8]{inputenc}

\RequirePackage[OT1]{fontenc}

\pubyear{2019}
\volume{0}
\issue{0}
\firstpage{1}
\lastpage{56}

\received{March 2019}
\revised{October 2019}
\accepted{October 2019}

\def\internallinenumbers{} 

\usepackage[top=1.45in, bottom=1.45in, left=1.5525in, right=1.5525in]{geometry}

\usepackage{blowup}
\blowUp{origin=x1.15,target={202.6mm,320mm},pos=c%
}

\usepackage[T1]{fontenc}
\usepackage{lmodern}
\usepackage{comment}

\usepackage{graphicx} 

\usepackage{etoolbox}
\usepackage{tikz,mathtools}
\mathtoolsset{showonlyrefs}  
\newrobustcmd*{\mytriangle}{\tikz{\draw (0,0) -- (0.2cm,0) -- (0.1cm,0.3cm) -- (0,0);}}

\usepackage{amsmath,amssymb,color,amsthm}

\usepackage{tcolorbox}
\usepackage[framemethod=tikz,middlelinewidth=.5mm,leftmargin=5cm,innerleftmargin=1cm]{mdframed} 
\newcounter{algocounter}
\newenvironment{coloralgorithm}[1][]{\refstepcounter{algocounter}
\begin{tcolorbox}[colback=red!5!white,colframe=red!75!black,adjusted title={Algorithm~\thealgocounter~(Output:~#1})]\internallinenumbers
\vspace{-\baselineskip}
~  \internallinenumbers
}{\end{tcolorbox}}

\newtheorem{theorem}{Theorem}[section]
\newtheorem{definition}[theorem]{Definition}

\newtheorem{proposition}[theorem]{Proposition}
\newtheorem{lemma}[theorem]{Lemma}
\newtheorem{remark}[theorem]{Remark}
\newtheorem{example}[theorem]{Example}
\AtEndEnvironment{example}{\null\hfill$\blacksquare$}%

\usepackage{MnSymbol} 
\usepackage{cite} 

\usepackage{xcolor}
\usepackage{hyperref}
\definecolor{darkgreen}{rgb}{0,0.4,0}
\definecolor{BrickRed}{rgb}{0.65,0.08,0}
\hypersetup{colorlinks=true,linkcolor=blue,citecolor=darkgreen,filecolor=BrickRed,urlcolor=darkgreen}

\newcommand{\LandauO}{O}
\newcommand{\Landauo}{o}
\newcommand{\E}{\mathbb{E}} 
\newcommand{\N}{\mathbb{N}} 
\newcommand{\Dc}{\mathcal{D}} 
\newcommand{\PR}{\mathbb{P}} 
\newcommand{\R}{\mathbb{R}} 
\newcommand{\Vb}{\mathbb{V}} 
 
\newcommand{\Prob}{\PR}

\newcommand{\Ac}{\mathcal{A}}

\newcommand{\Pc}{\mathcal{P}}

\newcommand{\Sc}{\mathcal{S}}
\newcommand{\Tc}{\mathcal{T}}

\newcommand{\Yc}{\mathcal{Y}}

\newcommand{\OEIS}[1]{\text{\href{https://oeis.org/#1}{{\small \tt #1}}}} 
\def\GG{\operatorname{GenGamma}}

\def\PGG{\operatorname{GenGammaProd}}
\def\bGamma{{\mathbf \Gamma}}
\def\PGGgen{\operatorname{GenGammaProd}}
\gdef\V{v_m}
\gdef\Yn{X_n}
\gdef\Ynv{E_\Yc(v)}

\usepackage{ifthen}
\def\longversionmode{no}        
\newcommand{\longversion}[2]{\ifthenelse{\equal{\longversionmode}{yes}}{%
#1%
\if\relax\detokenize{#1}\relax\ignorespaces\fi  
}{%
#2%
\if\relax\detokenize{#2}\relax\ignorespaces\fi  
}}

\newcommand{\ffac}[2]{{#1}^{\underline{#2}}}
\newcommand{\rfac}[2]{{#1}^{\overline{#2}}}

\usepackage{calc}
\newcommand*{\MyDef}{\, \mathcal{L}\,}
\newcommand*{\eqdefU}{\ensuremath{\mathop{\overset{\MyDef}{=}}}}
\renewcommand*{\=}{\,\mathop{\overset{\MyDef}{\resizebox{\widthof{\eqdefU}}{\heightof{=}}{=}}}\,} 
\def\inlaw{\quad \stackrel{\scriptstyle \mathcal L}{\longrightarrow} \quad } 

\def\fill{\operatorname{ext}}
\def\ext{\operatorname{ext}}
\def\Ss{|\Sc|}

\newcommand{\addorcid}[1]{\protect\includegraphics[height=3mm]{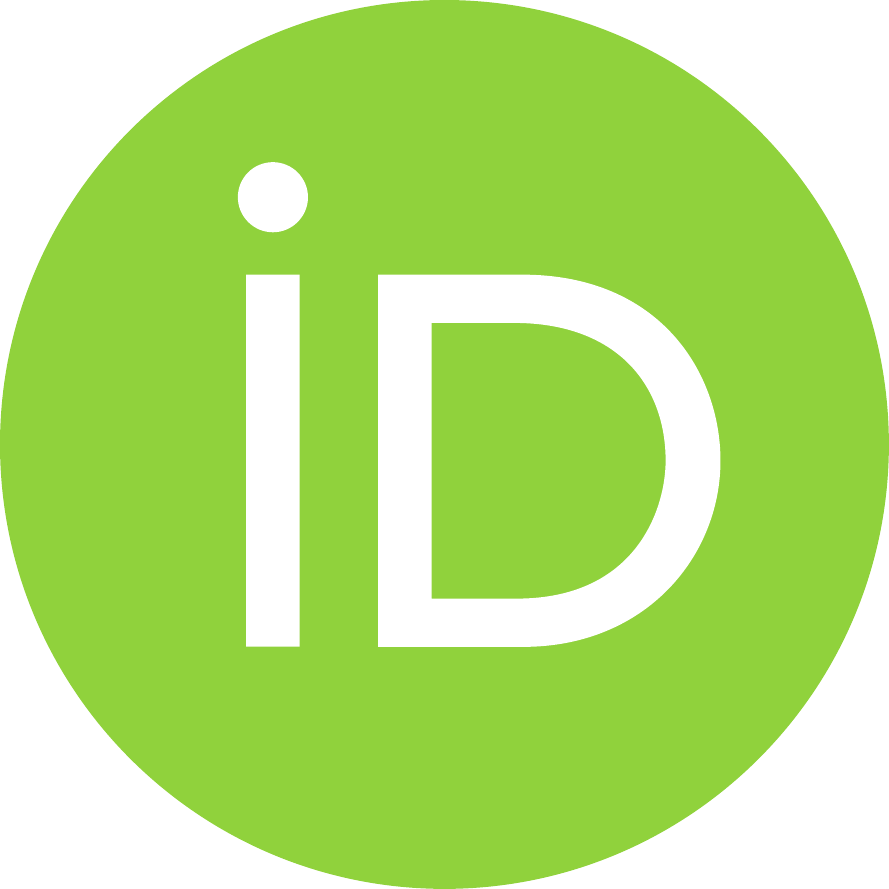} \href{{https://orcid.org/#1}}{#1}}
\newcommand{\addurl}[1]{URL: \url{#1}}

\begin{document}

\begin{frontmatter}

\title{Periodic P\'olya Urns, the Density Method,\\ and Asymptotics of Young Tableaux}
\runtitle{P\'olya Urns and Young Tableaux}

\begin{aug}
  \author{\fnms{Cyril}  \snm{Banderier}\thanksref{m1}
  },
  \author{\fnms{Philippe} \snm{Marchal}\thanksref{m1}
  }
  \and
  \author{\fnms{Michael}  \snm{Wallner}\corref{}\thanksref{m2}\thanksref{t3}
  }

  \thankstext{t3}{Michael Wallner was supported by a MathStic funding from the University Paris Nord September--December 2017 and the Erwin Schr{\"o}dinger Fellowship of the Austrian Science Fund (FWF):~J~4162-N35.}

  \runauthor{C. Banderier, P. Marchal, M. Wallner}

  \affiliation{Universit\'e Paris 13\thanksmark{m1} and Universit\'e de Bordeaux\thanksmark{m2}}
  
\address{Cyril Banderier\\
CNRS; Universit\'e Paris 13, LIPN\\
99, avenue Jean-Baptiste Cl\'ement\\
93430 Villetaneuse, France\\
\addurl{https://lipn.fr/~banderier}\\
\addorcid{0000-0003-0755-3022}}

\address{Philippe Marchal\\
CNRS; Universit\'e Paris 13, LAGA\\
99, avenue Jean-Baptiste Cl\'ement\\
93430 Villetaneuse, France\\
\addurl{https://math.univ-paris13.fr/~marchal/}\\
\addorcid{0000-0001-8236-5713}}

\address{Michael Wallner\\
Universit\'e de Bordeaux, LaBRI\\
351 Cours de la Lib\'eration\\
33405 Talence Cedex, France\\
\addurl{https://dmg.tuwien.ac.at/mwallner/}\\
\addorcid{0000-0001-8581-449X}}
\end{aug}

\begin{abstract}
\centering\begin{minipage}{0.775\textwidth}
\qquad P\'olya urns are urns where at each unit of time a ball is drawn and replaced with some other balls according to its colour.
We introduce a more general model: the replacement rule depends on the colour of the drawn ball \textit{and} the value of the time ($\operatorname{mod} p$). 
We extend the work of Flajolet et~al.~on P\'olya urns: the generating function encoding the evolution of the urn is studied by methods of analytic combinatorics.
We show that the initial  \textit{partial} differential equations lead to \textit{ordinary} linear differential equations which are related 
to
hypergeometric functions (giving the exact state of the urns at time~$n$). 
When the time goes to infinity, we prove that these \textit{periodic P\'olya urns} 
have asymptotic fluctuations which are described by a product of generalized gamma distributions. 
With the additional help of what we call the {\em density method} (a method which offers access to enumeration and random generation of poset structures),
we prove that the law of the south-east corner of a triangular Young tableau follows asymptotically a product 
of generalized gamma distributions.
This allows us to tackle some questions related to the continuous limit of large random Young tableaux and links with random surfaces.
\end{minipage}
\end{abstract}



\end{frontmatter}


{\footnotesize 
\tableofcontents 
}

\input{article_arxiv}

\let\OLDthebibliography\thebibliography
\longversion{\renewcommand\thebibliography[1]{
  \OLDthebibliography{#1}
  \setlength{\parskip}{0pt}
  \setlength{\itemsep}{5pt plus 0.9ex}
}}{
\renewcommand\thebibliography[1]{
  \OLDthebibliography{#1}
  \setlength{\parskip}{0pt}
  \setlength{\itemsep}{3pt plus 0.9ex}
}}

\begingroup
\renewcommand{\specialsection}[2]{}
\smallskip
\begin{center}
\MakeUppercase\refname
\end{center}
\vspace{-2mm}
{\scriptsize
The links written as [doi] may require a subscription. 
Whenever available, we additionally give a link towards a free version of the article: just click on the title! 
}

\bibliographystyle{cyrbiburl}
\longversion{\linespread{.98}\selectfont\bibliography{../aofa2018_urns}}
{\linespread{0.98}

\bibliography{aofa2018_urns_short}}
\endgroup

\end{document}

%% file: article_arxiv.tex
\section{Introduction}\label{sec:urn}
\subsection{Periodic P\'olya urns}

\textit{P\'olya urns} were introduced in a simplified version by George P\'olya and his PhD student Florian Eggenberger 
in~\cite{EggenbergerPolya23,EggenbergerPolya28,Polya30}, with applications to disease spreading and conflagrations.
They constitute a powerful model, which regularly finds new applications: see e.g.~Rivest's recent work on auditing elections~\cite{Rivest18}, or the analysis of deanonymization in Bitcoin's peer-to-peer network~\cite{FantiViswanath17}. 
They are well-studied objects in combinatorial and probabilistic literature~\cite{AthreyaNey04,FlajoletGabarroPekari05,Mahmoud09},
because they offer fascinatingly rich links with numerous objects like random recursive trees, $m$-ary search trees, and branching random walks~(see 
e.g.~\cite{BagchiPal85, Janson04,Janson05,ChauvinMaillerPouyanne15}).
In this paper we introduce a variation which leads to new links with another important combinatorial structure: Young tableaux.
What is more, we solve the enumeration problem of this new P\'olya urn model,
derive the limit law for the evolution of the urn, and give some applications to Young tableaux. 

\smallskip

In the \textit{P{\'o}lya urn model}, one starts with an urn with $b_0$ black balls and $w_0$ white balls at time~$0$. 
At every discrete time step one ball is drawn uniformly at random.
After inspecting its colour this ball is returned to the urn. 
If the ball is black, $a$ black balls and $b$ white balls are added; if the ball is white, $c$ black balls and $d$ white balls are added (where $a,b,c,d \in \N$ are non-negative integers). 
This process can be described by the so-called \textit{replacement matrix}:
\begin{align*}
	M &=
		\begin{pmatrix}
			a & b \\
			c & d
		\end{pmatrix},
		\qquad
		a,b,c,d \in \N.
\end{align*}
\indent We call an urn and its associated replacement matrix \textit{balanced} if $a+b = c+d$.
In other words, in every step the same number of balls is added to the urn.
This results in a deterministic number of balls after $n$ steps: $b_0 + w_0 + (a+b) n$ balls.

Now, we introduce a more general model which has rich combinatorial, probabilistic, and analytic properties. 

\begin{definition}
	A \textit{periodic P\'olya urn} of period~$p$ with replacement matrices~${M_1,M_2,\ldots,M_p}$
is a variant of a P\'olya urn in which the replacement matrix $M_k$ is used at steps $np+k$. 
Such a model is called \textit{balanced} if each of its replacement matrices is balanced. 
\end{definition}

For $p=1$, this model reduces to the classical model of P\'olya urns with one replacement matrix. 
In this article, we illustrate the aforementioned rich properties via the following model.

\begin{definition}
	\label{def:youngpolyaurn}
	Let $p,\ell \in \N$. We call a \textit{Young--P\'olya urn of period $p$ and parameter~$\ell$} the periodic P\'olya urn of period~$p$ (with 
$b_0\geq 1$ to avoid degenerate cases) and replacement matrices
\begin{equation*}
M_1= M_2 = \dots = M_{p-1} = 
	\begin{pmatrix}
			1 & 0 \\
			0 & 1
		\end{pmatrix} 
		\text{\ and \ }
		M_p=
		\begin{pmatrix}
			1 & \ell \\
			0 & 1+\ell
		\end{pmatrix}.
\end{equation*}
\end{definition}

\begin{example}
Consider a Young--P{\'o}lya urn with parameters $p=2$, $\ell=1$, and initial conditions $b_0=w_0=1$.
The replacement matrices are~$M_1:=
\begin{pmatrix}
	1 & 0 \\
	0 & 1
\end{pmatrix}$ for every odd step, and~$M_2:=
\begin{pmatrix}
	1 & 1 \\
	0 & 2
\end{pmatrix}$ for every even step.
This case was analysed by the authors in the extended abstract~\cite{BanderierMarchalWallner18}.
In the sequel, we will use it as a running example to explain our results.

Let us illustrate the evolution of this urn in Figure~\ref{fig:youngurnevolution}.
Each node of the tree corresponds to the current composition of the urn 
(number of black balls, number of white balls).
One starts with $b_0=1$ black ball and $w_0=1$ white.
In the first step, the matrix $M_1$ is used and leads to two different compositions. 
In the second step, matrix $M_2$ is used,
in the third step, matrix $M_1$ is used again, in the fourth step, matrix $M_2$, etc.
Thus, the possible compositions are
	$(2,1)$ and $(1,2)$ at time $1$,
	$(3,2), (2,3)$ and $(1,4)$ at time~$2$, 
	$(4,2), (3,3), (2,4)$ and $(1,5)$ at time $3$. 
	\end{example}

\begin{figure}[!ht]
		\begin{center}	
\begin{flushleft}
\begin{tabular}{@{}ll@{}}
\begin{tabular}{@{}l@{}}	\includegraphics[width=.65\textwidth]{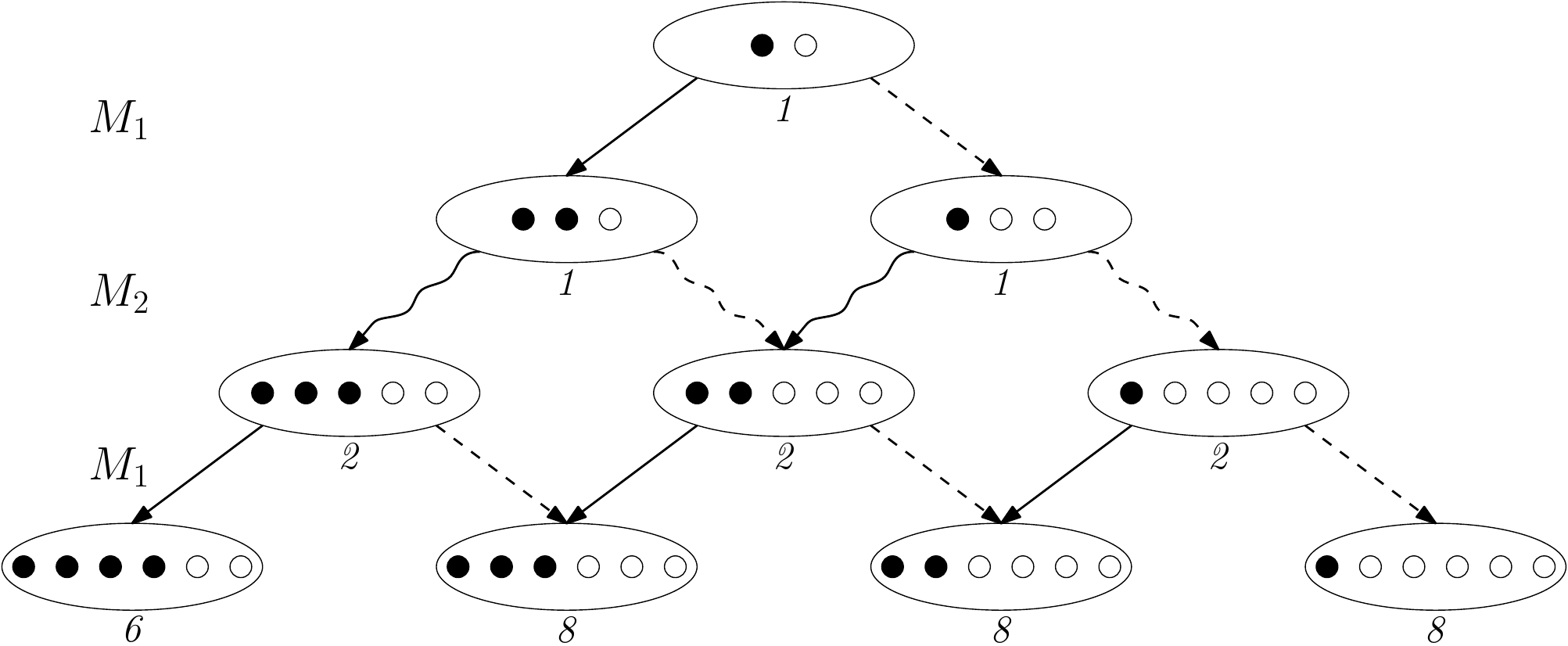} \end{tabular}
&	\hspace{-3.3mm} 		
\scalebox{.99}{\begin{tabular}{@{}l@{}}\\[0mm]
		 $h_0=xy$\\[7mm]
 $h_1=x^2y+xy^2$\\[7mm]
 $h_2=2x^3y^2+2x^2y^3+2xy^4$\\[7mm]
 $h_3=6x^4y^2+8x^3y^3+8x^2y^4+8xy^5$ \\[1cm]
			\end{tabular}} 
			\end{tabular}
			\end{flushleft}\internallinenumbers
			\setlength{\abovecaptionskip}{-3mm}
			\setlength{\belowcaptionskip}{-3mm}
			\caption{The evolution of the Young--P\'olya urn with period $p=2$ and parameter~$\ell=1$ with one initial black and one initial white ball. 
			Black arrows mark that a black ball was drawn, dashed arrows mark that a white ball was drawn. 
			Straight arrows indicate that the replacement matrix $M_1$ was used, curly arrows show that the replacement matrix $M_2$ was used.
			The number below each node is the number of possible transitions to reach this state. 
			In this article we give a formula for $h_n$ (which encodes all the possible states of the urn at time $n$) 
			and the corresponding asymptotic behaviour.}
			\label{fig:youngurnevolution}
		\end{center}
\end{figure}

\pagebreak 

In fact, each of these states may be reached in different ways, and such a sequence of transitions is called a \textit{history}.
(Some authors also call it a \textit{scenario}, an \textit{evolution}, or a \textit{trajectory}.) 
Each history comes with weight one. Implicitly, they induce a probability measure on the states at step $n$.
So, let $B_n$ and $W_n$ be random variables for the number of black and white balls after $n$ steps, respectively. 
As our model is balanced, $B_n + W_n$ is a deterministic process, reflecting the identity 
\begin{equation}B_n + W_n = b_0 + w_0 + n + \ell \left\lfloor\frac{n}{p}\right\rfloor.\end{equation}
So, from now on, we concentrate our analysis on $B_n$.

\subsection{The generalized gamma product distribution}

For the classical model of a single balanced P{\'o}lya urn, the limit law of the random variable $B_n$ is fully known:
the possible limit laws include a rich variety of distributions. 
To name a few, let us mention 
the uniform distribution~\cite{FlajoletDumasPuyhaubert06}, 
the normal distribution~\cite{BagchiPal85}, and the beta and Mittag-Leffler distributions~\cite{Janson04,Janson06}. 
Now, periodic P\'olya urns (which include the classical model) lead to an even larger variety of distributions 
involving a product of \textit{generalized gamma distributions}~\cite{Stacy62}. 

\begin{definition}\label{def:GG} 
	The generalized gamma distribution $\GG(\alpha,\beta)$ with real parameters $\alpha,\beta>0$ is defined on $(0,+\infty)$ by the density function 
	\begin{align*}
		f(t; \alpha,\beta) &:= \frac{ \beta \, t^{\alpha-1} \exp(-t^\beta)}{\Gamma\left(\alpha/\beta\right)},
	\end{align*}
	where $\Gamma$ is the classical gamma function $\Gamma(z):= \int_0^\infty t^{z-1}\exp(-t) \, dt$.
\end{definition}

The fact that $f(t; \alpha,\beta)$ is indeed a probability density function
can be seen by a change of variable $t\mapsto t^\beta$ in the definition of the $\Gamma$ function, or via the following link.

\begin{remark}	
	Let ${\mathbf \Gamma}(\alpha)$ be the gamma distribution\footnote{Caveat: it is traditional to use the same letter for both the $\Gamma$ function and the $\mathbf \Gamma$ distribution.
Also, some authors add a second parameter to the $\mathbf \Gamma$ distribution, which is set to $1$ here.} of parameter $\alpha>0$, 
given on $(0,+\infty)$ by 
	\begin{align*}	
		g(t; \alpha) &= \frac{t^{\alpha-1} \exp(-t)}{\Gamma(\alpha)}.
	\end{align*}	
Then, one has ${\mathbf \Gamma}(\alpha)\= \GG(\alpha,1)$ and, for $r>0$, 
the distribution of the $r$-th power
of a random variable distributed according to ${\mathbf \Gamma}(\alpha)$ is 
\begin{equation*}{\mathbf \Gamma}(\alpha)^r\=\GG(\alpha/r,1/r).\end{equation*}
\end{remark}

The limit distribution of our urn models is then expressed as a product of such generalized gamma distributions.
We prove in Theorem~\ref{ProdGenGammaGeneral} a more general version of the following:

\begin{theorem}[The generalized gamma product distribution $\PGG$ for Young--P\'olya urns] \label{theo:PGG}
The renormalized distribution of black balls in a Young--P\'olya urn of period~$p$ and parameter~$\ell$
is asymptotically for $n \to \infty$ given by the following product of distributions:
\begin{equation}\label{ProdGenGamma}
\frac{p^\delta}{p+\ell} \frac{B_n }{n^\delta}\inlaw \operatorname{Beta}(b_0,w_0) \prod_{i=0}^{\ell-1} \GG(b_0+w_0+p+i, p+\ell),
\end{equation}
with $\delta=p/(p+\ell)$, and $\operatorname{Beta}(b_0,w_0)=1$ when $w_0=0$ or 
$\operatorname{Beta}(b_0,w_0)$ is the beta distribution with support $[0,1]$ and density $\frac{\Gamma(b_0+w_0)}{\Gamma(b_0)\Gamma(w_0)} x^{b_0-1} (1-x)^{w_0-1}$ otherwise.
\end{theorem}

In the sequel, we call this distribution the \textit{generalized gamma product distribution} and denote it by $\PGG(p,\ell,b_0,w_0)$.
We will see in Section~\ref{sec:Moments} that this distribution is characterized by its moments, which have a nice factorial shape given in Formula~\eqref{m_r}.

\begin{example}
	\label{ex:YoungPolyaMainresult}
	In the case of the Young--P\'olya urn with $p=2$, $\ell=1$, and $w_0=b_0=1$, one has $\delta=2/3$. Thus, the previous result shows that the number of black balls
	converges in law to a generalized gamma distribution:
		\begin{align*}
\frac{2^{2/3}}{3} \frac{B_n}{n^{2/3}} & \inlaw \operatorname{Unif}(0,1) \cdot \GG(4,3) \quad = \quad \GG\left(1, 3 \right).\end{align*}
See Section~\ref{duality} and~\cite[Proposition~4.2]{Dufresne10} for more identities of this type.
\end{example}

\begin{remark}[Period one]\label{rem:p1}
When $p=1$, our results recover a classical (non-periodic) urn behaviour.	By~\cite[Theorem~1.3]{Janson06} the renormalization for the limit distribution of $B_n$ in an urn with replacement matrix 
	$\begin{pmatrix}
	1 & \ell \\
	0 & 1+\ell
\end{pmatrix}$
	is equal to $n^{-1/(1+\ell)}$. 
	For $\ell=0$ the limit distribution is the uniform distribution, whereas for $\ell=1$ it is a Mittag-Leffler distribution (see~\cite[Example~3.1]{Janson06}, \cite[Example~7]{FlajoletDumasPuyhaubert06}),
and even simplifies to a half-normal distribution\footnote{See~\cite{Wallner16} for other occurrences of the half-normal distribution in combinatorics.}
 when $b_0=w_0=1$.
	Thus, the added periodicity by using this replacement matrix only every $p$-th round and otherwise P\'olya's replacement matrix
	$\begin{pmatrix}
	1 & 0 \\
	0 & 1
	\end{pmatrix}$
	changes the renormalization to $n^{-p/(p+\ell)}$.
\end{remark}

\pagebreak 

	The rescaling factor $n^{-\delta}$ with $\delta=p/(p+\ell)$ on the left-hand side of~\eqref{ProdGenGamma} can also be obtained via a martingale computation. The true challenge is to get exact enumeration and the limit law.
It is interesting that there exist other families of urn models exhibiting the same rescaling factor, 
	however, these alternative models lead	to different limit laws.
	\smallskip
	\begin{itemize}
		\item A first natural alternative model consists in averaging
		 the $p$ replacement matrices. This leads to a classical triangular P\'olya urn model. 
		 The asymptotics is then 
		 \begin{align}\label{eq:deltadifferentmodels}
		 	\frac{B_n}{n^{\delta}} \inlaw \mathfrak{B}, 
		 \end{align}
		 
		 \noindent where the distribution of $\mathfrak{B}$ is e.g.~analysed by Flajolet et~al.~\cite{FlajoletDumasPuyhaubert06} via an analytic combinatorics approach, or by Janson~\cite{Janson06} and Chauvin et~al.~\cite{ChauvinMaillerPouyanne15} via a probabilistic approach relying on a continuous-time embedding introduced by Athreya and Karlin~\cite{AthreyaKarlin68}.		 
		 For example, averaging the Young--P\'olya urn with $p=2$, $\ell=1$, and $b_0=w_0=1$ leads to the replacement matrix 
		$
			\begin{pmatrix}
				1 & 1/2 \\
				0 & 3/2 
			\end{pmatrix}.
		$
The corresponding classical urn model leads to a limit distribution with moments given e.g.~by Janson in~\cite[Theorem~1.7]{Janson06}:
	\begin{align*}
		\E ({\mathfrak B}^r ) = \frac{\Gamma(4/3) \, r!}{\Gamma(2r/3+4/3)}.
	\end{align*}
	 Comparing these moments with the moments of our distribution
(Equation~\eqref{m_r} hereafter) proves that these two distributions are distinct. 
However, it is noteworthy that they have similar tails: we discuss this universality in Section~\ref{defTails}.
		 \smallskip 
		 
		\item Another interesting alternative model, called multi-drawing P\'olya urn model, consists in drawing multiple balls at once; see Lasmar et~al.~\cite{LasmarMaillerSelmi18} or 
		Kuba and Sulzbach~\cite{KubaSulzbach17}.
		Grouping $p$ units of time into one drawing leads to a new replacement matrix. 
		For example, for $p=2$ and $\ell=1$ we can approximate a Young--P\'olya urn by an urn where at each unit of time $2$ balls are drawn uniformly at random. 
		If both of them are black we add $2$ black balls and $1$ white ball, if one is black and one is white we add $1$ black and $2$ white ball, and if both of them are white we add $3$ white balls.
		Then, the same convergence as in Equation~\eqref{eq:deltadifferentmodels} holds, yet again with a different limit distribution, 
as can be seen by comparing the means and variances; compare Kuba and Mahmoud~\cite[Theorem~1]{KubaMahmoud15} with our Example~\ref{ex:varYP}. 
	\end{itemize}
		
		\smallskip
For all these alternative models, the corresponding histories are inherently different: 
none of them gives the exact generating function of periodic P\'olya urns 
nor gives the closed form of the underlying distribution. 
This also motivates the exact and asymptotic analysis of our periodic model, 
which therefore enriches the urn world with new special functions. 

\begin{figure}[!ht]
	\begin{minipage}{0.49\textwidth}
		\centering
		\includegraphics[width=.94\textwidth]{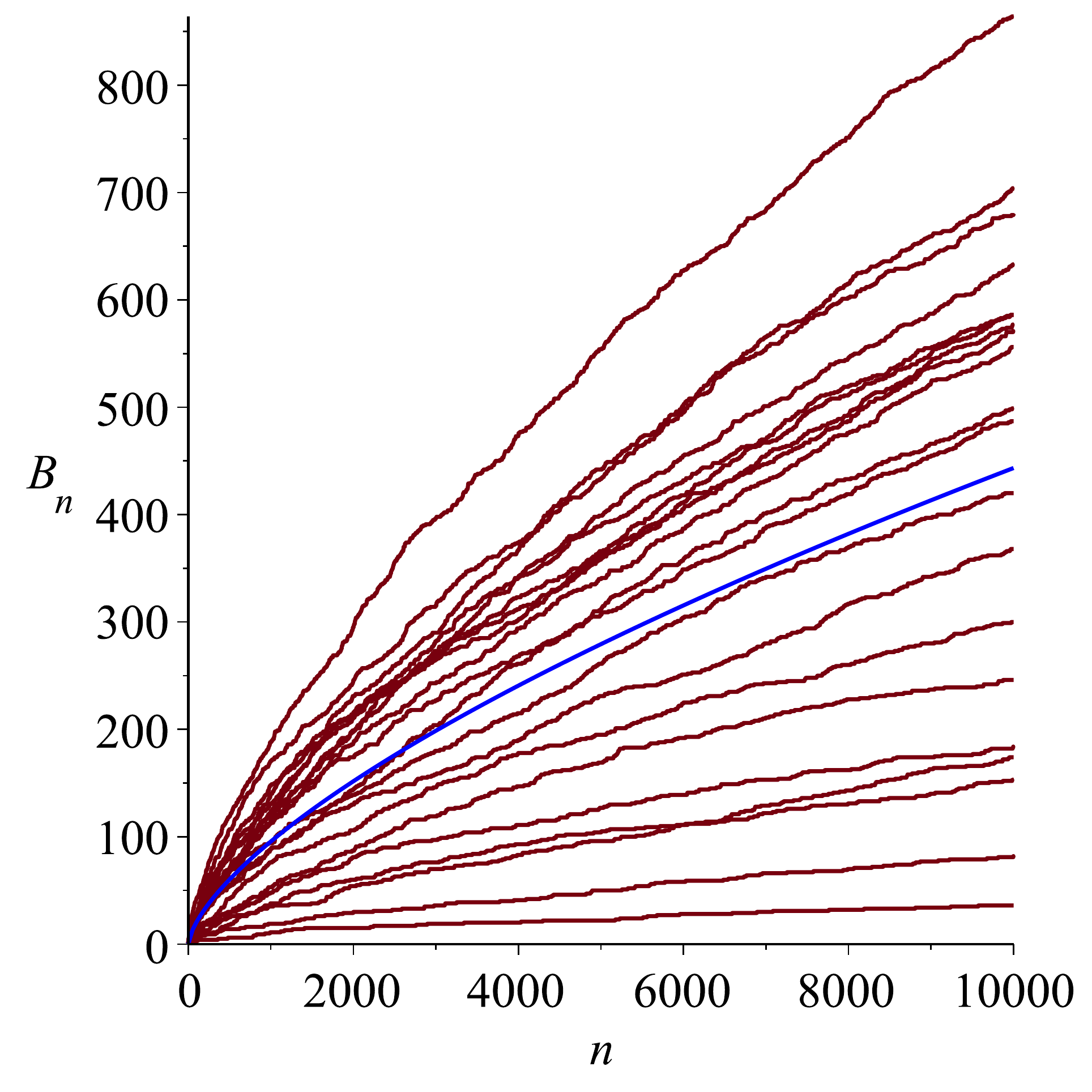}
	\end{minipage}
	\begin{minipage}{0.49\textwidth}
		\centering
		\includegraphics[width=.94\textwidth]{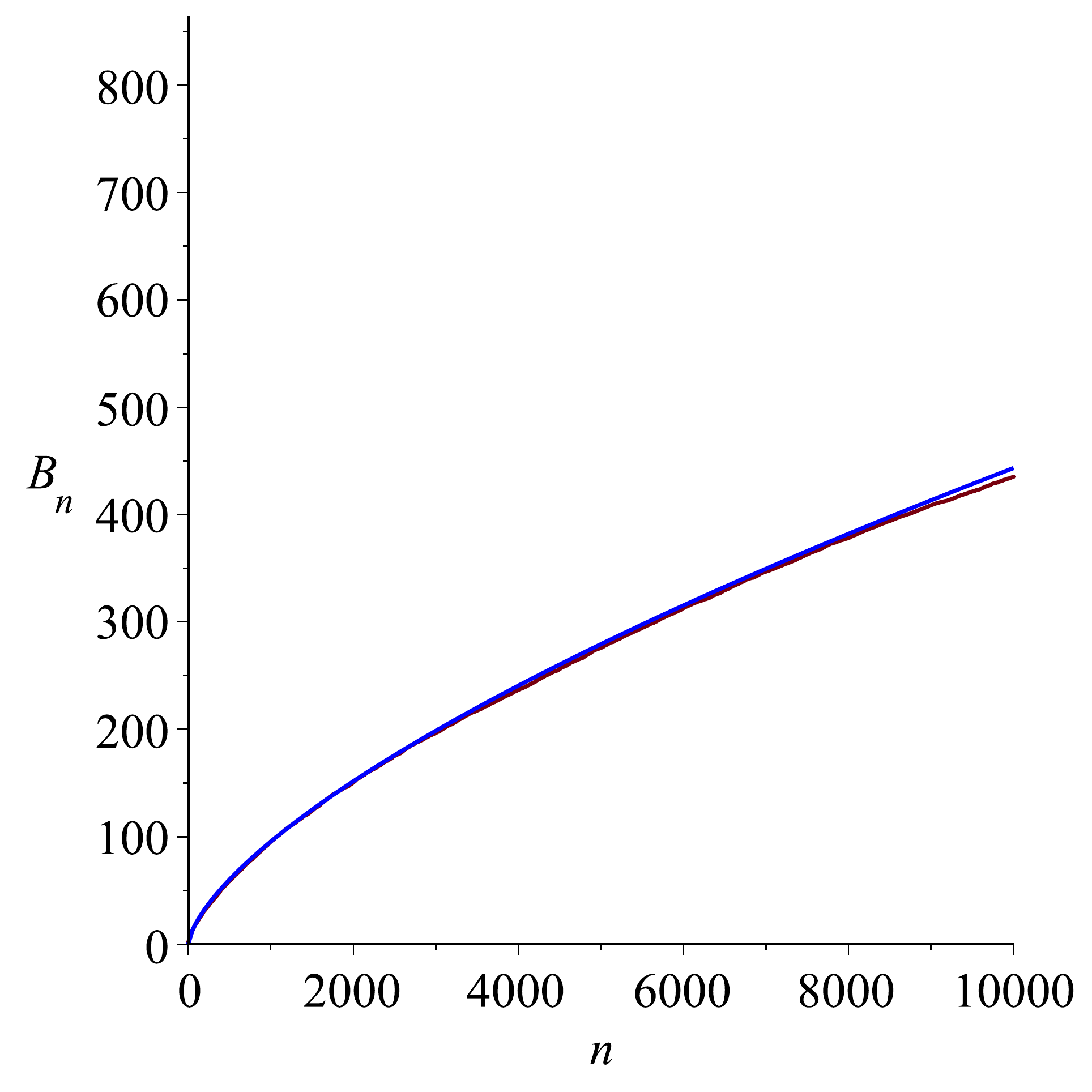}
	\end{minipage}
\internallinenumbers
	\caption{Left: $20$ simulations (drawn in red) of the evolution of $B_n$, the number of black balls in the Young--P\'olya urn with period $p=2$ and parameter~$\ell=1$
(first $10000$ steps, with initially $b_0=1$ black and $w_0=1$ white balls), and the mean $\E(B_n)$ (drawn in blue).
Right: the average (in red) of the $20$ simulations, fitting neatly (almost indistinguishable!) the limit curve $\E(B_n)=\Theta(n^{2/3})$ (in blue). 
} 
	\label{fig:randomgen}
\end{figure}

\pagebreak

Figure~\ref{fig:randomgen} shows that the distribution of $B_n$ is spread; this is consistent with our result 
that the standard deviation and the mean $\E(B_n)$ (drawn in blue) have the same order of magnitude\footnote{ 
The classical urn models with replacement matrices being either $M_1$ or $M_2$
	also have such a spread; see~\cite[Figure~1]{FlajoletDumasPuyhaubert06}.}.
The fluctuations around this mean 
are given by the generalized gamma product limit law from Equation~\eqref{ProdGenGamma},
as proven in Section~\ref{sec:Moments}. Let us first mention some articles where this distribution has already appeared before:
\begin{itemize}
	\item in Janson~\cite{Janson10}, as an instance of distributions with moments of gamma type, 
	like the distributions occurring for the area of the supremum process of the Brownian motion;
	
	\smallskip 
 	\item in Pek\"oz, R\"ollin, and Ross~\cite{PekoezRoellinRoss16}, as distributions of processes on walks, trees, urns, and preferential attachments in graphs, where these authors also consider what they call a P\'olya urn with immigration, 
which is a special case of a periodic P\'olya urn (other models or random graphs 
have these distributions as limit laws~\cite{Senizergues19,BubeckMosselRacz15});
	
	\smallskip 
	\item in Khodabin and Ahmadabadi~\cite{KhodabinAhmadabadi10} following a tradition to 
 generalize special functions by adding parameters in order to capture several probability distributions, such as e.g.~the normal, Rayleigh, and half-normal distribution, as well as the MeijerG function (see also the addendum of~\cite{Janson10}, mentioning a dozen
other generalizations of special functions).
\end{itemize}

\subsection{Plan of the article}
Our \emph{main results} are the explicit enumeration results
and links with hypergeometric functions
(Theorems~\ref{theo:DiffEq} 
and~\ref{theo_hypergeometric}), 
and the limit law involving a product of generalized gamma distributions
(Theorem~\ref{ProdGenGammaGeneral}, 
or the simplified version of it given for readability in Theorem~\ref{theo:PGG} above). 
It is a nice cherry on the cake that
this limit law also describes the fluctuations of the south-east\footnote{In this article, 
we use the French convention to draw the Young tableaux; see Section~\ref{sec:Tableaux} and~\cite{Macdonald15}.} corner of a random triangular Young tableau (as proven in Theorem~\ref{TheoremCornerGen}).
We believe that the methods used, 
i.e.~the generating functions for urns (developed in Section~\ref{Sec2}),
the way to access the moments (developed in Section~\ref{sec:Moments}),
and the density method for Young tableaux (developed in Section~\ref{sec:Tableaux})
are an original combination of tools, which should find many other applications in the future.
Finally, Section~\ref{Sec5} gives a relation between the south-east and the north-west corners of triangular Young tableaux
(Proposition~\ref{prop:factorizationOfGamma}) and a link with factorizations of gamma distributions.
Additionally, we discuss some universality properties of random surfaces,
and we show to what extent the tails of our distributions are related to the tails of 
Mittag-Leffler distributions (Theorem~\ref{prop:tailsML}), 
and when they are subgaussian (Proposition~\ref{prop:tailsSubgaussian}).

\smallskip

In the next section, we translate the evolution of the urn into the language of generating functions by encoding the dynamics of this process into partial differential equations.

\section{A functional equation for periodic P\'olya urns}\label{Sec2}
\subsection{Urn histories and differential operators}\label{Sec2.1}
Let $h_{n,b,w}$ be the number of histories of a periodic P\'olya urn after~$n$ steps with~$b$ black balls and~$w$ white balls, with an initial state of $b_0$ black 
and $w_0$ white balls. We define the polynomials
\begin{align*}
	h_{n}(x,y):= \sum_{b,w \geq 0} h_{n,b,w} x^b y^{w}.
\end{align*}
Note that these are indeed polynomials as there is just a finite number of histories after $n$ steps. 
Due to the balanced urn model these polynomials are homogeneous.
We collect all these histories in the trivariate exponential generating function
\begin{equation} \label{def_H}
	H(x,y,z):= \sum_{n \geq 0} h_n(x,y) \frac{z^n}{n!}.
\end{equation}

\begin{example}
For the Young--P\'olya urn with $p=2$, $\ell=1$, and $b_0=w_0=1$, we get for the first three terms of $H(x,y,z)$ the expansion (compare Figure~\ref{fig:youngurnevolution})\vspace{-3mm}
\begin{minipage}{.99\textwidth}
\begin{align*}
	H(x,y,z) = xy + \left( xy^2 + x^2y \right) z
	 + \left( 2xy^4 + 2x^2 y^3 + 2x^3 y^2 \right) \frac{z^2}{2} 
								+ \dots 
\end{align*}
\end{minipage}\vspace{-4mm}\quad\flushright\qquad                 
\end{example}

In this section, our goal is to derive a partial differential equation describing the evolution of the periodic P\'olya urn model. 
 
The periodic nature of the problem motivates to split the number of histories into $p$ residue classes. 
Let $H_{0}(x,y,z),H_{1}(x,y,z),\ldots,H_{p-1}(x,y,z)$ be the generating functions of histories after $0,1,\ldots,p-1$ draws modulo $p$, respectively. In particular, we have 
\begin{align*}
	H_i(x,y,z) &:= \sum_{n \geq 0} h_{pn+i}(x,y) \frac{z^{pn+i}}{(pn+i)!}, 
\end{align*}
for $i=0,1,\ldots,p-1$ such that
\begin{align*}
	H(x,y,z) &= H_{0}(x,y,z) + H_{1}(x,y,z) + \dots + H_{p-1}(x,y,z).
\end{align*}

Next, we associate with 
 the two distinct replacement matrices
\begin{equation*}\begin{pmatrix}
	1 & 0 \\
	0 & 1
\end{pmatrix}
\text{\qquad and \qquad }
\begin{pmatrix}
	1 & \ell \\
	0 & 1+\ell
\end{pmatrix}\end{equation*}
from Definition~\ref{def:youngpolyaurn} the differential operators $\Dc_1$ and $\Dc_2$, respectively. We get
\begin{equation*}
	\Dc_1:= x^2 \partial_x + y^2 \partial_y \text{\qquad and \qquad}
	\Dc_2:= y^\ell \Dc_1,
\end{equation*}
where $\partial_x$ and $\partial_y$ are defined as the partial derivatives $\frac{\partial}{\partial x}$ and $\frac{\partial}{\partial y}$, respectively. 
This models the evolution of the urn. For example, in the term $x^2\partial_x$, the derivative $\partial_x$ represents drawing a black ball and the multiplication by $x^2$ returning this black ball and an additional black ball into the urn. 
The other terms have analogous interpretations. 

With these operators we are able to link the consecutive drawings with the following system
\begin{align}
	\label{eq:funceqgen1}
	\begin{cases}
	\partial_z H_{i+1}(x,y,z) = \Dc_{1} H_i
	(x,y,z), & \text{ for } i = 0,1,\ldots,p-2,\\
	\partial_z H_0(x,y,z) = \Dc_2 H_{p-1}(x,y,z).
	\end{cases}
\end{align}
Note that the derivative $\partial_z$ models the evolution in time.
We see two types of transitions: in the first $p-1$ rounds the urn behaves like a normal P\'olya urn, but in the $p$-th round we additionally add $\ell$ white balls. The first transition type is modelled by the $\Dc_1$ operator and the second type by the $\Dc_2$ operator.
This system of partial differential equations naturally corresponds 
to recurrences on the level of coefficients $h_{n,b,w}$, and \textit{vice versa}.
This philosophy is well explained in the \textit{symbolic method} part of~\cite{FlajoletSedgewick09}
(see also~\cite{FlajoletGabarroPekari05,FlajoletDumasPuyhaubert06,MorcretteMahmoud12,HwangKubaPanholzer07} for examples of applications to urns).

As a next step, we want to eliminate the $y$ variable in these equations. This is possible as the number of balls in each round and the number of black and white balls are connected due to the fact that we are dealing with balanced urns. 
As observed previously, one has
\begin{align}
	\label{eq:ballnumber}
\text{number of balls after $n$ steps} =	s_0 + n + \ell \left\lfloor \frac{n}{p} \right\rfloor,
\end{align}
with $s_0 := b_0 + w_0$ being the number of initial balls. 
Therefore, for any $x^b y^{w} z^n$ appearing in $H(x,y,z)$, 
we have
\begin{align*}
		b + w &= s_0 + n + \ell \frac{n-i}{p} \qquad \text{ if $n \equiv i \mod p$,}
\end{align*}
which directly translates into the following system of equations (for $i=0,\dots,p-1$)
\begin{align}
	\label{eq:funceqgen2}	\quad
	x\partial_x H_i(x,y,z) + y \partial_y H_i(x,y,z) &= \left(1 + \frac{\ell}{p}\right) z\partial_z H_i(x,y,z) + \left(s_0-\frac{i \ell}{p}\right) H_i(x,y,z).
\end{align}
These equations are contractions in the metric space of formal power series in $z$ (see e.g.~\cite{BanderierDrmota15} or~\cite[Section A.5]{FlajoletSedgewick09}), 
so, given the initial conditions $[z^0] H_i(x,y,z)$, 
the Banach fixed-point theorem entails that this system has a unique solution:
our set of generating functions. 
Now, because of the deterministic link between the number of black balls and the number of white balls,
it is natural 
to introduce the shorthands $H(x,z) := H(x,1,z)$ and $H_i(x,z) := H_i(x,1,z)$. 
What is the nature of these functions? This is what we tackle now.

\subsection{D-finiteness of history generating functions}

Let us first give a formal definition of the fundamental concept of D-finiteness.
\begin{definition}[D-finiteness]
A power series $F(z) = \sum_{n \geq 0} f_n z^n$ with coefficients in some ring~$\mathbb{A}$ 
is called \emph{D-finite} if it satisfies a linear differential equation $L.F(z) =0$, where $L\neq 0$ is a differential operator, $L\in \mathbb{A}[z,\partial_z]$. Equivalently, 
the sequence $(f_n)_{n\in\N}$ is called \emph{P-recursive}: it
satisfies a linear recurrence with polynomial coefficients in $n$. 
Such functions and sequences are also sometimes called \emph{holonomic}.
\end{definition}

D-finite functions are ubiquitous in combinatorics, computer science, probability theory, number theory, 
physics, etc.; see e.g.~\cite{AbramowitzStegun83} or \cite[Appendix~B.4]{FlajoletSedgewick09}.
They possess closure properties galore; this provides an ideal framework for handling (via computer algebra) 
sums and integrals involving such functions~\cite{PetkovsekWilfZeilberger96,BostanChyzakGiustiLebretonLecerfSalvySchost17}. 
The same idea applies to a full family of linear operators 
(differentiations, recurrences, finite differences, $q$-shifts)
and is unified by what is called holonomy theory.
This theory leads to a fascinating algorithmic universe to deal 
with orthogonal polynomials, Laplace and Mellin transforms, and most of the integrals of special functions:
it offers powerful tools to prove identities, 
asymptotic expansions, numerical values, structural properties; see~\cite{PrudnikovBrychkovMarichev92,OlverLozierBoisvertClark10,KauersPaule11}. 

We have seen in Section~\ref{Sec2.1} that the dynamics of urns is intrinsically related to \emph{partial} differential equations (mixing $\partial_x$, $\partial_y$, and $\partial_z$).
It is therefore a nice surprise that it is also possible to describe their evolution in many cases with 
\emph{ordinary} differential equations (i.e.~involving only $\partial_z$).

\pagebreak

\begin{theorem}[Differential equations for histories]\label{theo:DiffEq}
	The generating functions describing a Young--P\'olya urn of period $p$ and parameter~$\ell$ with initially $s_0 = b_0 + w_0$ balls, where $b_0$ are black and $w_0$ are white, 
 satisfy the following system of $p$ partial differential equations:
\begin{align}
	\label{eq:fundamentalpde}
	\partial_z H_{i+1}(x,z) = x(x-1) \partial_x H_i(x,z) + \left(1 + \frac{\ell}{p}\right) z\partial_z H_i(x,z) + \left(s_0-\frac{i\ell}{p}\right) H_{i}(x,z),
\end{align}
for $i=0,\ldots,p-1$ with $H_{p}(x,z) := H_{0}(x,z)$.
Moreover, if any of the corresponding generating functions (ordinary, exponential, ordinary probability, or exponential probability)
 is D-finite in $z$, then all of them are D-finite in $z$.
\end{theorem}
\begin{proof}
First, let us prove the system involving $\partial_z$ and $\partial_x$ only. 
Combining \eqref{eq:funceqgen1} and \eqref{eq:funceqgen2}, we eliminate~$\partial_y$. Then it is legitimate to insert $y=1$ as there appears no differentiation with respect to $y$ anymore. This gives~\eqref{eq:fundamentalpde}.

Now, assume the ordinary generating function is D-finite. 
Multiplying a holonomic sequence by $n!$ (or by $1/n!$, or more generally by any holonomic sequence)
gives a new sequence, which is also holonomic.
In other words, the Hadamard product of two holonomic sequences is still holonomic~\cite[Chapter 6.4]{Stanley99}.
This proves that the ordinary and exponential versions of our generating functions $H$ and~$H_i$ are D-finite in $z$.

Finally, for the probability generating function defined as 
\begin{equation}
\sum_{n,b,w} \Prob\left(B_n=b \text{ and } W_n=w\right) x^b y^w z^n 
= \sum_{n} \frac{h_n(x,y)}{h_n(1,1)} z^n,
\end{equation}
it is in general not the case that it is holonomic if the initial ordinary generating function is holonomic!
But in our case a miracle occurs:
in each residue class of $n$ mod $p$, the sequence $(h_{pm+i}(1,1))_{m\in \N}$ is hypergeometric 
(as shown in Theorem~\ref{theo_hypergeometric}), therefore the $p$
subsequences $(1/h_{pm+i}(1,1))_{m \in \N}$ are also hypergeometric, and thus the above probability generating function 
(which is the sum of $p$ holonomic functions, each one being the Hadamard product of two holonomic functions)
is holonomic.
\end{proof}

Experimentally, in most cases a few terms suffice to guess a holonomic sequence in $z$. 
We believe that this sequence is always holonomic, yet we were not able to prove it in full generality. 
We plan to comment more on this and other related phenomena in the forthcoming article~\cite{BanderierWallner19}.

\pagebreak 

\begingroup 
\def\arraystretch{1.7}
\begin{table}[ht]
	\begin{center}
	\begin{tabular}{|c|c||c|c|c|}
		\hline Type & Generating function & Order in $\partial_z$ & Degree in $z$ & Degree in $x$ \\
	\hline 
			EGF & $\displaystyle{\sum_{n,b,w} h_{n,b,w} x^b y^w \frac{z^n}{n!}}$ &
				$5$ & $13$ & $16$ \\ 
			OGF & $\displaystyle{\sum_{n,b,w} h_{n,b,w} x^b y^w z^n}$ &
				$7$ & $23$ & $20$ \\ 
			EPGF & $\displaystyle{\sum_{n,b,w} \Prob\left(B_n=b \text{ and } W_n=w\right) x^b y^w \frac{z^n}{n!}}$ &
				$8$ & $4$ & $15$ \\ 
			OPGF & $\displaystyle{\sum_{n,b,w} \Prob\left(B_n=b \text{ and } W_n=w\right) x^b y^w z^n}$ &
				$3$ & $13$ & $14$ \\ 
		\hline 
 \end{tabular}
\end{center}\internallinenumbers
\caption{Size of the D-finite equations for the four types of generating functions of histories
(for the urn model of Example~\ref{ex:YoungPolyaODE}). 
We use the abbreviations EGF (exponential generating function), OGF (ordinary generating function), EPGF (exponential probability generating function), OPGF (ordinary probability generating function).
We omit the degree of the variable~$y$, as, for balanced urns, it is trivially related to the degree in $x$.}
\label{tab:dfinitegftypes}
\end{table}
\endgroup

\begin{example}\label{ex:YoungPolyaODE}
	In the case of the Young--P\'olya urn with $p=2$, $\ell=1$, and $b_0=w_0=1$, the differential equations for histories~\eqref{eq:fundamentalpde} are
	\begin{align*}
		\begin{cases}
		\begin{aligned}
		\partial_z H_0(x,z) &= x(x-1) \partial_x H_1(x,z) + \frac{3}{2} z \partial_z H_1(x,z) + \frac{3}{2} H_1(x,z),\\[2mm]
		\partial_z H_1(x,z) &= x(x-1) \partial_x H_0(x,z) + \frac{3}{2} z \partial_z H_0(x,z) + 2 H_0(x,z).
		\end{aligned}
		\end{cases}
	\end{align*}

In addition to this system of {\bf partial} differential equations, 
there exist also two {\bf ordinary} linear differential equations in~$z$
 for $H_0$ and $H_1$, and therefore for their sum $H:=H_0+H_1$, the generating function of all histories. 

In Table~\ref{tab:dfinitegftypes} we compare the size of the D-finite equations\footnote{When we say {\bf the} equation, we mean the linear differential equation of minimal order in $\partial_z$, and then minimal degrees in $z$ and $x$, up to a constant factor for its leading term.}
 for the different generating functions. For example, for the ordinary probability generating function one has the equation $L.F(x,z) =0$, 
where $L$ is the following differential operator of order 3 in $\partial_z$:\\[-2mm]
\def\Dz{\partial_z}
{
\begin{minipage}{\textwidth}
\begin{align*}
	L&= 9\,z \textcolor{black}{ \left( z-1 \right)} \left( z+1 \right) \left( 15\,{x}^{13}{z}^{10}+\dots+3 \right) \textcolor{black}{ \Dz^{3}}+ 3\, \left( 375\,{x}^{13}{z}^{12}+\dots-21 \right) \textcolor{black}{ \Dz^{2}} \\
	&\quad + 2\, \left( 1020\,{x}^{13}{z}^{11}+\dots+42 \right) \textcolor{black}{ \Dz}+600\,{x}^{13}{z}^{10}+\dots+1.
\end{align*}
\end{minipage}
}

\vspace{2mm}

The singularity at $z=1$ of the leading coefficient reflects the fact that $F$ is a probability generating function (and thus has radius of convergence equal to~$1$).
It is noteworthy that some roots of the indicial polynomial of $L$ at $z=1$ differ by an integer, 
this phenomenon is sometimes called resonance, and often occurs in the world of hypergeometric functions; we will come back to these facts and what they imply for the asymptotics (see also~\cite[Chapter~IX.~7.4]{FlajoletSedgewick09}).
\end{example}

Note that the fact to be D-finite has an unexpected consequence:
it allows a surprisingly fast computation of~$h_n$ in time $O(\sqrt n \log^2 n)$ 
(see~\cite[Chapter 15]{BostanChyzakGiustiLebretonLecerfSalvySchost17} for a refined complexity analysis
of the corresponding algorithm).
Such efficient computations are e.g.~implemented in the Maple package \texttt{gfun} (see~\cite{SalvyZimmermann94}).
This package, together with some packages for differential elimination (see~\cite{GerdLangeHegermannRobertz2019,BoulierLemairePoteauxMorenomaza2019}),
allows us to compute the different D-finite equations from Table~\ref{tab:dfinitegftypes}, 
via the union of our Theorem~\ref{theo_hypergeometric} on the hypergeometric closed forms 
and the closure properties mentioned above. 

Another important consequence of the D-finiteness is that the type of the singularities that the function can have 
is constrained. In particular, one important subclass of D-finite functions can be automatically analysed:

\begin{remark}\label{remark:FlajoletLafforgue}
Flajolet and Lafforgue have proven that under some ``generic'' conditions, 
such D-finite equations lead to a Gaussian limit law (see~\cite[Theorem~7]{FlajoletLafforgue94} and \cite[Chapter IX. 7.4]{FlajoletSedgewick09}).
It is interesting that these generic conditions are not fulfilled in our case:
we have a cancellation of the leading coefficient of $L$ at $(x,z)=(1,1)$, a confluence for the indicial polynomial, and the resonance phenomenon mentioned above!
The natural model of periodic P\'olya urns thus leads to an original analytic situation, which offers a new (non-Gaussian) limit law.
\end{remark}
We thus need another strategy to determine the limit law.
In the next section, we use the system of equations~\eqref{eq:fundamentalpde} to iteratively derive the moments of the distribution of black balls after~$n$~steps.

\section{Moments of periodic P\'olya urns} \label{sec:Moments}

In this section, we give the proof of Theorem~\ref{theo:PGG} and a generalization of it. As it will use the method of moments,
let us introduce $m_r(n)$, the $r$-th factorial moment of the distribution of black balls after $n$ steps, i.e.
\begin{align*}
	m_r(n):= \E\left( B_n(B_n-1)\cdots(B_n-r+1) \right).
\end{align*}
Expressing them in terms of the generating function $H(x,z)$, it holds that
\begin{align*}
	m_r(n) = \frac{[z^{n}] \left. \frac{\partial^r}{\partial x^r} H(x,z) \right|_{x=1}}{[z^{n}]H(1,z)},
\end{align*}
where $[z^n]\sum_{n}f_n z^n := f_n$ is the coefficient extraction operator. 

We will compute the sequences of the numerator and denominator separately.
We start with the denominators, the total number of histories after $n$ steps.

\subsection{Number of histories: a hypergeometric closed form}

We prove that $H(1,z)$ satisfies a miraculous property which does not hold for 
$H(x,z)$: it is a sum of generalized hypergeometric functions 
(see e.g.~\cite{AndrewsAskeyRoy99} for an introduction to this important class of special functions).

\pagebreak

\begin{theorem}[Hypergeometric closed forms]\label{theo_hypergeometric}
Let $h_n := n![z^n] H(1,z)$ be the number of histories after $n$ steps in a Young--P\'olya urn 
of period $p$ and parameter~$\ell$ with initially $s_0 = b_0 + w_0$ balls, where $b_0$ are black and $w_0$ are white.
Then, for each~$i$, $(h_{pm+i})_{m\in \N}$ is a hypergeometric sequence,
satisfying the recurrence
\begin{align}\label{hpmi_rec}
	h_{p(m+1)+i} 
	&= \prod_{j=0}^{i-1} \left((p+\ell)(m+1) + s_0+j \right) \prod_{j=i}^{p-1} \left((p+\ell)m + s_0+j \right) h_{pm+i}.
\end{align}
Equivalent closed forms are given in Equations~\eqref{hpmi1} and~\eqref{hpmi2}.
\end{theorem}
\begin{proof}Substituting $x=1$ into \eqref{eq:fundamentalpde} and extracting the coefficient of $z^{n}$ for $i=0,\ldots,p-1$ gives the recurrence
\begin{align}
	h_{n+1} &= \left((1+\frac{\ell}{p}) n + b_n \right) h_{n}, \quad \text{ with } \label{eq:hneq}\\
	b_n : &= s_0 - \frac{\ell}{p} (n \mod p), \label{eq:hneqab}
\end{align}
where $n \mod p$ gives values in $\{0,1,\ldots,p-1\}$. Iterating this recurrence relation $p$ times gives~\eqref{hpmi_rec}.
This leads to the following equivalent closed forms
\begin{align} 
	h_{pm+i} &= \frac{ (p+\ell)^{pm+i}}{\prod_{j=0}^{p-1}\Gamma\left(\frac{s_0+j}{p+\ell}\right)} \prod_{j=0}^{i-1} \Gamma\left(m+1 + \frac{s_0+j}{p+\ell}\right) \prod_{j=i}^{p-1} \Gamma\left(m + \frac{s_0+j}{p+\ell}\right), \label{hpmi1}\\
h_{pm+i} &= (p+\ell)^{pm} \frac{\Gamma(s_0+(p+\ell) m+i) }{\Gamma(s_0+(p+\ell) m)}
\prod_{j=0}^{p-1} \frac{\Gamma\left(m+\frac{s_0+j}{p+\ell}\right)}{\Gamma\left(\frac{s_0+j}{p+\ell}\right)}. \label{hpmi2}
\end{align}
Accordingly, the function $H(1,z)$ is the sum of $p$ generalized hypergeometric functions ${}_pF_0$.
\end{proof}

\begin{example}
	In the case of the Young--P\'olya urn with $p=2$, $\ell=1$, and $b_0=w_0=1$,
	one has the hypergeometric closed forms for $h_n:=n! [z^n] H(1,z)$:
	\begin{align}\label{hn}
		h_n &=
		\begin{cases}
			 3^n \frac{\Gamma\left(\frac{n}{2} + 1\right) \Gamma\left(\frac{n}{2} + \frac{2}{3}\right)}{\Gamma(2/3)} & \text{ if $n$ is even,} \\
			3^n \frac{\Gamma\left(\frac{n}{2} + \frac{1}{2}\right) \Gamma\left(\frac{n}{2} + \frac{7}{6}\right)}{\Gamma(2/3)} & \text{ if $n$ is odd.}
		\end{cases}
	\end{align}
	Alternatively, this sequence satisfies $h(n+2)=\frac{3}{2} h(n+1)+\frac{1}{4}(9\,n^2+21\,n+12)h(n)$.
	This sequence was not in the On-Line Encyclopedia of Integer Sequences, accessible at \href{https://oeis.org}{https://oeis.org}.
	We added it there; it is now \OEIS{A293653}, and it starts like this:
	$1, 2, 6, 30, 180, 1440, 12960, 142560, 1710720, \ldots$
The exponential generating function can be written as the sum of two hypergeometric functions:\\[2mm]
\begin{minipage}{.99\textwidth}
\begin{equation*}H(1,z)={}_2F_1\left(\left[\frac{2}{3}, 1\right], \left[\frac{1}{2}\right], \left(\frac{3z}{2}\right)^2\right)+ 
2z \,\, {}_2F_1\left(\left[\frac{5}{3}, 1\right], \left[\frac{3}{2}\right], \left(\frac{3z}{2}\right)^2\right).\end{equation*}
\end{minipage}
\vspace{-8mm}\quad\flushright\qquad
\end{example}

\subsection{Mean and critical exponent}

Let us proceed with the computation of moments. For this purpose, define
\begin{align}\label{defhnr}
	h_{n}^{(r)} := n! [z^n] \left.\frac{\partial^r}{\partial x^{r}} H(x,z) \right|_{x=1},
\end{align}
as the coefficient of $\frac{(x-1)^r z^n}{r! \, n!}$ of $H(x,z)$. Then the $r$-th moment is obviously computed as
$
	m_{r}(n) = \frac{h_n^{(r)}}{h_n}.
$
The key idea why to use these quantities comes from the differential equations for histories~\eqref{eq:fundamentalpde}. The derivative of $H_i(x,z)$ with respect to $x$ has a factor $(x-1)$, which makes it possible to compute $h_{n}^{(r)}$ iteratively by taking the $r$-th derivative with respect to $x$ and substituting $x=1$. Let us define the auxiliary functions 
\begin{align*}
	H_i^{(r)}(z) := \left. \frac{\partial^r}{\partial x^{r}} H_i(x,z) \right|_{x=1}.
\end{align*}
We get for $i=0,\ldots,p-1$ (with $b_i$ as defined in~\eqref{eq:hneqab}):
\begin{align*}
	\partial_z H_{i+1}^{(r)}(z) &= (1+\frac{\ell}{p}) z \partial_z H_{i}^{(r)}(z) + (b_i + r) H_{i}^{(r)}(z) + (r-1) r H_{i}^{(r-1)}(z).
\end{align*}	
 From this equation we extract the $n$-th coefficient with respect to $z$ and multiply by $n!$ to get
\begin{align}
	h_{n+1}^{(r)} &= \left( (1+\frac{\ell}{p}) n + b_n + r \right) h_{n}^{(r)} + (r-1) r h_{n}^{(r-1)}. \label{eq:hneqr}
\end{align}
We reveal a perturbed version of~\eqref{eq:hneq}. In particular, this is a non-homogeneous linear recurrence relation. Yet, the inhomogeneity only emerges for $r\geq 2$. Thus, the mean is derived directly with the same approach as $h_{n}$ previously. Note that for $r=1$ Equation~\eqref{eq:hneqr} is exactly of the same type as~\eqref{eq:hneq} after replacing $s_0$ by $s_0+r$ and $h_0$ by $b_0$. We get without any further work
\begin{align*}
	h_{pm+i}^{(1)} &= C_{1} \, (p+\ell)^{pm+i} \prod_{j=0}^{i-1} \Gamma\left(m+1 + \frac{s_0+1+j}{p+\ell}\right) \prod_{j=i}^{p-1} \Gamma\left(m + \frac{s_0+1+j}{p+\ell}\right),\\
	C_{1} &= b_0 \prod_{j=0}^{p-1}\Gamma\left(\frac{s_0+1+j}{p+\ell}\right)^{-1}.
	\end{align*}
Combining the last two results, we get a (surprisingly) simple expression
\begin{align*}
	\E B_{pm+i} &= \frac{h_{pm+i}^{(1)}}{h_{pm+i}} 
		 = \frac{C_{1}}{C_{0}} \frac{\prod_{j=0}^{i-1} \Gamma\left(m+1 + \frac{s_0+1+j}{p+\ell}\right) \prod_{j=i}^{p-1} \Gamma\left(m + \frac{s_0+1+j}{p+\ell}\right)}{\prod_{j=0}^{i-1} \Gamma\left(m+1 + \frac{s_0+j}{p+\ell}\right) \prod_{j=i}^{p-1} \Gamma\left(m + \frac{s_0+j}{p+\ell}\right)}\\
		&= b_0 \frac{\Gamma\left(\frac{s_0}{p+\ell}\right)}{\Gamma\left(\frac{s_0+p}{p+\ell}\right)} \left(m + \frac{s_0+i}{p+\ell}\right)\frac{\Gamma\left(m + \frac{s_0+p}{p+\ell}\right)}{\Gamma\left(m+1 + \frac{s_0}{p+\ell}\right)}.
\end{align*}
In particular, it is straightforward to compute an asymptotic expansion for the mean by Stirling's approximation. 
For $i=0,1,\ldots,p-1$, we get
\begin{align*}
	\E B_{pm+i} 
		&= b_0\frac{\Gamma\left(\frac{s_0}{p+\ell}\right)}{\Gamma\left(\frac{s_0+p}{p+\ell}\right)} 
		 m^{\frac{p}{p+\ell}} \left(1 + \LandauO\left(\frac{1}{m}\right)\right).
\end{align*}

This leads to the following proposition.

\begin{proposition}[Formula for the mean of Young--P\'olya urns]
	\label{prop:mean}
	The expected number of black balls in a Young--P\'olya urn of period~$p$ and parameter~$\ell$ with initially $s_0 = b_0 + w_0$ balls, where $b_0$ are black and $w_0$ are white, satisfies for large~$n$
	\begin{align*}
		\E B_{n} 
			&= b_0\frac{\Gamma\left(\frac{s_0}{p+\ell}\right)}{\Gamma\left(\frac{s_0+p}{p+\ell}\right)} 
				 \left(\frac{n}{p}\right)^{\frac{p}{p+\ell}} \left(1 + \LandauO\left(\frac{1}{n}\right)\right).
	\end{align*}
\end{proposition}

\begin{remark}[Critical exponent]
As will be more transparent from discussions in the next sections, the exponent $\delta:=\frac{p}{p+\ell}$ 
is here the crucial quantity to keep in mind. It is sometimes called 
``critical exponent'' as such exponents can often be captured by ideas from statistical mechanics, as a signature of a phase transition phenomenon.
\end{remark}

\begin{example}
For the Young--P\'olya urn with $p=2$, $\ell=1$, and $b_0=w_0=1$, the expected number of 
black balls at time $n$ is thus
	\begin{align*}
		\E B_{n} 
			&= \frac{\Gamma(2/3)}{\Gamma(4/3)} 
				 \left(\frac{n}{2}\right)^{\frac{2}{3}} \left(1 + \LandauO\left(\frac{1}{n}\right)\right)
\approx 0.9552\,\, n^{2/3} \left(1 + \LandauO\left(\frac{1}{n}\right)\right).
	\end{align*}
This is coherent with the renormalization used
for the limit law of $B_n$ in Example~\ref{ex:YoungPolyaMainresult}.
\end{example}

\subsection{Higher moments}

When computing higher moments, the first idea is to transform the non-homogeneous recurrence relation~\eqref{eq:hneqr}
into a homogeneous one. To this aim, one rewrites this equation into 
\begin{align}
	\label{eq:recgenmom}\qquad
	y_{n+1} - \left(a n + b_n + r \right) y_{n} &= (r-1) \, r\, h_{n}^{(r-1)} && \text{ and } &
	y_0 &= \partial_x^r H(x,0)|_{x=1}.
\end{align}
Note that we have $y_n = h_n^{(r)}$, the $r$-th moment we want to determine.
 From now on we speak of the \textit{homogeneous equation} to refer to the left-hand side of Equation~\eqref{eq:recgenmom}
set equal to $0$, whereas Equation~\eqref{eq:recgenmom} itself is called the \textit{non-homogeneous equation}.
In order to get $h_n^{(r)}$ we proceed by induction on~$r$: we assume 
that the $(r-1)$-st moment is known (thus, we know the right-hand side of~\eqref{eq:recgenmom}), 
and we want to express the $r$-th moment $h_{n}^{(r)}$ (i.e.~we want to solve the recurrence~\eqref{eq:recgenmom} for $y_n$)
in terms of this previously computed quantity.

As for any linear recurrence, its solution is given by a combination of a solution $h_{n,\text{hom}}^{(r)}$ of the homogeneous equation and of 
a particular solution $h_{n,\text{par}}^{(r)}$ such that
\begin{align}\label{eq:hnansatz}
	h_{n}^{(r)} &= C_{r} \, h_{n,\text{hom}}^{(r)} - h_{n,\text{par}}^{(r)}\,,
\end{align}
with $C_{r} \in \R$ such that the initial condition in~\eqref{eq:recgenmom} is satisfied. 
We will show that asymptotically only the solution $h_{n,\text{hom}}^{(r)}$ 
of the homogeneous equation 
is dominant. 
First of all, this solution is easy to compute, as it is again of the same type as~\eqref{eq:hneq}. We have 
\begin{align}\label{eq:hnhom}\qquad
	h_{pm+i,\text{hom}}^{(r)} &= (p+\ell)^{pm+i} \prod_{j=0}^{i-1} \Gamma\left(m+1 + \frac{s_0+r+j}{p+\ell}\right) \prod_{j=i}^{p-1} \Gamma\left(m + \frac{s_0+r+j}{p+\ell}\right).
\end{align}
The next idea is to find a particular solution of the non-homogeneous recurrence relation~\eqref{eq:recgenmom}. We will show that the equation exhibits a phenomenon similar to resonance and we will show that the particular solution is 
\begin{align}\label{hn2}
	h_{n,\text{par}}^{(r)} &= \sum_{j=1}^{r-1} d_j h_{n}^{(j)}, \text{\quad for constants $d_j \in \R$.}
\end{align}

We will compute the coefficients $d_j$ by induction from $r-1$ to $1$. 
First, we observe that the inhomogeneous part in the $r$-th equation is a multiple of the solution $h_n^{(r-1)}$ of the $(r-1)$-st equation. 
This motivates us to set $y_n = h_{n}^{(r-1)}$ in the homogeneous equation of the $r$-th equation. Using~\eqref{eq:recgenmom} then leads to
\begin{align*}
	h_{n+1}^{(r-1)} - \left(a n + b_n + r \right) h_{n}^{(r-1)} = (r-1)(r-2) h_{n}^{(r-2)} - h_{n}^{(r-1)}.
\end{align*}
Thus, by linearity we choose $h^{(r)}_{n,\text{par}} = z_n - (r-1) r h_n^{(r-1)}$, i.e.~$d_{r-1} = (r-1) r$, as a first candidate for a particular solution where $z_n$ is 
(still) an undetermined sequence. Inserting this into~\eqref{eq:recgenmom}, we get a recurrence relation for $z_n$,
where we reduced the order of the inhomogeneity by one in~$r$ (in comparison with~\eqref{eq:recgenmom}):
\begin{align*}
	z_{n+1} - \left(a n + b_n + r \right) z_{n} = r (r-1)^2 (r-2) h_{n}^{(r-2)}.
\end{align*}
Continuing this approach, we compute all $d_j$'s inductively. 
As the order in $r$ decreases, this approach terminates at $r = 1$. 
One thus identifies the constants $d_j$ of Formula~\eqref{hn2}:
\begin{align*}
	d_{j} = \prod_{i=j+1}^{r} \frac{(i-1) i}{r-i+1} 
					= \binom{r-1}{j-1} \frac{r!}{j!}
					= L(r,j),
\end{align*}
with $L(r,j)$ being the Lah numbers, which express the rising factorials in terms of falling factorials\footnote{\label{fn:fallingfac}The falling factorial $\ffac{x}{r}$ is defined by $\ffac{x}{r}:= x (x-1) \cdots (x-r+1)=\Gamma(x+1)/\Gamma(x-r+1)$,
while the rising factorial $\rfac{x}{r}$ is defined by $\rfac{x}{r} := \Gamma(x+r)/\Gamma(x)=x(x+1)\cdots(x+r-1)$. 
These notations were introduced as an alternative to the Pochhammer symbols 
by Graham, Knuth, and Patashnik in~\cite{GrahamKnuthPatashnik94}.} (see~\cite{Lah54} and~\cite[p.~43]{Riordan02}):
\begin{align}
	\label{eq:Lahrisingfalling}
	 \sum_{j=1}^{r} L(r,j) \ffac{x}{j} = \rfac{x}{r}.
\end{align}
Then, by~\eqref{eq:hnansatz} we get the general solution of the $r$-th moment
\begin{align}\label{eqtiti}
	h_{n}^{(r)} &= C_{r} \, h_{n,\text{hom}}^{(r)} - \sum_{j=1}^{r-1} L(r,j) h_{n}^{(j)}.
\end{align}
For $n=0$, Equation~\eqref{eqtiti} becomes 
\begin{align*}
h_0^{(r)} & = \partial_x^r H(x,0)|_{x=1} = \ffac{b_0}{r} = C_{r} \, h_{0,\text{hom}}^{(r)} - \sum_{j=1}^{r-1} L(r,j) \ffac{b_0}{j},
\end{align*}
which gives together with~\eqref{eq:Lahrisingfalling} that $C_{r} \, h_{0,\text{hom}}^{(r)} = \rfac{b_0}{r}$.

Finally, we are now 
able to compute the asymptotic expansion of the $r$-th (factorial) moment. 
Using Stirling's approximation, 
the quotient of the quantities given by~\eqref{eqtiti} and \eqref{eq:hnhom} gives that
$\frac{h_{n}^{(j)}}{h_{n,\text{hom}}^{(r)}}= \LandauO\left( n^{-\frac{(r-j)p}{p+\ell}} \right),$ 
for $j=1,\ldots,r-1$. 
Hence, for the $r$-th moment given by~\eqref{eqtiti}, we proved that the contribution of $h_{n,\text{hom}}^{(r)}$ 
is the asymptotically dominant one. 
This leads to the main result on the asymptotics of the moments:
\begin{proposition}[Moments of Young--P\'olya urns] 
\label{prop:highmoments}
The $r$-th (factorial) moment of $B_n$ (the number of black balls in the Young--P\'olya urn of period~$p$ and parameter~$\ell$ with initially $s_0 = b_0 + w_0$ balls, where $b_0$ are black and $w_0$ are white) for large~$n$ satisfies\addtocounter{footnote}{-1}\addtocounter{Hfootnote}{-1}\footnotemark
\begin{align*}
	m_r(n) = \gamma_r \, n^{\delta r} \left(1+\LandauO\left(\frac{1}{n}\right)\right),  \text{ \, with \, } 
	\gamma_r = \frac{\rfac{b_0}{r}}{p^{\delta r}} \prod_{j=0}^{p-1} \frac{\Gamma\left(\frac{s_0+j}{p+\ell}\right)}{\Gamma\left(\frac{s_0+r+j}{p+\ell}\right)} 
 \text{ \, and\, $\delta = \frac{p}{p+\ell}$}.
\end{align*}
\end{proposition}
\begin{example}
\label{ex:varYP}
For the Young--P\'olya urn with $p=2$, $\ell=1$, and $b_0=w_0=1$, the variance of the number of 
black balls at time~$n$ is thus
	\begin{align*}
		\Vb B_{n} 
			&= \frac{27}{8} \frac{ \Gamma\left(\frac{2}{3}\right)^2\left( 3\Gamma\left(\frac{4}{3}\right) - \Gamma\left(\frac{2}{3}\right)^2 \right)}{2^{1/3} \pi^2} 
				 n^{4/3} \left(1 + \LandauO\left(\frac{1}{n}\right)\right)
\approx 0.42068\,\, n^{4/3} \left(1 + \LandauO\left(\frac{1}{n}\right)\right). 
	\end{align*}
\quad\flushright\qquad
\end{example}

\textit{Nota bene:} The reasoning following Equation~\eqref{eqtiti} 
shows that these asymptotics are the same for the moments and the factorial moments,
so in the sequel we refer to this result indifferently from both points of view. 
              
\subsection{Limit distribution for periodic P\'olya urns}
\label{sec:carleman}
We use the method of moments to prove Theorem~\ref{theo:PGG} (the generalized gamma product distribution for Young--P\'olya urns). 
The natural factors occurring in the constant $\gamma_r$ of Proposition~\ref{prop:highmoments},
may they be $1/\Gamma(\frac{s+r+j }{p+\ell})$ or $(\rfac{b_0}{r})^{1/p}/\Gamma(\frac{s+r+j }{p+\ell})$,
do not satisfy the determinant/finite difference positivity tests for the Stieltjes/Hamburger/Hausdorff moment problems, therefore no continuous distribution has such moments (see~\cite{Wall48}). 
However, the full product does correspond to moments of a distribution which is easier to identify if 
we start by transforming the constant $\gamma_r$ by the Gauss multiplication formula of the gamma function; this gives
\begin{align*}
	\gamma_r 
			&= \frac{(p+\ell)^{ r}}{p^{\delta r}} 
						\frac{\Gamma\left(b_0+r\right) \Gamma\left(s_0\right)}{\Gamma\left(b_0\right) \Gamma\left(s_0+r\right)}
						\prod_{j=0}^{\ell-1} 
						\frac{\Gamma\left(\frac{s_0+r+p+j}{p+\ell}\right)}{\Gamma\left(\frac{s_0+p+j}{p+\ell}\right)}.
\end{align*}

\noindent Combining this result with 
the r-th (factorial) moment $m_r(n)$ from Proposition~\ref{prop:highmoments}, 
we see that the moments $\E\left( {B_n^*}^r\right)$ of the rescaled random variable 
$B_n^*:= \frac{p^{\delta}}{p+\ell} \frac{B_n}{n^{\delta}}$
converge for $n \to \infty$ to the limit
\begin{align} \label{m_r}
	m_r:= \frac{\Gamma\left(b_0+r\right) \Gamma\left(s_0\right)}{\Gamma\left(b_0\right) \Gamma\left(s_0+r\right)}
						\prod_{j=0}^{\ell-1} 
						\frac{\Gamma\left(\frac{s_0+r+p+j}{p+\ell}\right)}{\Gamma\left(\frac{s_0+p+j}{p+\ell}\right)},
\end{align}

\noindent a simple formula involving the parameters $(p,\ell,b_0,w_0)$ of the model (with $s_0:=b_0+w_0$). 

Next note that the following sum diverges (recall that $0 \leq (1-\delta) < 1$):
\begin{align}\label{Carleman}
	\sum_{r > 0} m_{r}^{-1/(2r)} = \sum_{r > 0} \left(\frac{(p+\ell)e}{r}\right)^{(1-\delta)/2}(1+\Landauo(1)) = +\infty\,.
\end{align}

\noindent Therefore, a result by Carleman 
(see~\cite[pp.~189--220]{Carleman23}) implies that there exists a unique distribution (let us call it $\mathcal D$) with such moments~$m_r$.
Then, by the limit theorem of Fr\'echet and Shohat~\cite[p.~536]{FrechetShohat31}\footnote{As a funny coincidence, Fr\'echet and Shohat mention in~\cite{FrechetShohat31} that 
the generalized gamma distribution with parameter $p\geq 1/2$ is uniquely characterized by its moments.}, $B_n^*$ converges to $\mathcal D$.

Finally, we use the shape of the moments in~\eqref{m_r} in order to express this distribution $\cal D$ in terms of the main functions defined in Section~\ref{sec:urn}.
First, note that if for some independent random variables $X, Y, Z$, one has 
$\E(X^r) = \E(Y^r) \E(Z^r)$ (and if $Y$ and $Z$ are determined by their moments), then $X \= Y Z$.
Therefore, we treat the factors independently.
The first factor corresponds to a beta distribution $\operatorname{Beta}(b_0,w_0)$.
For the other factors it is easy to check that if $X \sim \GG(\alpha,\beta)$ is a generalized gamma distributed random variable (as defined in Definition~\ref{def:GG}),
then it is a distribution determined by its moments, which are given by
$\E(X^r) = \frac{\Gamma\left(\frac{\alpha+r}{\beta}\right)}{\Gamma\left(\frac{\alpha}{\beta}\right)}.$
Therefore, the expression in~\eqref{m_r} characterizes the $\PGG$ distribution. 
This completes the proof of Theorem~\ref{theo:PGG}. \qed

\bigskip
For reasons which would be clear in Section~\ref{sec:Tableaux},
it was natural to focus first on Young--P\'olya urns. 
However, the method presented is this section allows us to handle more general models. 
It would have been quite indigestible to present directly the general proof with heavy notations 
and many variables but now that the reader got the key steps of the method,
she should be delighted to recycle all of this for free in the following much more general result:

\begingroup 
\begin{theorem}[The generalized gamma product distribution for triangular balanced urns]
\label{theo:general}
Let $p\geq 1$ and $\ell_1,\ldots,\ell_p\geq 0$ be non-negative integers. Consider a periodic P\'olya urn of period $p$ with replacement matrices $M_1,\ldots,M_p$ given by
$
M_j:=
	\begin{pmatrix}
			1 & \ell_j \\
			0 & 1+\ell_j
		\end{pmatrix}.
$ 
Then, the renormalized distribution of black balls is asymptotically for $n \to \infty$ given by the following product of distributions: 
\begin{equation}\label{ProdGenGammaGeneral}
\frac{p^\delta}{p+\ell} \frac{B_n }{n^\delta}\inlaw \operatorname{Beta}(b_0,w_0) 
	\prod_{\substack{i=1 \\i\neq\ell_1+\dots+\ell_j+j \text{ with } 1\leq j \leq p-1}}^{p+\ell-1} \GG(b_0+w_0+i, p+\ell).
\end{equation}
with $\ell = \ell_1+\dots+\ell_p$, $\delta=p/(p+\ell)$, and $\operatorname{Beta}(b_0,w_0)=1$ when $w_0=0$.
\end{theorem}

In the sequel, we denote this distribution by $\PGG([\ell_1,\ldots,\ell_p];b_0,w_0)$.

\begin{proof} The proof relies on the same steps as in Sections~\ref{Sec2} and \ref{sec:Moments}
with some minor technical changes, so we only point out the main differences.

The behaviour of the urn is now modeled by the $p$ differential operators 
$\Dc_j = y^{\ell_j} ( x^2 \partial_x + y^{2} \partial_y)$. 
As the matrices are balanced, there is (like in Equation~\eqref{eq:ballnumber})
a direct link between the number of black balls and the total number of balls.
This allows to eliminate the $y$ variable and leads to the following system of partial differential equations
(which generalizes Equation~\eqref{eq:fundamentalpde}):
\begin{align}
	\label{eq:funceqgenGEN}	
	\partial_z H_{i+1}(x,z) = x(x-1) \partial_x H_i(x,z) + \left(1 + \frac{\ell}{p}\right) z\partial_z H_i(x,z) + \left(s_0 - \sum_{j=1}^i \ell_j - \frac{i\ell}{p}\right) H_{i}(x,z),
\end{align}
for $i=0,\ldots,p-1$ with $H_{p}(x,z) := H_{0}(x,z)$.
Here, one again applies the method of moments used in this Section~\ref{sec:Moments}. 
In particular, Equation~\eqref{eq:hneq} remains the same.
 Only the coefficients $b_n$ in Equation~\eqref{eq:hneqab} 
change to $s_0 - \sum_{j=1}^i \ell_j - \frac{\ell}{p}(i \mod p)$.

Hence, we get the following asymptotic result for the moments generalizing Proposition~\ref{prop:highmoments}:
\begin{align}\label{MomentsOfLove}
	m_r(n) = \gamma_r \, n^{\delta r} \left(1+\LandauO\left(\frac{1}{n}\right)\right), & \text{ with } 
	\gamma_r = \frac{\rfac{b_0}{r}}{p^{\delta r}} \prod_{j=0}^{p-1} \frac{\Gamma\left(\frac{s_0}{p+\ell}+\frac{j+\sum_{k=1}^j \ell_k}{p+\ell}\right)}{\Gamma\left(\frac{s_0+r}{p+\ell}+\frac{j+\sum_{k=1}^j \ell_k}{p+\ell}\right)}.
\end{align}
After rewriting $\gamma_r$ via the Gauss multiplication formula,
we recognize the product of distributions (characterized by their moments) which we wanted to prove.
\end{proof}

\pagebreak

Let us illustrate this theorem with what we call 
the staircase periodic P\'olya urn (this model will reappear later in the article). 

\begin{example}[Staircase periodic P\'olya urn]
 \label{othermodel}
For the P\'olya urn of period $3$ with replacement matrices
\begin{equation*}
M_1:=
	\begin{pmatrix}
			1 & 0 \\
			0 & 1
		\end{pmatrix}, \text{\,}
M_2:=
	\begin{pmatrix}
			1 & 1 \\
			0 & 2
		\end{pmatrix}, \text{\, and }
M_3:=	\begin{pmatrix}
			1 & 2 \\
			0 & 3
		\end{pmatrix},
\end{equation*}
the number $B_n$ of black balls has the limit law $\PGG([0,1,2];b_0,w_0)$:\\[-1.7mm]
\begin{minipage}{.99\textwidth}
\begin{align*}\label{eq:slopevariant}
\frac{\sqrt{3}}{6} \frac{B_n }{\sqrt{n}}\inlaw \operatorname{Beta}(b_0,w_0) 
 \prod_{i=2,4,5} \GG(b_0+w_0+i, 6).
\end{align*}
\end{minipage}
\vspace{-9mm}\quad\flushright\qquad 
\end{example}
\endgroup

\smallskip

In the next section, we will see what are the implications of our results for urns on an apparently unrelated topic: Young tableaux.

\section{Urns, trees, and Young tableaux}\label{sec:Tableaux}

\vspace{-1mm}
As predicted by Anatoly Vershik in~\cite{Vershik01}, 
the $21^{\text{st}}$ century should see a lot of challenges and advances on the links between probability theory 
and (algebraic) combinatorics. 
A key r{\^o}le 
is played here by Young tableaux\footnote{A Young tableau of size~$n$ is an array with columns of (weakly) decreasing height,
in which each cell is labelled, 
and where the labels run from $1$ to $n$ and are strictly increasing along rows from left to right and columns from bottom to top; see Figure~\ref{fig:tabtreeurn}. 
We refer to~\cite{Macdonald15} for a thorough discussion on these objects.} because of their ubiquity in representation theory.
Many results on their asymptotic shape have been collected, 
but very few results are known on their asymptotic content when the shape is fixed
(see e.g.~the works by Pittel and Romik, Angel et~al., Marchal~\cite{PittelRomik07,AngelHolroydRomikVirag07,Romik15,Marchal16}, who have studied 
the distribution of the values of the cells in random rectangular or staircase Young tableaux, 
while the case of Young tableaux with a more general shape seems to be very intricate).
It is therefore pleasant that our work on periodic P\'olya urns allows us to get advances on the case of a triangular shape, with any rational slope. \vspace{-1mm}

\begin{definition}\label{triangYoung} For any fixed integers $n,\ell, p\geq 1$,
we define a \textit{triangular Young tableau} of parameters $(\ell,p,n)$ 
as a classical Young tableau with $N:= p\ell n(n+1)/2$ cells, with length $n\ell$, and height $np$ such that the first $\ell$ columns 
have $np$ cells, the next $\ell$ columns have $(n-1)p$ cells, and so on (see Figure~\ref{fig:tabtreeurn}).
\end{definition}

For such a tableau, we now study what is the typical value of its south-east corner (with the French convention of drawing tableaux; see~\cite{Macdonald15}
but, however, take care that on page~$2$ therein, Macdonald advises readers preferring the French convention to ``read this book upside down in a mirror''! Some French authors quickly propagated the joke that Macdonald was welcome to apply his own advice while reading their articles!).

\pagebreak

It could be expected (e.g.~via the Greene--Nijenhuis--Wilf hook walk algorithm for generating Young tableaux; see~\cite{GreeneNijenhuisWilf84}) 
that the entries near the hypotenuse should be $N-o(N)$. 
Can we expect a more precise description of these $o(N)$ fluctuations? Our result on periodic urns enables us to exhibit the right critical exponent, and the limit law in the corner:

\begin{theorem}\label{TheoremCorner}
 Choose a uniform random triangular Young tableau of parameters $(\ell,p,n)$ and of size
 $N = p\ell n(n+1)/2$ and put $\delta=p/(p+\ell)$. 
Let $\Yn$ be the entry of the south-east corner.
Then $(N-\Yn)/n^{1+\delta}$ converges in law to the same limiting
distribution as the number of black balls in the periodic Young--P\'olya urn 
with initial conditions $b_0=p$, $w_0=\ell$ and
with replacement matrices $M_1= \dots = M_{p-1} = 
	\begin{pmatrix}
			1 & 0 \\
			0 & 1
		\end{pmatrix}$ 
		and 
		$M_p=
		\begin{pmatrix}
			1 & \ell \\
			0 & 1+\ell
		\end{pmatrix}$,
i.e.~we have the convergence in law, as $n$ goes to infinity, towards 
 $\PGG$ (the distribution defined by Formula~\eqref{ProdGenGamma}, page~\pageref{ProdGenGamma}):
				\begin{equation*}
				\frac{2}{p \ell} \frac{N-\Yn}{ n^{1+\delta} } \stackrel{\mathcal L}{\longrightarrow} 
				\PGG(p,\ell,p,\ell	).
				\end{equation*}
\end{theorem}

\begin{remark}
The case $p=1$ corresponds to a classical (non-periodic) urn; see Remark~\ref{rem:p1}.
The case $p=2$ and $\ell=1$ corresponds to our running example of a Young--P\'olya urn; see Example~\ref{ex:YoungPolyaMainresult}.
\end{remark}

\begin{remark}
If we replace the parameters $(\ell,p,n)$ by $(K\ell,Kp,n)$ for some integer $K>1$, 
we are basically modelling the same triangle, 
yet the limit law is $\PGG(Kp,K\ell,Kp,K\ell)$, which differs from $\PGG(p,\ell,p,\ell)$.
It is noteworthy that one still has some universality: 
the critical exponent $\delta$ remains the same and, besides, the limit laws are closely related in the sense that
they have similar tails. We address these questions in Section~\ref{slope}.
\end{remark}

\begin{proof}
As this proof involves several technical lemmas (which we prove in the next subsections),
we first present its structure 
so that the reader gets a better understanding of the key ideas.
Our proof starts by establishing a link between Young tableaux 
and linear extensions of trees. After that we will be able to conclude via a second link between these trees and 
periodic P\'olya urns. 

Let us begin with Figure~\ref{fig:tabtreeurn} which describes the link 
between the main characters of this proof: the Young tableau~$\Yc$ 
and the ``big'' tree $\cal T$ (which contains the ``small'' tree $\cal S$).
More precisely, we define the rooted planar tree $\cal S$ as follows:
\begin{itemize} 
\item The leftmost branch of~$\cal S$ is a sequence of vertices which we call $v_1, v_2, \dots$  
\item Set $m:=n\ell$. The vertex $\V$ (the one in black in Figure~\ref{fig:tabtreeurn}) has $p-1$ children. 
\item For $2\leq k\leq n-1$, the vertex $v_{k\ell}$ has $p+1$ children.
\item All other vertices $v_j$ (for $j<m, j\neq k\ell$) have exactly one child.
\end{itemize}
\pagebreak

Now, define $\Tc$ as the ``big'' tree obtained from the ``small'' tree $\Sc$ by adding a vertex~$v_0$ as the parent of~$v_1$ 
and adding a set $\Sc'$ of children to~$v_0$. The size of~$\Sc'$ is chosen such that $|\Tc|= 1+|\Sc|+|\Sc'|= 1+N$, 
where $N$ is the number of cells of the Young tableau~$\cal Y$.
Moreover, the hook length of each cell (in grey) in the first row 
of~$\cal Y$ is equal to the hook length\footnote{The hook length of a vertex in a tree is the size of the subtree rooted at this vertex.} 
of the corresponding vertex (in grey also) in the leftmost branch of~$\Sc$.

Let us now introduce a linear extension~$E_{\cal T}$ of~$\cal T$, i.e.~a bijection
from the set of vertices of~$\cal T$ to $\{1, \ldots, N+1\}$ 
such that $E_{\cal T}(u)<E_{\cal T}(u')$ whenever $u$ is an ancestor of $u'$.
A key result, which we prove hereafter in Proposition~\ref{theo:tableautree}, 
 is the following: if $E_{\cal T}$ is a uniformly random linear extension of~$\cal T$, 
then $\Ynv$
(the entry of the south-east corner $v$ in a uniformly random Young tableau $\cal Y$) has the same law as $E_{\cal T}(\V)$: 
\begin{equation}\label{eq1} 1+\Ynv \= E_{\cal T}(\V). \end{equation}

Note that in the statement of the theorem, $\Ynv$ is denoted by $\Yn$ to initially help the reader to follow the dependency on~$n$.

Furthermore, recall that $\cal T$ was obtained from $\cal S$ by adding a root and some children to this root. Therefore, one can obtain
a linear extension of the ``big'' tree\, $\cal T$ from
 a linear extension of the ``small'' tree $\cal S$. 
In Section~\ref{linkyu}, we show that this allows us to construct a uniformly random linear extension~$E_{\cal T}$ of~$\cal T$ 
and a uniformly random linear extension~$E_{\cal S}$ of~$\cal S$ such that
\begin{equation} \label{eq2} |{\cal T} |-E_{\cal T}( \V )
\= n \, (\Ss-E_{\cal S}(\V) + \text{smaller order error terms}). \end{equation} 

The last step, which we prove in Proposition~\ref{prop:urntree}, is that 
\begin{equation} \label{eq3}\quad \Ss-E_{\cal S}(\V) \= \text{distribution of periodic P\'olya urn} + \text{deterministic quantity.} \end{equation}

Indeed, more precisely 
$\Ss-E_{\cal S}(\V)$ has the same law as the number of black balls in a periodic urn
after $(n-1)p$ steps (an urn with period $p$, with parameter~$\ell$, and with initial conditions 
$b_0=p$ and $w_0=\ell$). Thus, our results on periodic urns from Section~\ref{sec:Moments} 
and the conjunction of Equations~\eqref{eq1}, \eqref{eq2}, and \eqref{eq3} give the convergence in law for $\Ynv$ which we wanted to prove.
\end{proof}

The subsequent sections are dedicated to the proofs of the auxiliary propositions that are crucial for the proof of Theorem~\ref{TheoremCorner}. First, we establish
a link between our problem on Young tableaux and a related problem on trees. Second, we explain
the connection between the related problem on trees and the model of periodic urns.

\pagebreak

\begin{figure}[ht!]
		\begin{center}	
			\includegraphics[width=.6\textwidth]{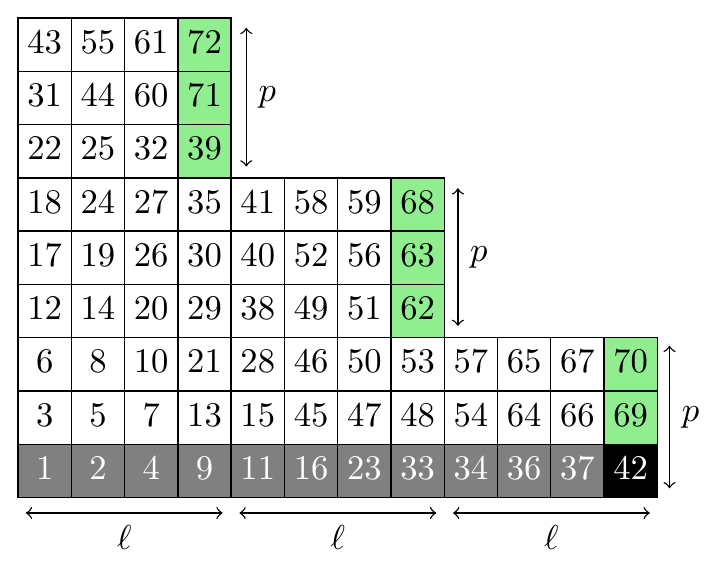} \\ \bigskip \smallskip 
			\includegraphics[width=.87\textwidth]{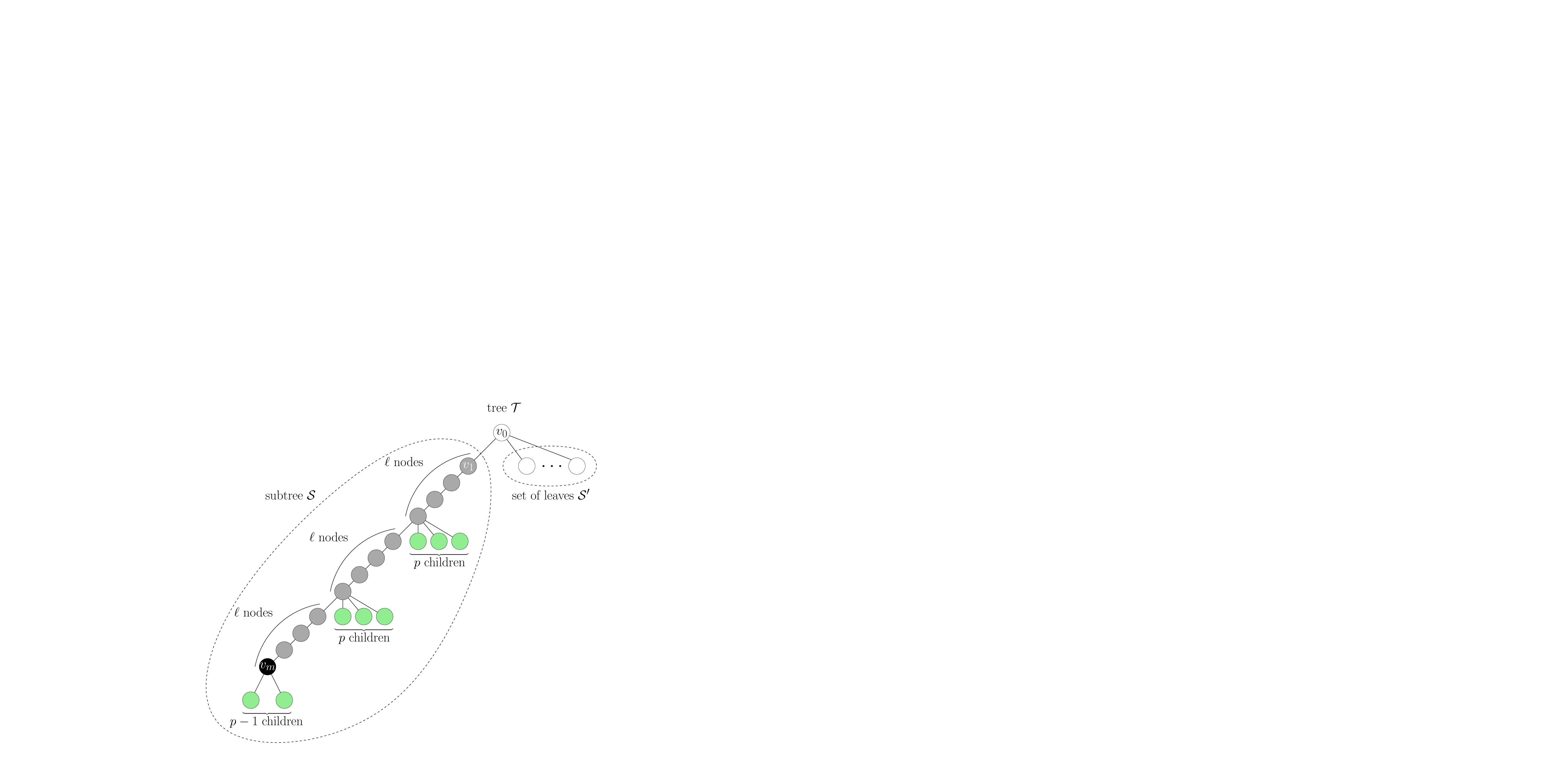}
\end{center} \internallinenumbers
			\caption{In this section, we see that there is a relation between Young tableaux with a given periodic shape, some trees, and the
			periodic Young--P\'olya urns. 
			The key observation is that the cells (in grey) in the first row of the tableaux have the same hook lengths as the nodes (in grey) in the leftmost branch of the tree. 
			The south-east cell $v$ (in black) of this Young tableau has also the same hook length as the node $\V$ (in black) in the tree, 
and is following the same distribution we proved for urns (generalized gamma product distribution).} 		\label{fig:tabtreeurn}
\end{figure}

\clearpage

\subsection{The link between Young tableaux and trees}\label{sec4.1}

We will need the following definitions. 

\begin{definition}[The shape of a tableau\footnote{Some authors define the shape of a tableau as its row lengths from bottom to top. In this article we use the list of column lengths, as it directly gives the natural quantities to state our results in terms of trees and urns.}]\label{ShapeTableau} 
We say that a tableau has \textit{shape} $\lambda_1^{i_1}\cdots \lambda_n^{i_n}$
(with $\lambda_1> \dots > \lambda_n$)
if it has (from left to right) first $i_1$ columns of height $\lambda_1$, etc., and ends with 
$i_n$ columns of height $\lambda_n$.
\end{definition}
As an illustration, the tableau on the top of Figure~\ref{fig:tabtreeurn} has shape $9^4 6^4 3^4$.

\begin{definition}[The shape of a tree] \label{ShapeTree}
Consider a rooted planar
tree $\Tc$ with at least two vertices
and having the shape of a ``comb'': at each level only the leftmost node can have children.
It has \textit{shape} $(i_0,j_0; i_1,j_1; \ldots; i_n, j_n)$ if
\begin{itemize}
\item
when $n=0$, then $\Tc$ is the tree with $j_0$ leaves and $i_0$ internal nodes, all of them unary except for the last one which has $j_0$ children;
\item
when $n\geq 1$, then $\Tc$ is the tree with \textit{shape} 
$(i_0,j_0; i_1,j_1; \ldots; i_{n-1}, j_{n-1})$ to which we attach a tree of shape $(i_n,j_n)$ as a new leftmost subtree to
the parent of the leftmost leaf.
\end{itemize}
\end{definition}

Figure~\ref{fig:treeshapeex} illustrates the recursive construction of a tree of shape $(1,4; 1,2; 2,2)$.
As another example, the tree~$\Tc$ in Figure~\ref{fig:tabtreeurn} has shape $(1,|\Sc'|;4,3 ;4,3; 4,2)$, where $|\Sc'|$ stands for the number of leaves in $\Sc'$.	

\begin{figure}[hb!]
		\begin{center}	
			\includegraphics[width=.78\textwidth]{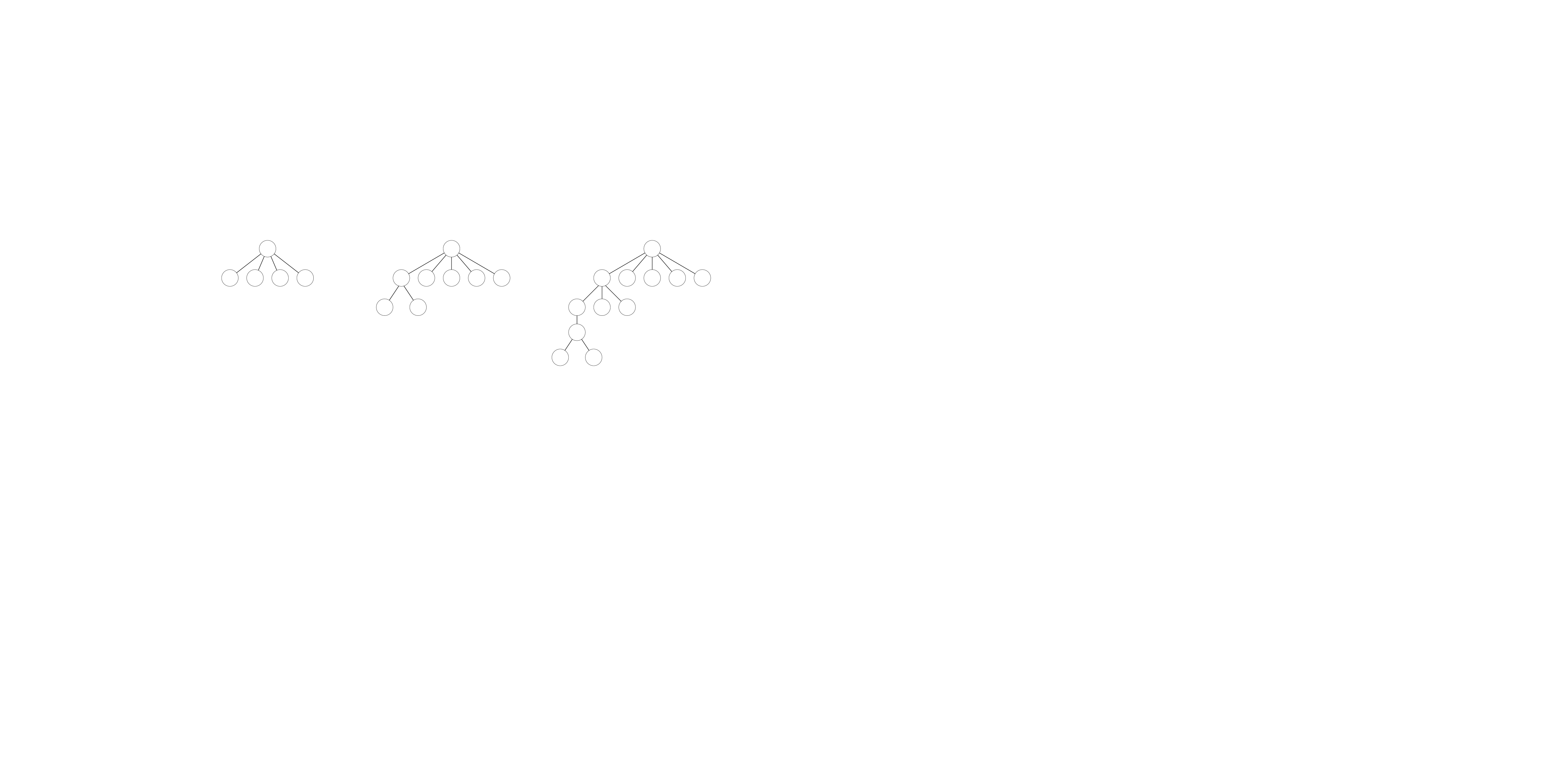}\internallinenumbers
			\caption{The recursive construction of a tree of shape $(1,4; 1,2; 2,2)$. First, a tree of shape $(1,4)$, second, a tree of shape $(1,4; 1,2)$, third, a tree of shape $(1,4; 1,2; 2,2)$.} \label{fig:treeshapeex}
		\end{center}
\end{figure}

Let us end this small collection of definitions with a more classical one:
\begin{definition}[Linear extension of a poset and of a tree] \label{LinearExtension}
A \textit{linear extension}~$E$ of a poset~$\Ac$ of size $N$ 
is a bijection between this poset and $\{1,\ldots, N\}$ satisfying $E(u)\leq E(v)$ whenever $u\leq v$. 
Accordingly, a linear extension of a tree $\Ac$ with $N$ vertices
is a bijection $E$ between the vertices of $\Ac$ and $\{1,\ldots, N\}$ 
satisfying $E(u)\leq E(v)$ whenever $u$ is a child of~$v$.\newline We denote by $\ext(\Ac)$ the number of linear extensions of $\Ac$.
\end{definition}

\begin{remark}
	In combinatorics, a linear extension is also called an increasing labelling. In the sequel, we will sometimes say ``(increasing) labelling'' instead of ``linear extension'',
hoping that this less precise terminology will help the intuition of the reader.
\end{remark}

We are now ready to state the following result:
\begin{proposition}[Link between the south-east corner of Young tableaux and linear extensions of trees]\label{theo:tableautree} 
\gdef\propclaim{Fix 
a tableau with shape $\lambda_1^{i_1} \cdots \lambda_n^{i_n}$ and consider
a random uniform Young tableau~$\Yc$ with this given shape. Let $E_{\Yc}(v)$ 
be the entry of the south-east corner of this Young tableau.
Let $\Tc$ be a tree with shape 
$(1,N-m-\lambda_1+1;$ $i_1,\lambda_1-\lambda_2; i_2, \lambda_2-\lambda_3; \ldots; i_n,\lambda_n-1)$,
where $N=\sum \lambda_k i_k$ is the size of the tableau $\Yc$ and $m=i_1 +\dots +i_n$
is the number of its columns.
Let $E_\Tc$ be a random uniform linear extension of~$\Tc$,
and $\V$ be the $m$-th vertex in the leftmost branch of this tree~$\Tc$.
Then $E_\Tc(\V)$ and $1+E_{\Yc}(v)$ have the same law.}
\propclaim 
\end{proposition}
\begin{proof}
The proof will be given on page~\pageref{proof:theo:tableautree},
as it requires two ingredients, which have their own interest 
and which are presented in the two next sections (\ref{densityYoung} on the density method for Young tableaux, 
and \ref{densitytrees} on the density method for trees).
\end{proof}

\begin{example}
	Let us apply the previous result to the tree of shape $(1,4;1,2;$ $2,2)$ from Figure~\ref{fig:treeshapeex}.
	There we have $n=2$, $m=3$.
	Then, this tree corresponds to a Young tableau of shape $5^1 3^2$ and size $N=11$.
\end{example}

\begin{remark}
In the simplest case when the tableau is a rectangle (i.e.~it has shape $\lambda_1^{i_1}$), the associated tree has shape
$(1,(\lambda_1-1)(i_1-1); i_1,\lambda_1-1)$. In that case, the law of $E_\Tc(\V)$ is easy to compute and we get
an alternative proof of the following formula, first established in~\cite{Marchal16}:
\begin{equation*}
\Prob(E_{\Yc}(v)=k)=\frac{\binom{k-1}{i_1-1}\binom{\lambda_1\,i_1\,-\,k}{\lambda_1-1}}
{\binom{\lambda_1\,i_1}{\lambda_1+i_1-1}}.
\end{equation*}
\end{remark}

The fact that $\Yc$ and $\Tc$ are related is obvious from the construction of~$\Tc$,
but it is not a priori granted that it will lead to a simple, nice link
between the distributions of~$v$ and~$v_m$ (the two black cells in Figure~\ref{fig:tabtreeurn}). 
So, $E_\Tc(\V) \= 1+E_{\Yc}(v)$ deserves a detailed proof: 
it will be the topic of the next subsections. The proof has a nice feature: it uses a generic method, which we call the \textit{density method}
and which was introduced in our articles~\cite{Marchal18,BanderierMarchalWallner18a}.
In fact, \textit{en passant}, these next subsections also illustrate the efficiency of the density method 
in order to enumerate (and to perform uniform random generation)
of combinatorial structures (like we did in the two aforementioned articles for permutations with some given pattern, 
or rectangular Young tableaux with ``local decreases'').

The advantage of Proposition~\ref{theo:tableautree} is that linear extensions of a tree are 
easier to study than Young tableaux and can, in fact, be related to our periodic urn models, as 
shown in Section~\ref{densitytrees}.

\subsection{The density method for Young tableaux}\label{densityYoung}

Trees and Young tableaux can be viewed as posets~\cite{Stanley86}.
We will use this point of view to prove
Proposition~\ref{theo:tableautree}. We recall here some general facts
that will be useful in the sequel.

\begin{definition}[Order polytope of a poset]
Let $\mathcal{A}$ be a general poset with cardinality $N$ and order relation $\leq$. 
We can associate with $\mathcal{A}$ a polytope $\Pc \subset[0,1]^{\mathcal{A}}$ 
defined by the condition $(Y_e)_{e \in \mathcal{A}}\in \Pc$ if and only if $Y_e\leq Y_{e'} $ whenever $e\leq e'$. 
Then $\Pc$ is called the \textit{order polytope} of the poset~$\mathcal{A}$.
\end{definition}

\begin{example}
	\label{ex:orderpoly}
	Let $\Ac$ be the set of subsets of $\{a,b\}$ ordered by inclusion.
	Then its order polytope is given by $\Pc = \{ (Y_{\emptyset},Y_{\{a\}},Y_{\{b\}},Y_{\{a,b\}}) \in [0,1]^4 : $ $Y_{\emptyset} \leq Y_{\{a\}},$ $Y_{\emptyset} \leq Y_{\{b\}},$ $Y_{\emptyset} \leq Y_{\{a,b\}},$ $Y_{\{a\}} \leq Y_{\{a,b\}},$ $Y_{\{b\}} \leq Y_{\{a,b\}} \}$.
\end{example}

Let $Y=(Y_e)_{e \in \mathcal{A}} \in [0,1]^{\mathcal{A}}$ be a tuple of random variables\footnote{When the poset is a Young tableau, this corresponds to what is called a Poissonized Young tableau in~\cite{GorinRahman18}.} chosen according to the uniform measure on the polytope $\Pc$.
Then we consider the function $X$ having integer values, defined by $X_e:=k$ if $Y_e$ is the $k$-th smallest real in the set of reals $\{Y_e : e\in \mathcal{A}\}$. It is sometimes called order statistic. Note that $X$ is a random variable, defined almost surely as we have a zero probability that some marginals of~$Y$ have the same value,
and $X$ is uniformly distributed on the set of all linear extensions of~$\mathcal{A}$. 
The last claim holds because the wedges of each linear extension have equal size $1/N!$ for $N = |\Ac|$ being the size of the poset $\Ac$.

\begin{example}
	Continuing Example~\ref{ex:orderpoly}, there are two linear extensions of $\Ac$: $(X_{\emptyset},$ $X_{\{a\}},$ $X_{\{b\}},$ $X_{\{a,b\}})$ $= (1,2,3,4)$ and $(X_{\emptyset},$ $X_{\{a\}},$ $X_{\{b\}},$ $X_{\{a,b\}})$ $= (1,3,2,4)$. 	
	They correspond to the following two wedges in $\Pc$: $Y_{\emptyset} \leq Y_{\{a\}} \leq Y_{\{b\}} \leq Y_{\{a,b\}}$ and $Y_{\emptyset} \leq Y_{\{b\}} \leq Y_{\{a\}} \leq Y_{\{a,b\}}$. 
	The volume of each of them is $1/24$, while the volume of $\Pc$ is $1/12$.
\end{example}

Conversely, if $X$ is a random uniform increasing labelling of $\mathcal{A}$, one gets a random variable~$Y$ on the polytope $\Pc$ via 
$Y_e:=T_{X_e}$, where $T$ is a random uniform $N$-tuple from the set $\left\{ (T_1,\dots,T_N) \in [0,1]^N :T_1 < \dots <T_N \right\}$.
Therefore, $Y$ is uniformly distributed on $\Pc$.
What is more, $T_k$ is the $k$-th largest uniform random variable among $N$ 
independent uniform random variables. 
Thus, it has density
$ k\binom{N}{k} x^{k-1}(1-x)^{N-k}$. 
As a consequence, for any~$e\in\mathcal{A}$, $Y_e$ has density
\begin{equation}
\label{eq:orderstat2}
g_e(x)=\sum_{k=1}^{N} \PR(X_e=k) k \binom{N}{k} x^{k-1}(1-x)^{N-k}.
\end{equation}

\pagebreak

This formula can be read as two different writings of the same polynomial in two different bases;
thus, by elementary linear algebra, it implies that $\PR(X_e=k)$ can be deduced from the polynomial~$g_e$.
In particular, we have the following property:

\begin{lemma}
\label{lemma:fact}
Let $\mathcal{A}$, $\mathcal{A}'$ be two posets with the same cardinality, and let $\Pc$, $\Pc'$ be their respective order polytopes. 
Let $X$ (resp. $X'$) be a random linear extension of~$\mathcal{A}$ (resp. $\mathcal{A}'$).
Let $Y$ (resp. $Y'$) be a uniform random variable on $\Pc$ (resp. $\Pc'$).
Then, for any $e \in \Ac$ and $e' \in \Ac'$, such that $Y_e$ and $Y'_{e'}$ have the same density, $X_e$ and $X'_{e'}$ have the same law.
\end{lemma}

Let $\Yc$ be a tableau with shape $\lambda_1^{i_1} \dots \lambda_n^{i_n}$ and total size $N=\sum_k {\lambda_k i_k}$.
We view $\Yc$ as a poset: $\Yc$ is a set of $N$ cells equipped
with a partial order ``$\leq$'', where $c\leq c'$ if one can go from~$c$ to $c'$ with only north and east steps. 
We denote by $\mathcal{P}$ the order polytope of the tableau $\Yc$.

We will introduce an algorithm generating a random element of $\mathcal{P}$ according to the uniform measure.
In order to do so, we fill the diagonals one by one. Let us introduce some notation.
The tableau $\Yc$ can be sliced into $M=\lambda_1+i_1 +\dots +i_n-1$ diagonals $D_1,\ldots, D_M$ as follows: 
$D_1$ is the north-west corner and recursively, $D_{k+1}$ is the set of cells which are adjacent to one of the cells of
$D_1\cup\dots \cup D_k$ and which are not in $D_1\cup\dots \cup D_k$. In particular,
$D_M$ is the south-east corner. For example, Figure~\ref{fig:tabtreeurn} has $M=20$ such diagonals.

Note that between two consecutive diagonals $D_k$ and $D_{k+1}$ (let us denote their cell entries by $y_1 < \cdots < y_j$ and $x_1< \dots < x_{j'}$), 
there exist four different interlocking relations illustrated by Figure~\ref{fig:diagtypes}. 
The shape of the tableau implies that for each~$k$ we are in one of these four possibilities, each of them thus corresponds to a polytope $\mathcal{P}_k$ defined as:

\vspace{-3mm}
\begin{align}
	\label{case1} \text{ case 1:} &&& \mathcal{P}_k := \{ y_1< x_1< \dots< y_j< x_j \}, \\
	\label{case2} \text{ case 2:} &&& \mathcal{P}_k := \{ x_1< y_1< \dots< x_j< y_j\}, \\
	\label{case3} \text{ case 3:} &&& \mathcal{P}_k := \{ y_1< x_1< \dots< x_{j-1} < y_j\}, \\
	\label{case4} \text{ case 4:} &&& \mathcal{P}_k := \{ x_1< y_1< \dots< x_j< y_j< x_{j+1}\}. 
\end{align}

\begin{figure}[ht!]
		\begin{center}	
			\includegraphics[width=.23\textwidth]{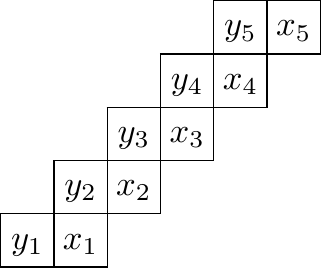} \,
			\includegraphics[width=.23\textwidth]{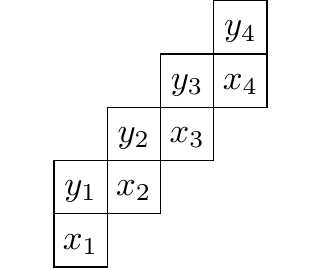} \, 
			\includegraphics[width=.23\textwidth]{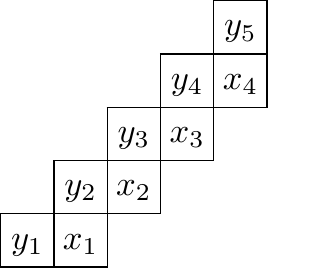} \,
			\includegraphics[width=.23\textwidth]{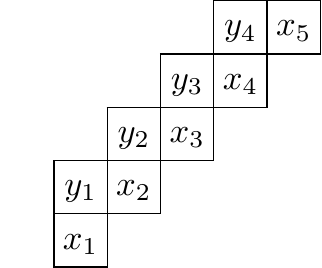} \internallinenumbers
			\caption{Young tableaux of any shape can be generated by a sequence of ``diagonals'', which interlock according to the four possibilities above.} 
			\label{fig:diagtypes}
		\end{center}
\end{figure}

Our algorithm will make use of conditional densities along the $M$ diagonals of $\Yc$.
For this purpose, for every $k\in \{1,\ldots,M\}$ we define a polynomial 
$g_k$ in $|D_k|$ variables as follows.
First, one sets $g_1:=1$; the next polynomials are defined by induction. 
Suppose that $1\leq k\leq M-1$ and $D_k=(y_1<\dots< y_j)$.
The four above-mentioned possibilities for $D_{k+1}$ lead to the definition 
of the following polynomials.
\begin{enumerate}
\item
In the first case (interlocking given by~\eqref{case1}), this gives
\begin{equation}
g_{k+1}(x_1,\ldots, x_j) :=
\int_{0}^{x_1} dy_1 \int_{x_1}^{x_2}dy_2 \ldots \int_{x_{j-1}}^{x_j}dy_{j} \, 
g_k (y_1,\ldots, y_j).
\end{equation}
\item
In the second case (interlocking given by~\eqref{case2}), this gives
\begin{equation}
g_{k+1}(x_1,\ldots, x_j) :=
\int_{x_1}^{x_2} dy_1 \int_{x_2}^{x_3}dy_2 \ldots \int_{x_{j-1}}^{x_j}dy_{j-1}
\int_{x_j}^{1}dy_j \, g_k (y_1,\ldots, y_j).
\end{equation}
\item
In the third case (interlocking given by~\eqref{case3}), this gives
\begin{equation}
g_{k+1}(x_1,\ldots, x_{j-1}) :=
\int_{0}^{x_1} dy_1 \int_{x_1}^{x_2}dy_2 \ldots \int_{x_{j-1}}^{1}dy_{j} \, 
g_k (y_1,\ldots, y_j).
\end{equation}
\item
In the fourth case (interlocking given by~\eqref{case4}), this gives
\begin{equation}
g_{k+1}(x_1,\ldots, x_{j+1}) :=
\int_{x_1}^{x_2} dy_1 \int_{x_2}^{x_3}dy_2 \ldots \int_{x_{j}}^{x_{j+1}}dy_{j} \, 
g_k (y_1,\ldots, y_j).
\end{equation}
\end{enumerate}

Now, we use these polynomials to formulate a random generation algorithm which will also be able to enumerate 
the corresponding Young tableaux. Note that faster random generation algorithms are known
(like the hook walk from~\cite{GreeneNijenhuisWilf84}), 
but it is striking that the above polynomials~$g_k$ will be the key 
to relate the distributions of different combinatorial structures, allowing us to capture
second order fluctuations in Young tableaux, trees, and urns.
It is also noteworthy that our density method is in some cases 
the most efficient way to enumerate and generate combinatorial objects 
(see~\cite{BanderierMarchalWallner18a} for applications 
on variants of Young tableaux, where the hook length formula is no more available,
and see~\cite{Devroye86} for algorithmic subtleties related to sampling conditional multivariate densities).

Recall that $\mathcal{P}$ is
the order polytope of the tableau $\Yc$ and that we want to 
generate a random element of $\mathcal{P}$ according to
the uniform measure. 
The algorithm is the following. We generate by descending induction
on $k$, for each diagonal $D_k$, a $|D_k|$-tuple of reals in 
$[0,1]$ which will be the entries of the cells of $D_k$.

First, remark that the functions defined by~\eqref{algo1Gm} and~\eqref{conditionaldensity}
in Algorithm~\ref{algo1} 
are indeed probability densities. 
That is, they are measurable, positive functions and their integral is equal to~$1$. 
To prove this, remark first that these functions are polynomials and therefore measurable. 
Next, by definition, as integrals of positive functions, they are positive. 
Finally, the fact that the integral is equal to~$1$ follows from their definition. 

\pagebreak

\begin{coloralgorithm}[a random uniform Young tableau~$\Yc$, via the density method]
\label{algo1}
\begin{itemize} 
\item[Step 1.] 
Recall that $D_M$ is the south-east corner. Generate the corresponding cell entry at random with probability density 
\begin{equation}\label{algo1Gm}
\frac{g_M(x)}{\int_0^1 g_M(y) \, dy}.
\end{equation}
\item[Step 2.]
By descending induction on $k$ from $M-1$ down to $1$, 
generate the diagonal $D_k$ (seen as a tuple of $|D_k|$ reals in $[0,1]$) according to the density
\begin{equation}\label{conditionaldensity}
\frac{g_k(x_1,\ldots, x_{|D_k|}) } {g_{k+1}(D_{k+1})} {\bf 1}_{\Pc_k},
\end{equation}
where $g_k$ and ${\bf 1}_{\Pc_k}$ are chosen according to the cases given by~\eqref{case1}, \eqref{case2}, \eqref{case3}, \eqref{case4}.
\end{itemize}
\end{coloralgorithm}

\medskip

We then claim that Algorithm~\ref{algo1} 
yields a random element $(D_1,\ldots, D_M)$ of $\mathcal{P}$ with the uniform measure. 
Indeed, by construction, its density is the product of the conditional densities of the diagonals~$D_1,\ldots, D_M$. 
The crucial observation now is that the product of the conditional densities~\eqref{conditionaldensity} is a telescopic product,
so the algorithm generates each Young tableau $\Yc$ with the same ``probability'' (or more rigorously, as we have continuous variables, with the same \textit{density}):
\begin{equation}\label{telescop}
\frac{g_M(D_M)}{\int_0^1 g_M(y) \, dy}
 \prod_{k=1}^{M-1} \frac{g_k(D_k) } {g_{k+1}(D_{k+1})}{\bf 1}_{\Pc_k}= 
\frac{{\bf 1}_{\{\Yc \in \mathcal{P}\}}}{\int_0^1 g_M(y) \, dy}.
\end{equation}

This indeed means that our algorithm yields a uniform random variable on the order polytope~$\mathcal{P}$.
Alternatively, one can say that the Young tableau $\Yc$ is a random variable on $[0,1]^N$ with density given by~\eqref{telescop}, therefore
\begin{equation}
\int_{[0,1]^N} {\bf 1}_{\{Z\in \mathcal{P}\}} \, dZ=\int_0^1 g_M(y) \, dy.
\end{equation}
Now, suppose that we pick uniformly at random an element $Z'$ of $[0,1]^N$. Then one has
\begin{equation}
\PR(Z' \in \mathcal{P}) \,\,=\,\, \int_{[0,1]^N} {\bf 1}_{\{Z\in \mathcal{P}\}} \, dZ
\,\,=\,\, \frac{\fill(\Yc)}{N!},
\end{equation}
where $\fill(\Yc)$ is the number of increasing labellings (linear extensions) of the tableau $\Yc$. Thus, 
\begin{equation*}
\fill(\Yc)=N!\int_0^1 g_M(y) \, dy.
\end{equation*}

In the next section, we turn our attention to the density method for trees.

\subsection{The density method for trees}\label{densitytrees} 

Let the tree $\Tc$, its subtree $\Sc$, and the vertices $v_0,\dots,v_m$ be defined as on page~\pageref{eq1} (see Figure~\ref{fig:tabtreeurn}).
As in Section~\ref{densityYoung}, it is possible to construct a random linear extension of~$\Sc$ by using a uniform random variable~$Y$ on the order polytope of~$\Sc$. 
The vertex $v_m$ has then a random value $Y_{v_m}$ between~$0$ and~$1$, and we want to compute its density. 
To this aim, we associate to each internal node~$v_k$ a polynomial $f_k$ (in
$\sigma_k$ variables, where $\sigma_k$ is the number of siblings of~$v_k$). 
These polynomials $f_k$ are defined by induction starting with $f_1:=1$, while $f_2,\dots, f_{m-1}$ are defined by
\begin{equation}
f_{k}(x_0,\ldots, x_{\sigma_{k}}):=\int_0^{\inf \{x_0,\ldots, x_{\sigma_{k}}\}}dy_0 \int_0^1dy_1\ldots \int_0^1dy_{\sigma_{k-1}} \, f_{k-1}(y_0,y_1,\ldots, y_{\sigma_{k-1}}),
\end{equation}
The last polynomial, $f_m$, additionally depends on the number $j$ of children of~$v_m$:
\begin{align}
\label{h}\qquad
\begin{aligned}
 &f_m(x_0,\ldots, x_{\sigma_m}) := \\ 
                 &(1-x_0)^{j} \int_0^{\inf \{x_0,\ldots, x_{\sigma_m}\}} dy_0 \int_0^1 \! dy_1\ldots \int_0^1 \! dy_{\sigma_{m-1}} \, f_{m-1}(y_0,y_1,\ldots, y_{\sigma_{m-1}}).
\end{aligned}
\end{align} 
\gdef\hvm{h_{v_m}}
We also define $\hvm$:
\begin{equation}
\hvm(x):= \int_0^1 dx_1 \ldots \int_0^1dx_{\sigma_m} \, f_m(x,x_1,\ldots, x_{\sigma_m}).
\end{equation}
We claim that $\hvm(x)$ is (up to a multiplicative constant) the density of $Y_{v_m}$. 
This is shown as in Section~\ref{densityYoung} using 
Algorithm~\ref{algo2}, which generates uniformly at random a labelling of $\Sc$.

\begin{coloralgorithm}[a random uniform increasing labelling $Y$ of the tree~$\Sc$]
\label{algo2}
\begin{itemize}
\item[Step 1.]
Generate $Y_{v_m}$ according to the density 
$\frac{\hvm(x)}{\int_0^1 \hvm(x) \, dx}$.
\item[Step 2.]
If $v_m$ has $j$ children 
$s_1,\ldots, s_{j}$, then
generate $(Y_{ s_1},\ldots, Y_{s_j})$ according to the density
\begin{equation}
\frac{\prod_{i=1}^j {\bf 1}_{\{y_i> Y_{v_m}\}}}{(1-Y_{v_m})^j}.
\end{equation}
\item[Step 3.]
If $v_m$ has $j$ siblings $s_1,\ldots, s_j$, 
then generate $(Y_{ s_1},\ldots, Y_{s_j})$ according to the density
\begin{equation}
\frac{ f_m(Y_{v_m}, y_1,\ldots, y_j)} {\int_0^1dy_1\ldots \int_0^1dy_j \, f_m(Y_{v_m}, y_1,\ldots, y_j)}.
\end{equation}
\item [Step 4.]
By descending induction for $k$ from $m-1$ down to 1, if $v_k$ has $j$ siblings $s_1,\ldots,s_j$, then generate the tuple ${\bf Y}_k=(Y_{ v_k}, Y_{s_1},\ldots, Y_{s_j})$
according to the density
\begin{equation}
\frac{f_k(y_0,\ldots, y_j)} 
{f_{k+1}({\bf Y}_{k+1})}
{\bf 1}_{\{ y_0 < \min {\bf Y}_{k+1}\}}.
\end{equation}
\end{itemize}
\end{coloralgorithm}

\pagebreak

Indeed, the random tuple $Y$ generated by this algorithm is by construction an element of the order polytope.
What is more, we have the uniform distribution, as the probabilities of all $Y$'s are equal
to a telescopic product similar to Formula~\eqref{telescop}. 
Therefore $h_m(x)$ is (up to a multiplicative constant) the density of $Y_{v_m}$ and the number $\ext(\Sc)$ of linear extensions of~$\Sc$ is given by \vspace{-2mm}
\begin{equation}
\ext(\Sc)=|\Sc|!\int_0^1 \hvm(x) \, dx.
\end{equation}

\vspace{-2mm}

It remains to connect the densities of~$v$ in $\Yc$ and $v_m$ in $\Sc$; we do this in the following lemma.

\begin{lemma}
	\label{lemma:filament}
	The polynomial $g_M(x)$ (which gives the density of~$v$, the south-east corner of the Young tableau $\Yc$)
 and the polynomial $\hvm(x)$ (which gives the density of~$v_m$ in the tree $\Sc$) are equal up to a multiplicative constant:
	\begin{align*}
		\hvm(x) &= c\, g_M(x) & \text{ with } &&
		c &= \frac{|\Yc|!}{|\Sc|!}\frac{\ext(\Sc)}{\fill(\Yc)}.
	\end{align*}
\end{lemma}

\begin{proof}
	The main idea of the proof consists in adding a filament to the tree and to the tableau,
 and inspecting the consequences 
 via the density method. 
 
\textit{Part 1 (adding a filament to the tableau).} Let $\Yc_L$ be the tableau obtained 
	by adding to~$\Yc$ $L$~cells horizontally to the right of its south-east corner $v$ (and denote these new cells by $e_1,\dots,e_L$).
	We can generate a random element of the order polytope of $\Yc_L$
	as follows: remark that $\Yc$ is a subtableau of $\Yc_L$ and that the first $M$ diagonals $D_1,\ldots, D_M$
	of $\Yc_L$ are the same as the first $M$ diagonals of $\Yc$ 
	(recall that the diagonals are 	lines with positive slope $+1$, starting from each cell of the first column and row).
	In particular, $D_M$ is the south-east corner cell~$v$. 
	Then, we can extend Algorithm~\ref{algo1} in the following way:
	\begin{coloralgorithm}[a random uniform increasing labelling $X$ of the tableau with $L$ added cells]
	\label{algo3} 
	\begin{itemize}
	\item[Step 1.]
	Generate $X_{M,L}$ the entry of the cell~$v$ according to the density 
	
 \vspace{-2.7mm} 
\begin{equation}
	\frac{g_{M,L}(x)}{\int_0^1 g_{M,L}(y) \, dy}
	\text{\qquad where \qquad}
	g_{M,L}(x):= \frac{g_M(x)(1-x)^L}{L!}.
	\end{equation}
	\item[Step 2.]
	Generate the entries of the diagonals $D_{M-1},\ldots, D_1$ 
	as in Algorithm~\ref{algo1}.
	\item[Step 3.]
	Generate the entry $X_1$ of $e_1$ with density 
	\begin{equation}
	L\frac{(1-x)^{L-1}}
	{(1-X_{M,L})^L} {\bf 1}_{\{x> X_{M,L} \}}.
	\end{equation}
	\item[Step 4.]
	For $i$ from $1$ to $L-1$, 
 generate the entry $X_{i+1}$ of $e_{i+1}$ with density 
	\begin{equation}
	(L-i)\frac{(1-x)^{L-i-1}}	{(1-X_i)^{L-i}} {\bf 1}_{\{x> X_i \}}.
	\end{equation}
	\end{itemize}
	\vspace{-1\baselineskip}
	\end{coloralgorithm}

\clearpage

	Using the same arguments as for Algorithm~\ref{algo1}, we can show that 
	Algorithm~\ref{algo3} yields a uniform random variable on the
	order polytope of $\Yc_L$ and that 
	the number of increasing labellings of $\Yc_L$ is
	\begin{equation*}
	\fill(\Yc_L) = (N+L)!\int_0^1 g_{M,L}(y) \, dy= (N+L)!\int_0^1 \frac{g_{M}(y)(1-y)^L}{L!} \, dy.
	\end{equation*}

	On the other hand, using the hook length formula, we see that the hook 
	lengths of $\Yc_L$ are the same as those of $\Yc$, except for the first row. A straightforward computation shows that
	\begin{equation*}
	\frac{\fill(\Yc)}{N!}=\frac{\fill(\Yc_L)}{(N+L)!}\times G_L,
	\end{equation*}
	where, as $\Yc$ has shape $\lambda_1^{i_1} \cdots \lambda_n^{i_n}$, the constant $G_L$ is given by
	\begin{equation}\label{G_L}
	G_L=L!\prod_{k=1}^n \frac{ \ffac{(i_1+\dots+i_k+L+\lambda_k-1)}{i_k}}
	{ \ffac{(i_1+\dots+i_k+\lambda_k-1)}{i_k}},
	\end{equation}
	where we reuse the falling factorial notation
	$\ffac{a}{b}=a(a-1)\cdots(a-b+1)$.
	Therefore, this leads to
\begin{equation} \label{fillFc}
	\int_0^1 g_{M}(y)(1-y)^L \, dy= \frac{L!}{G_L} \frac{\fill(\Yc)}{N!}.
\end{equation}	

\textit{Part 2 (adding a filament to the tree).} Suppose that we extend the tree $\Sc$ by adding a filament of length $L$.
Let $\Sc_L$ be the tree obtained from $\Sc$ by attaching to~$v_m$ a subtree consisting of a line with $L$ vertices. Put 
\begin{equation}f_L(x):=\frac{(1-x)^{L}\hvm(x)}{L!}.\end{equation}
With the same arguments as for the function $\hvm$ defined in \eqref{h}, we
see that $f_L/\int_0^1f_L(x) \, dx$ is the density of 
$Y_L(v_m)$ where $Y_L$ is a uniform random variable on the order polytope of~$\Sc_L$. 
Following the same reasoning, we can show that the number of linear extensions of~$\Sc_L$ is
\begin{equation}\ext(\Sc_L)=(|\Sc|+L)!\int_0^1 f_L(y) \, dy.\end{equation}
On the other hand, recall that a version of the hook length formula holds for trees (see e.g.~\cite{SaganYeh89,Han10,KubaPanholzer16a}):
the number of linear extensions of a tree of size~$N$ is given by
\begin{equation}\frac{N!}{\prod_{v\in S}\operatorname{hook}(v)},\end{equation} 
where here $\operatorname{hook}(v)$ is the number of descendants of~$v$ (including $v$ itself). 

\noindent Applying this formula to the tree $\Sc$ yields
\begin{equation}\frac{\ext(\Sc)}{|\Sc|!}=\frac{\ext(\Sc_L)}{(|\Sc|+L)!}\times G_L,\end{equation}
with the same $G_L$ as in~\eqref{G_L}.
Indeed, the most crucial point is that the hook lengths of the Young tableau on the first row
are \textit{the same} as the hook lengths of the tree along the leftmost branch. 
This key construction allows us to connect these two structures. Hence, one has
\begin{equation}\label{eqSc}
\int_0^1 \hvm(y)(1-y)^L \, dy= \frac{L!}{G_L} \frac{\ext(\Sc)}{|\Sc|!}.
\end{equation}

\textit{Part 3 (conclusion): linking tableaux and trees.}
Comparing~\eqref{fillFc} with~\eqref{eqSc}, we see that for every integer $L\geq 1$,
\begin{equation*}\int_0^1 \hvm(y)(1-y)^L \, dy=c\int_0^1 g_{M}(y)(1-y)^L \, dy,\end{equation*}
where $c$ is the constant given by
\begin{equation*}c=\frac{|\Yc|!}{|\Sc|!}\frac{\ext(\Sc)}{\fill(\Yc)}.\end{equation*}
Since $\hvm(x)$ and $g_M(x)$ are polynomials, this implies that $\hvm=c \, g_M$.
\end{proof}

Before establishing the final link between Young tableaux and urns,
we start by collecting what we got via the density method:
this gives the proof of Proposition~\ref{theo:tableautree}, which we now restate.

\begingroup
\newtheorem*{proposition*}{Proposition \ref{theo:tableautree}}
\begin{proposition*}[Link between the corner of a Young tableau and linear extensions of trees]
\propclaim
\end{proposition*}
\begin{proof}\label{proof:theo:tableautree}
The reader is invited to have a new look on Figure~\ref{fig:tabtreeurn} (page~\pageref{fig:tabtreeurn}), 
 which illustrates for this proof 
the idea of the trees $\Tc$, $\Sc$, and the set of leaves~$\Sc'$.
We first introduce a forest $\Tc^*:=\Sc \cup \Sc'$ obtained by adding $N-m-\lambda_1+1$ vertices without any order relation to the tree~${\Sc}$. 
$\Tc^*$ has an order relation inherited from the order relation $\leq$ on $\Sc$: 
two nodes~$x,y$ of~$\Tc^*$ are comparable if and only if they belong to~$\Sc$ and in that case, the order relation on~$\Tc^*$ is the same as the one on~$\Sc$.

Let $\mathcal{P'}$ be the order polytope of~$\Sc$. Then it is clear that the order polytope of~$\Tc^*$ is
\begin{equation*} \mathcal{P}=\mathcal{P'}\times [0,1]^{N-m-\lambda_1+1}. \end{equation*}
In particular, if $Y'$ is a uniform random variable on $\mathcal{P'}$ and if $Y$ is a uniform random variable on~$\mathcal{P}$, then $Y'_v$ and $Y_v$ have the same density. 
This density is proportional to the function $\hvm$ computed in Section~\ref{densitytrees}.
Next, recall the notation $g_M$ and $D_M$ from Section~\ref{densityYoung}. 
Lemma~\ref{lemma:filament} gives that $\hvm=c \, g_M$.
Thus, the density of $Y_{v_m}$ is the same as the density of $D_M$. Moreover,
$\Tc^*$ and $\Yc$ have the same cardinality. Therefore, Lemma~\ref{lemma:fact}
entails that if $E_{\Tc^*}$ is a random uniform linear extension
of~$\Tc^*$ and if $\Ynv$ is the entry of the south-east corner in a random increasing labelling
of~$\Yc$, 
then $E_{\Tc^*}(v_m)$ and $\Ynv$ have the same distribution. 

Now, it is easy to deduce from $E_{\Tc^*}$ a random uniform linear extension~$E_{\Tc}$ of~$\Tc$:
set $E_{\Tc}(u)=1$ if $u$ is the root of~$\Tc$, and 
set $E_{\Tc}(u)=1+E_{\Tc^*}(u)$ for the other nodes (since any such node $u$ can be identified as a node of~$\Tc$).
Applying this to the vertex $v_m$ finishes the proof of Proposition~\ref{theo:tableautree}.
\end{proof}
\endgroup

\subsection{The link between trees and urns}\label{linkyu} 

In order to end the proof of Theorem~\ref{TheoremCorner}, we need two more propositions.
\vspace{-1mm}

\begin{proposition}[Link between trees and urns]\label{prop:urntree} 
 Consider a tree $\Sc$ with shape $(i_1,j_1;\ldots; i_n, j_n)$.
Let $v$ be the parent of the leftmost leaf if $j_n\geq 1$, or the 
leftmost leaf if $j_n=0$. Let $E_{\Sc}$ be a random uniform 
linear extension of~$\Sc$.

Let $X=|\Sc|-E_{\Sc}(v)$. Then, $X$ has the same law as the number of black balls in the following urn process:
\begin{itemize}
\item
Initialize the urn with $b_0:=j_n+1$ black balls and $w_0:=i_n$ white balls.
\item 
For $k$ from $n-1$ to $1$, perform the following steps:
\begin{enumerate}
\item
Perform $j_{k}-1$ times the classical P\'olya urn with replacement matrix 
$	\begin{pmatrix}
			1 & 0 \\
			0 & 1
		\end{pmatrix}$. 
\item 
Make one transition with the replacement matrix 
$	\begin{pmatrix}
			1 & i_k \\
			0 & 1+ i_k
		\end{pmatrix}$. 
\end{enumerate}
\end{itemize}
\end{proposition}

\begin{remark} 
Note that the urn scheme described in the proposition is precisely the model of periodic P\'olya urns 
covered by Theorem~\ref{ProdGenGammaGeneral}.
For Young--P\'olya urns, one has
$i_{k}=\ell$ and $j_{k}= p$ for $k<n$, and $i_n=\ell$ and $j_n=p-1$, compare Figure~\ref{fig:tabtreeurn}. 
\end{remark}

\begin{proof}[Proof (Proposition~\ref{prop:urntree})]
First consider the transition probabilities in the classical P\'olya urn.
At step $i>0$ the composition $(B_i,W_i)$ is obtained from $(B_{i-1},W_{i-1})$ 
by adding a black ball with probability
$\frac{B_{i-1}}{B_{i-1}+ W_{i-1}}$ and a white ball with probability
$\frac{W_{i-1}}{B_{i-1}+ W_{i-1}}$.
We will now show that the same transition probabilities are imposed by the linear extension of the tree.

\begingroup
\def\RR{\mathcal R}
We start with a definition. If $\RR\subset \Sc$ 
we define $E_{\RR}:\RR\to \{1,\ldots,|\RR|\}$ as the only bijection preserving the order relation induced by~$E_\Sc$. 
That is, $E_{\RR}(u)=k$ if and only if $E_{\Sc}(u)$ is the $k$-th smallest value in the set~$\{E_{\Sc}(r) : r\in \RR\}$. 
It is easy to check that $E_{\RR}$ is a uniform linear extension of~$\RR$ seen as a poset equipped with the order relation inherited 
from~$\Sc$.
\endgroup

Let us prove our claim. 
On the one hand, for every vertex $w$ which is one of the $j_n$ children of~$v$, we have $E_{\Sc}(w) > E_{\Sc}(v)$. 
On the other hand, for every vertex $u$ which is one of the $(i_n-1)$-st most recent ancestors 
of~$v$, we have $E_{\Sc}(u) < E_{\Sc}(v)$. 
Let $S_0$ be the set consisting of~$v$, all 
its children and its $(i_n-1)$-st most recent ancestors; see Figure~\ref{fig:tree2urn}.

We will perform two nested inductions. The outer one is decreasing from $k=n-1$ to~$1$, and each inner one increasing from $1$ to $j_k$.
We start with $k=n-1$.
First, let $u_n$ be the $i_n$-th most recent ancestor of~$v$.
The node $u_n$ has $j_{n-1}$ children which are not ancestors of~$v$.
Call these $u_{n,1},\ldots, u_{n,j_{n-1}}$. 
Let $S_1:=S_0\cup\{u_{n,1}\}$,
then $E_{S_1}(u_{n,1})$ is uniformly distributed on $\{1,\ldots, |S_1|\}$. 
As a consequence, $E_{S_1}(u_{n,1}) > E_{S_1}(v)$ with probability 
$(j_n+1)/(j_n+1+i_n)$. This probability can be expressed as
$\frac{b_0}{b_0+w_0}$,
where $b_0$ is the number of vertices $u$ in $S_0$ such that 
$E_{\Sc}(u) \geq E_{\Sc}(v)$ and $w_0$
is the number of vertices $u$ in $S_0$ such that 
$E_{\Sc}(u) \leq E_{\Sc}(v)$.
Conditionally on the initial configuration $S_0$, this defines two random variables: let $B_1$ be the number of vertices 
$u$ in $S_1$ such that 
$E_{\Sc}(u) \geq E_{\Sc}(v)$ and $W_1$
be the number of vertices $u$ in $S_1$ such that 
$E_{\Sc}(u) \leq E_{\Sc}(v)$. 

Next, let $S_2:=S_1\cup\{u_{n,2}\}$,
then $E_{S_2}(u_{n,2})$ is uniformly distributed on $\{1,\ldots,|S_2|\}$. 
Then, conditionally on $B_1$ and $W_1$, one has
$E_{S_2}(u_{n,2}) \geq E_{S_2}(v)$, with probability 
$\frac{B_1}{B_1+W_1}$.
This process is then continued by induction until $S_{j_{n-1}}$. 
After that $i_{n-1}$ white balls are added. 

Continuing this process via a decreasing induction in $k$ from $n-2$ to $1$ finishes the proof.
\end{proof}
\begingroup
\setlength{\belowcaptionskip}{-15pt}
\begin{figure}[!ht]
		\begin{center}	
			\includegraphics[width=.56\textwidth]{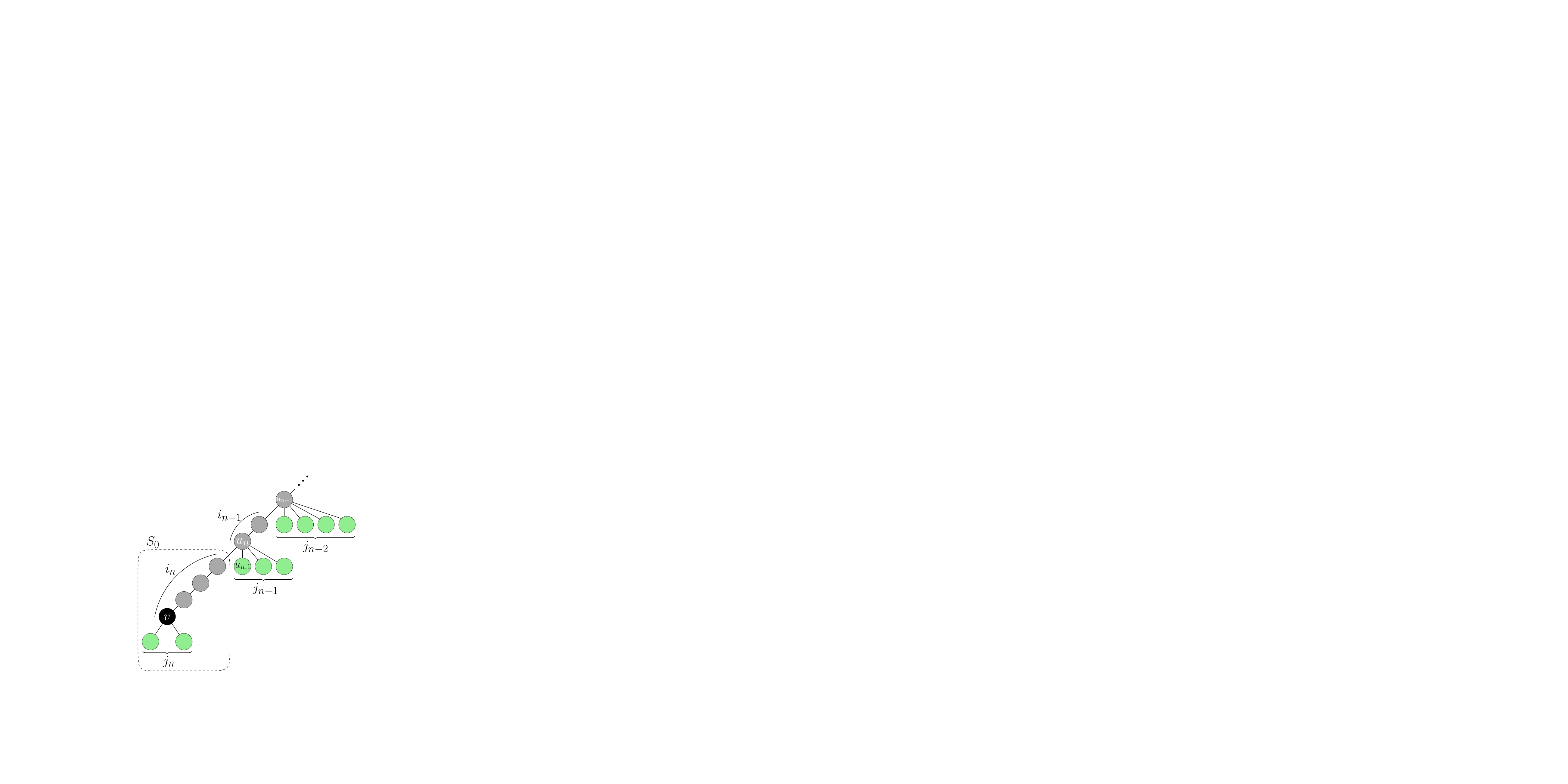}\internallinenumbers
			\caption{Proposition~\ref{prop:urntree} relates the labels in the tree $\Sc$ with a P\'olya urn process.	
			For periodic shapes, it gives a periodic P\'olya urn.
			The initial conditions are given by $S_0$. 
			The tree is traversed bottom to top, along vertices not in the leftmost branch, starting at $u_{n,1}$. 
			Each of these nodes corresponds to a classical P\'olya urn step, whereas each vertex in the leftmost branch corresponds to an additionally added white ball.}
 \label{fig:tree2urn} 		
	\end{center}
\end{figure}
\endgroup

Our final proposition requires first the following basic lemma.
\begingroup
\def\Ss{s}
\begin{lemma}[Order statistics comparisons]\label{orderstatistics}
Let $(Z_i, 1\leq i\leq N-\Ss-1)$ be independent, uniform random variables on $[0,1]$
and let $Z$ be a random variable on $[0,1]$, independent of each $Z_i$,
and distributed like $\operatorname{Beta}(a,\Ss+1-a)$.
Let $I$ be the number of indices $i\geq 1$ such that $Z_i<Z$.
Then, one has
\begin{equation}\label{orderstatistics1}\E(I)=\frac{(N-\Ss-1)\, a}{\Ss+1}\end{equation}
and 
\begin{equation}\label{orderstatistics2}\E(I^2) = \frac{a (N-\Ss-1) ((a+1) N-(\Ss+2) a)}{(\Ss+1)(\Ss+2)}.\end{equation}
\end{lemma}
\pagebreak
\begin{proof}
The density of the beta distribution $Z$ was already encountered 
 in Equation~\eqref{eq:orderstat2}; $Z$ is thus the $a$-th order statistic of the uniform distribution. 
It is easily seen that for all $1\leq i<j\leq N-\Ss-1$,
\begin{equation}
	\PR(Z_i<Z)=\frac{a}{\Ss+1} \text{ \qquad and \qquad }
	\PR(Z_i<Z, Z_j<Z)=\frac{a(a+1)}{(\Ss+1)(\Ss+2)}.
\end{equation}

Moreover, writing the random variable $I$ as $I=\sum_{i=1}^{N-\Ss-1}{\bf 1}_{\{ Z_i<Z\}}$, we get
\begin{equation}
\E(I)=\sum_{i=1}^{N-\Ss-1} \PR(Z_i<Z)=\frac{(N-\Ss-1)a}{\Ss+1}
\end{equation}
\begin{eqnarray*}
\hspace{2cm} \E(I^2)&=&\sum_{1\leq i\ne j\leq N-\Ss-1}\PR(Z_i<Z, Z_j<Z)+\sum_{i=1}^{N-\Ss-1} \PR(Z_i<Z)\\ 
&=& (N-\Ss-1)(N-\Ss-2) \frac{a(a+1)}{(\Ss+1)\, (\Ss+2)}
+\frac{(N-\Ss-1)a}{\Ss+1}\text.\hspace{4mm}\qedhere 
\end{eqnarray*} 
\end{proof}
\endgroup

In order to finish the proof of Theorem~\ref{TheoremCorner},
we still have to relate $\Ss-E_{\cal S}(\V) $ to the quantity that we are interested in, namely
$N-E_{\cal T}(\V)$.

\begin{proposition}[Same asymptotic densities] \label{prop:TeqS}
The random variables $E_{\cal S}(\V)$ and $E_{\cal T}(\V)$ satisfy asymptotically the following link: for any $s,t\in \R^+$, one has
\begin{equation}\label{TeqS}
\lim_{n \to \infty} \PR\left(s<\frac{\Ss-E_{\cal S}(\V)}{n^{\delta}} < t \right)=
\lim_{n \to \infty} \PR\left(s<\frac{2(p+\ell)}{p \ell} \frac{N-E_{\cal T}(\V)}{n^{1+\delta}}<t \right).
\end{equation}
\end{proposition}
\begin{proof}

Let $\cal T^*=\Sc\cup\Sc'$ be the graph obtained from $\cal T$ by removing the root. 
Then $\cal T^*$ is a poset where there is no order relation between any vertex of~$\Sc'$ and any other vertex from $\cal T^*$. 
Due to this independence, the order polytope of~$\cal T^*$ is
 the Cartesian product of the order polytope of~$\Sc$ and $[0,1]^{|\Sc'|}$.
Now, let $a>0$ be an integer and 
let $F_a$ be the event that 
\begin{equation*}\Ss-E_{\cal S}(\V)=a.
\end{equation*}
In other words, $a$ is the number of vertices in ${\cal S}$ with a label greater than $E_{\cal S}(\V)$.
Let $I$ be the random variable counting the number of vertices in ${\cal S'}$ with a label greater than $\E_{\cal T}(\V)$.
Then, conditionally on the event $F_a$, the random variable $N-E_{\cal T}(\V)$ has the same law as $I+a$.
Indeed, $N-E_{\cal T}(\V)$ counts the number of vertices in ${\cal T}$ with a label greater than $E_{\cal T}(\V)$.
Note that $I$ satisfies the conditions of Lemma~\ref{orderstatistics} (with $s:=\Ss$ therein),
due to the order polytope independence mentioned above.

Recall that $\Ss=\Theta(n)$ while $N=\Theta(n^2)$
(in fact, $\Ss=(p+\ell)n-1$ and $|\mathcal T|=N=\frac{1}{2} p\ell n (n+1)$). 
Therefore, if $(a_n)_{n\geq 1}$ is a sequence of integers tending to $+\infty$ and such that $a_n=o(n)$, 
then, thanks to~\eqref{orderstatistics1}, 
we have the estimates for the conditional expectation
\begin{equation}\label{muI}
\E(I|F_{a_n})\sim \frac{a_nN}{ \Ss}\sim c \, n a_n,
\end{equation}																		with the constant $c = \frac{p \ell}{2 (p+\ell)}$										 
and, thanks to~\eqref{orderstatistics2}, for the conditional variance 
\begin{equation}\label{varI}
\operatorname{var}(I|F_{a_n})=\E(I^2|F_{a_n}) - (\E(I|F_{a_n}))^2 \sim 
c^2 n^2 a_n.
\end{equation}
Combining~\eqref{muI} and~\eqref{varI}, the Bienaym\'e--Chebyshev inequality gives that (for any $\kappa>0$):

\begin{equation}\label{titi}
\PR\left(\left\{\left|\frac{I}{cna_n}-1\right|>\kappa\right\}|F_{a_n}\right)
\leq \frac{1+\varepsilon_n}{\kappa^2a_n}\,,
\end{equation}
where $\varepsilon_n$ is a sequence converging to 0 as $n\to \infty$. Since we have
\begin{equation}
\frac{N-E_{\cal T}(\V)}{na_n}=\frac{I+a_n}{na_n}=\frac{I}{na_n}+\frac{1}{n},
\end{equation}
the inequality~\eqref{titi} can be rewritten into
\begin{equation}\label{toto}
\PR\left(\left\{\left|\frac{N-E_{\cal T}(\V)}{cna_n}-1\right|>\kappa\right\}|F_{a_n}\right)
\leq \frac{1+\varepsilon'_n}{\kappa^2a_n}\,,
\end{equation}
where $\varepsilon'_n$ is a sequence converging to 0 as $n\to \infty$.
In particular, for any $t>0$ and $0<\delta<1$, 
setting $a_n=\left \lceil{ tn^\delta}\right \rceil$ in~\eqref{toto} gives
\begin{equation}\label{dirac}
\PR\left(\left\{\left|\frac{N-E_{\cal T}(\V)}{c n^{1+\delta}}-t\right|>\kappa t\right\}|F_{a_n}\right)
\leq \frac{1+o(1)}{\kappa^2 t n^{\delta}}\,.
\end{equation}

Finally, for all reals $0<s<t$, define the event 
\begin{equation} F_{s,t}\,\,=\,\,\bigcup_{sn^\delta<a<tn^\delta}F_a\,\,=\,\,
\left\{
	s < \frac{\Ss-E_{\cal S}(\V)}{n^{\delta}} < t \, 
\right\}.
\end{equation}
According to~\eqref{dirac} (set $\kappa = \varepsilon/t$ for any $\varepsilon>0$), we have for $n\to\infty$
\begin{equation}
	\PR\left( \left\{s<\frac{N-E_{\cal T}(\V)}{cn^{1+\delta}}<t\right\} | F_{s,t}\right) \to 1. 
\end{equation}
Thus, conditioning on the complementary event ${\bar F}_{s,t}$, we have 
\begin{equation}\label{Zurich1}
\lim_{n \to \infty} \PR\left(\left\{s<\frac{N-E_{\cal T}(\V)}{cn^{1+\delta}}<t\right\}\cap {\bar F}_{s,t} \right) = 0,
\end{equation}
whereas conditioning on $F_{s,t}$ gives
\begin{equation}\label{Zurich2}
\lim_{n \to \infty} \PR\left(\left\{s<\frac{N-E_{\cal T}(\V)}{cn^{1+\delta}}<t\right\}\cap F_{s,t} \right) = \lim_{n \to \infty} \PR\left(s<\frac{\Ss-E_{\cal S}(\V)}{n^{\delta}} < t \right).
\end{equation} 
Summing~\eqref{Zurich1} and~\eqref{Zurich2} 
leads to~\eqref{TeqS}.
\end{proof}

\smallskip

In summary, in this Section we have proven that the four following quantities
have asymptotically the same distribution: 
\begin{equation}\label{fullstory}
\frac{2}{p \ell} \frac{N-E_{\Yc}(v)}{ n^{1+\delta}}
 \underset{\mathllap{
 \begin{tikzpicture}
 \draw[->] (-0.3, 0) to[bend right=20] ++(0.3,2ex);
 \node[below left,text width=3cm] at (0,0) {\phantom{i}Proposition~\ref{theo:tableautree} (density method)};
 \end{tikzpicture}
 }}{\,\,\,=\,\,\,}
\frac{2}{p \ell} \frac{N-E_{\Tc}(\V)}{ n^{1+\delta}}
 \underset{
	 \mathclap{
 \begin{tikzpicture}
 \draw[->] (0, 0) to ++(0,2ex);
 \node[below,text width=3cm] at (0,0) {Proposition~\ref{prop:TeqS} \phantom{\,}(order statistics)};
 \end{tikzpicture}
		}}
	{\,\,\,\sim\,\,\,}
 \frac{1}{p+\ell} \frac{\Ss-E_{\Sc}(\V)}{ n^{\delta}}
 \underset{\mathrlap{
 \begin{tikzpicture}
 \draw[->] (0.3, 0) to[bend left=20] ++(-0.3,2ex);
 \node[below right,text width=2.8cm] at (0,0) {Proposition~\ref{prop:urntree} \phantom{aai}(P\'olya urn)};
 \end{tikzpicture}
 }}{\,\,\,=\,\,\,}
\frac{1}{p+\ell} \frac{B_{(n-1)p}}{n^\delta}.
\end{equation} 
In conjunction with Theorem~\ref{theo:PGG} proven via analytic combinatorics methods, this implies that the four quantities 
in~\eqref{fullstory}
converge in law to the distribution $\PGG(p,\ell,b_0,w_0)$, when $\delta=p/(p+\ell)$.
This is exactly the statement of Theorem~\ref{TheoremCorner}.

\smallskip

\noindent\textit{Nota bene:} It should be stressed that the sequence of transformations in~\eqref{fullstory} is \textit{not} a bijection between Young tableaux and urns,
it is only \textit{asymptotically} that the corresponding distributions are equal.

\smallskip

\begingroup 
The perspicacious reader would have noted that in the previous pages, we used several small 
lemmas and propositions which were stated with slightly more generality than what was a priori needed.
In fact, this now allows us to state an even stronger version of Theorem~\ref{TheoremCorner}.
(It would have been not pedagogical to introduce it first: we think it would have been
harder for the reader to digest the different key steps/definitions/figures used in the proof.)
In order to state this generalization to any Young tableau with a more general periodic shape,
we need a slight extension of the shape $\lambda_1^{i_1}\cdots \lambda_n^{i_n}$
introduced in Definition~\ref{ShapeTableau}:
we allow some of the indices $i_k$ to be equal to zero, in which case there is no column of height $\lambda_k$:
\begin{definition}[Periodic tableau] \label{def:PeriodicYoungTableau}
 For any tuple of nonnegative integers $(\ell_1,\ldots, \ell_p)$, a tableau with periodic pattern shape $(\ell_1,\ldots, \ell_p ; n)$ is a tableau with shape
\begin{equation*} \big( (np)^{\ell_p}(np-1)^{\ell_{p-1}}\cdots (np-p+1)^{\ell_1} \big) \,\times\, \big( ((n-1)p)^{\ell_p}\cdots ((n-1)p-p+1)^{\ell_1}\big) \,\times\, \cdots \,\times\, \big(p^{\ell_p} \cdots 1^{\ell_1}\big).
\end{equation*}
 A uniform random Young tableau with periodic pattern shape $(\ell_1,\ldots, \ell_p ; n)$ 
 is a uniform random 
 filling of a tableau with periodic pattern shape $(\ell_1,\ldots, \ell_p ; n)$.
\end{definition}

Let us put the previous pattern in words: 
 we have a tableau made of $n$ blocks, each of these blocks consisting of $p$ smaller blocks of length $\ell_p,\ldots, \ell_1$,
and the height decreases by $1$ between each of these smaller blocks.
This leads to a tableau length $(\ell_1+\dots+\ell_p) n$, which repeats periodically the same sub-shape along its hypotenuse. 

Note that the triangular Young tableau of parameters $(\ell,p,n)$ from Definition~\ref{triangYoung}
corresponds to Definition~\ref{def:PeriodicYoungTableau} for the $(p+1)$-tuple $(0,\ldots, 0,\ell ; n)$. 
In order to state our main result in full generality, we extend the above-defined Young tableau by additional rows from below.

\begin{definition}\label{GeneralPatternFull}
 Let $b_0,w_0 >0$.
 A tableau of shape 
	$\lambda_1^{i_1}\cdots\lambda_n^{i_n}$ 
	\emph{shifted by a block $b_0^{w_0}$} is a tableau of shape 
	$(\lambda_1+b_0)^{i_1}\cdots(\lambda_n+b_0)^{i_n} b_0^{w_0}.$
\end{definition}

We can now state the main theorem of this section:

\begin{theorem}[The distribution of the south-east entry in periodic Young tableaux] \label{TheoremCornerGen}
 Choose a uniform random Young tableau with periodic pattern shape $(\ell_1,\ldots, \ell_p ; n)$ 
	shifted by a block~$b_0^{w_0}$.
Let $N$ be its size, set $\ell := \ell_1+\dots+\ell_p$ and $\delta:=p/(p+\ell)$.
Let $\Yn$ be the entry of the south-east corner.
Then $(N-\Yn)/n^{1+\delta}$ converges in law to the same limiting
distribution as the number of black balls in the periodic Young--P\'olya urn 
with initial conditions $(b_0,w_0)$ and 
with replacement matrices $M_i= 
		\begin{pmatrix}
			1 & \ell_i \\
			0 & 1+\ell_i
		\end{pmatrix}$:
\begin{equation}
\frac{2}{p \ell} \frac{N-\Yn}{ n^{1+\delta} } \stackrel{\mathcal L}{\longrightarrow} \operatorname{Beta}(b_0,w_0) 
	\prod_{\substack{i=1 \\i\neq\ell_1+\dots+\ell_j+j \\ \text{ with } 1\leq j \leq p-1}}^{p+\ell-1} \GG(b_0+w_0+i, p+\ell).
\end{equation}
\end{theorem}
\begin{proof}
One just follows the same steps as in~\eqref{fullstory}.
The final proof holds \textit{verbatim}, only the equality $N = \frac{p \ell}{2}n(n+1)$ has to be replaced by an asymptotic $N \sim \frac{pl}{2}n$, which is anyway the only information that is used.
One then concludes via Theorem~\ref{ProdGenGammaGeneral}. 
\end{proof}

\bigskip
In the next section, we discuss some consequences of our results in the context of limit shapes of random Young tableaux. 
\endgroup

\section{Random Young tableaux and random surfaces}\label{Sec5}
There is a vast and fascinating literature related to the asymptotics of Young tableaux when their shape is free, but the number of cells is going to infinity:
it even originates from the considerations of Erd\H{o}s, Szekeres, and Ulam
on longest increasing subsequences in permutations 
(see~\cite{AldousDiaconis99,Romik15} for a nice presentation of these fascinating aspects). There, algebraic combinatorics and variational calculus appear to play a key r{\^o}le,
as became obvious with the seminal works of Vershik and Kerov, Logan and Shepp~\cite{VershikKerov77, LoganShepp77}.
The asymptotics of Young tableaux when the shape is constrained is harder to handle,
and this section tackles some of these aspects.

\nocite{BorodinOlshanski17,McKayMorseWilf02,Petrov15}
\subsection{Random surfaces}
\begingroup
\begin{figure}[ht!]
		\begin{center}	
			\includegraphics[width=1\textwidth]{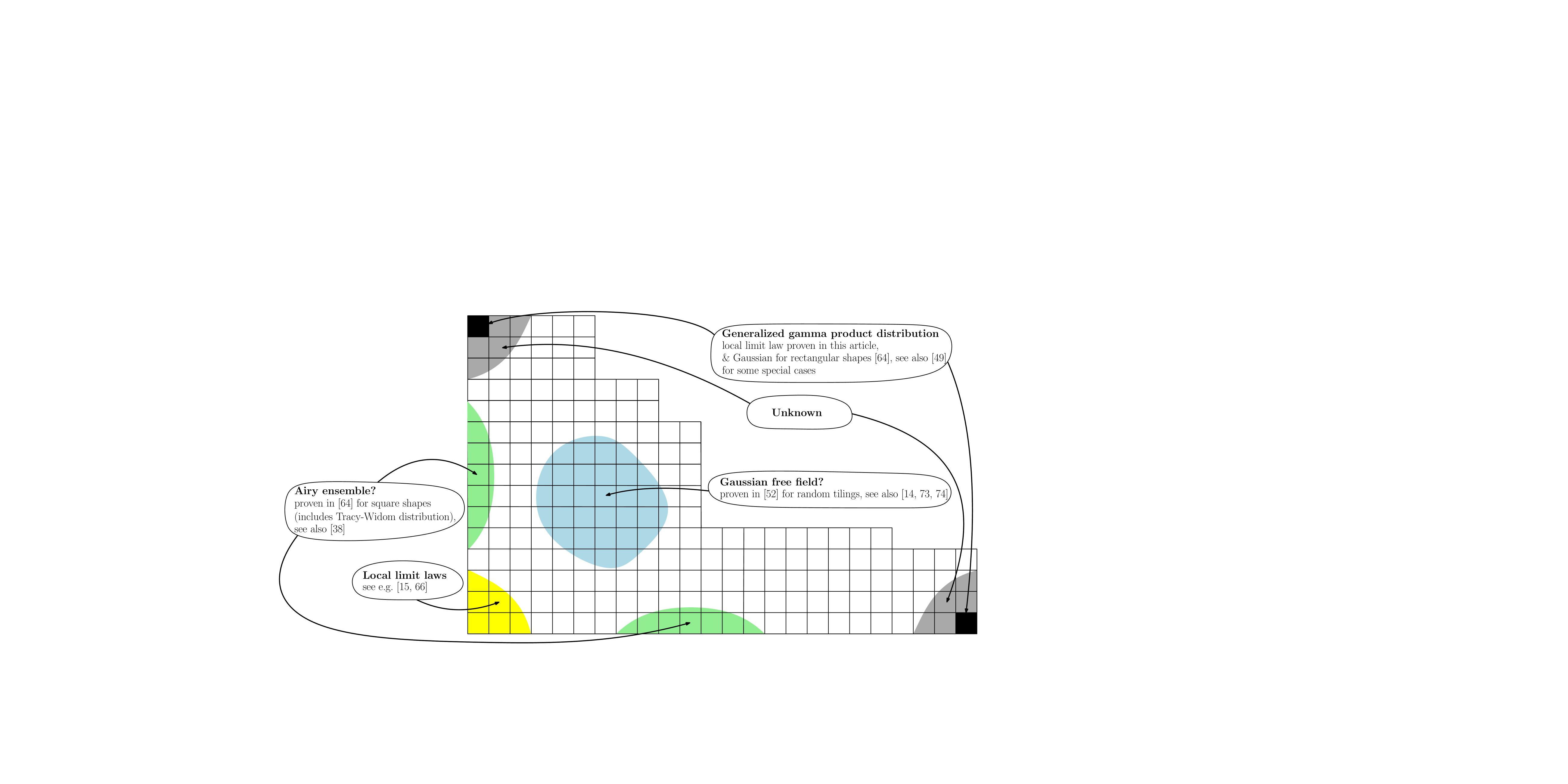}	
			\internallinenumbers
			\caption{Known and conjectured limit laws of random Young tableaux. 
Would it one day lead to a nice notion of ``continuous Young tableau''?} \label{fig:tablimit} 	
		\end{center}
\end{figure}
\endgroup
\def\Tr{\rm Tr} 
Figure~\ref{fig:tablimit} illustrates some known results and some conjectures 
on ``the continuous'' limit of Young tableaux (see also the notion of continual Young tableaux in~\cite{Kerov93}). 
Let us now explain a little bit 
what is summarized by this figure, which, in fact, refers to different levels of renormalization in order to catch the right fluctuations.
It should also be pinpointed that some results are established under the Plancherel distribution, while some others are established under the uniform distribution
(like in the present work).

First, our Theorem~\ref{TheoremCorner} can be seen as a result on random surfaces arising from
Young tableaux with a fixed shape. Let us be more specific. Consider a fixed rectangular triangle
$\Tr$ where the size of the edges meeting at the right angle are $p$ and $q$, respectively, where
$p$ and $q$ are integers. One can approximate $\Tr$ by a sequence of tableaux
$({\cal Y}_n)_{n\geq 0}$ of the same form as~$\cal Y$ in Section~\ref{sec:Tableaux} where the size
of the sides meeting at the right angle are $pn$ and $qn$.

\pagebreak

For each of these tableaux, one can pick a random standard filling and one can interpret it as a random discretized surface. 
More precisely, if $0\leq x\leq p$ and $0\leq y\leq q$ are two reals
and if the entry of the cell $(\left \lfloor{xn}\right \rfloor, \left \lfloor{yn}\right \rfloor)$
is $z$, then we set $f_n(x,y):=2z/(pqn^2)$. Thereby,
we construct a random function $f_n: \Tr\to[0,1]$ which is {\bf dis}continuous but it is to be expected
that, in the limit, the functions $f_n$ converge in probability to a deterministic,
continuous function $f$ (see Figure~\ref{fig:surfacesT}).
Intuitively, for every point $(x,y)$ on the hypotenuse, one will have $f(x,y)=1$ and this is
the case in particular for the south-east corner, that is, the point $(p,0)$. Then, one can view
Theorem~\ref{TheoremCorner} as a result on the fluctuations of the random quantity
$f_n(p,0)$ away from its deterministic limit, which is $1$.

\begingroup
\begin{figure}[ht!]
		\begin{center}	
			\includegraphics[height=55mm]{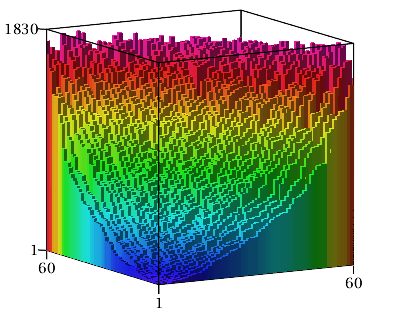} \hfill
			\includegraphics[height=55mm]{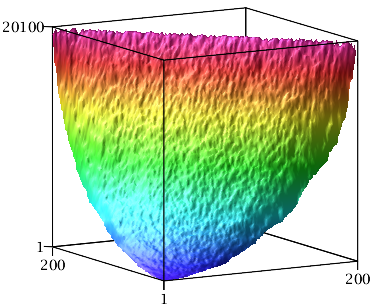} 
		\end{center}
		\begin{center}	
			\includegraphics[width=0.49\textwidth]{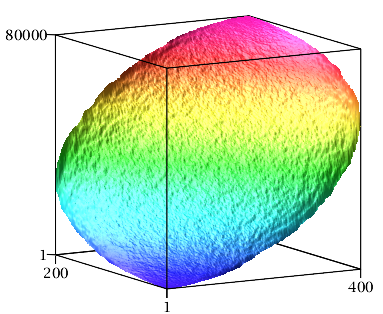} \hfill
			\includegraphics[width=0.47\textwidth]{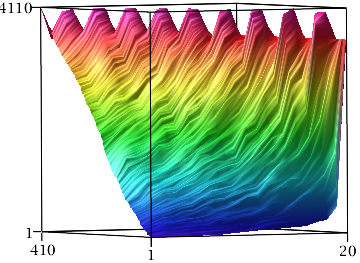} \\
	\internallinenumbers
			\caption{Random generation  
			of Young tableaux, seen as random surfaces (the colours correspond to level lines):\newline
\text{ }\hspace{1.5cm}$\bullet$ top: triangular Young tableaux (size $60\times60$, seen as histogram, and $200\times200$),\newline
\text{ }\hspace{1.5cm}$\bullet$ bottom: rectangular and triangular Young tableaux ($400\times200$ and $410\times 20$).\newline 
If one watches such surfaces from above, then one sees exactly the triangular/rectangular shapes, but one loses the 3D effect.
The images are generated via our own Maple package available at \url{https://lipn.fr/~cb/YoungTableaux}, relying on a variant of the hook-length walk of~\cite{GreeneNijenhuisWilf84}.
}
	\label{fig:surfacesT}
		\end{center}
\end{figure}
\endgroup

\begingroup
As a matter of fact, the convergence of $f_n$ to $f$ has only been studied when the shape of the tableau is fixed.
The convergence towards a limiting surface was first proven
 when the limit shape is a finite union of rectangles; see Biane~\cite{Biane98}. There, the limiting surface 
 can be interpreted in terms of characters of the symmetric group and free probability but this leads to
 complicated computations from which it is difficult to extract explicit expressions.
 
For rectangular Young tableaux, the limiting surface is described more precisely by Pittel and Romik~\cite{PittelRomik07}. 
A limiting surface also exists for staircase tableaux: it can be obtained by taking the limiting surface of a square tableau and cutting it along the diagonal; see~\cite{AngelHolroydRomikVirag07,LinussonPotkaSulzgruber18}. 
This idea does not work for {\em rectangular} (non-square) Young tableaux: 
if one cuts such tableaux along the diagonal, one 
{\em does not} get the limiting surface of triangular Young tableaux 
(the hypotenuse would have been the level line $1$, but the diagonal is in fact not even a level line, as visible in Figure~\ref{fig:surfacesT} and proven in~\cite{PittelRomik07}).

Apart from the particular cases mentioned above, convergence results for surfaces arising from Young tableau 
seem to be lacking. There are also very few results about the fluctuations away from the limiting surface. 
For rectangular shapes, these fluctuations were studied by Marchal~\cite{Marchal16}:
they are Gaussian in the south-east and north-west corner, 
while the fluctuations on each edge follow a Tracy--Widom limit law, 
at least when the rectangle is a square (for general rectangles, there remain some technicalities,
although the expected behaviour is the same). For staircase triangles, Gorin and Rahman~\cite{GorinRahman18}
use a sorting network representation to obtain asymptotic formulas using double integrals. In particular, they find the
limit law on the edge. Their approach may be generalizable to other triangular shapes.
Also, instead of renormalized limits, one may be interested in local limits,
there are then nice links with the famous jeu de taquin~\cite{Sniady14} and characters of symmetric groups~\cite{BorodinOlshanski17}.

There is another framework where random surfaces naturally arise, namely random tilings and related structures
(see e.g.~\cite{Sheffield05}). Indeed,
one can associate a height function with a tiling: this gives an interpretation as a surface. 
In this framework, similarly to the Young tableaux, there are results on the fluctuations of these surfaces. 
In the case of the Aztec diamond shape, Johansson and Nordenstam~\cite{JohanssonNordenstam07} proved
that the fluctuations of the Artic curve are related to eigenvalues of GUE minors 
(and are therefore Gaussian near the places where the curve is touching the edges, 
whereas they are Tracy--Widomian when the curve is far away from the edges). 
Note that this gives the same limit laws as for the Artic curve of a TASEP jump process
associated to rectangular Young tableaux~\cite{Romik12,Marchal16}.
Similar results were also obtained 
for 
pyramid partitions~\cite{BouttierChapuyCorteel17, BoutillierBouttierChapuyCorteelRamassamy17}.
Moreover, in other models of lozenge tilings, it is proven
that for some singular points, other limit laws appear: they are called cusp-Airy distributions,
and are related to the Airy kernel~\cite{DuseJohanssonMetcalfe16}.
It has to be noticed that, up to our knowledge, the generalized gamma distributions, which
appear in our results, have not been found in the framework of random~tilings.

\smallskip
A major challenge would be to capture the fluctuations of the surface in the interior of the domain.
For Young tableaux, it is reasonable to conjecture that 
these fluctuations could be similar to those observed for random tilings:
in this framework, Kenyon~\cite{Kenyon01} and Petrov~\cite{Petrov15}
proved that the fluctuations are given by the Gaussian free field (see also~\cite{BufetovGorin19}).
\smallskip

Finally, a dual question would be: in which cell does a given 
entry lie in a random 
filling of the tableau? In the case of triangular shapes 
like ours, if we look at the largest entry, we get:\endgroup 
\begin{proposition}[Limit law for the location of the maximum in a triangular Young tableau]
Choose a uniform random triangular Young tableau of parameters $(\ell,p,n)$ (see Definition~\ref{triangYoung}). 
Let $\operatorname{Posi}_n\in \{1,\ldots,\ell n\}$ be the $x$-coordinate of the 
cell containing the largest entry. 
Then, one has
\begin{equation*}\frac{\operatorname{Posi}_n}{\ell n} \stackrel{\mathcal L}{\longrightarrow} \operatorname{Arcsine}(\delta), \text{\,\, where $\delta:=p/(p+\ell)$.}
\end{equation*}
\end{proposition}
\begingroup 
\begin{proof}
Remove from the Young tableau $\Yc$ the cell containing its largest entry, and call $\Yc^*$ this new tableau.
Then, using the hook length formula, the probability that the largest entry of $\Yc$ is situated at $x$-coordinate~$k \ell$ is
\begin{equation}
\def\hook{\operatorname{hook}}
\PR(\operatorname{Posi}_n=k \ell)=\frac{ \prod_{c\in \Yc^*} \hook_{\Yc^*}(c)}{ \prod_{c\in \Yc} \hook_\Yc(c)} 
= \prod_{\substack{c \in {\Yc^*} \text{with $(x$-coord of $c) =k\ell$}\\ \text{ or $(y$-coord of $c)= (n-k+1)p$ }}} \frac{\hook_{\Yc^*}(c)}{1+\hook_{\Yc^*}(c)}\,.
\end{equation} An easy computation then gives (with $\delta=p/(p+\ell)$): 

\begin{equation}
\def\hook{\operatorname{hook}}
\PR(\operatorname{Posi}_n=k \ell) \sim
\frac{(k/n)^{\delta-1}(1-k/n)^{-\delta}}{\Gamma(\delta)\Gamma(1-\delta)} \frac{1}{n}.
\end{equation}
Here, one recognizes an instance of the generalized arcsine law 
on $[0,1]$ with density
\begin{equation*}
\frac{x^{\delta-1}(1-x)^{-\delta}}{\Gamma(\delta)\Gamma(1-\delta)}\,. \qedhere
\end{equation*}
\end{proof}
So, if we compare models with different $p$ and $\ell$, 
then the largest entry will have the tendency to be on the top of the hypotenuse when $\ell$ is much larger than $p$, 
while it will be on its bottom if $p$ is much larger than $\ell$ 
(and on the bottom or the top with equally high probabilities when $p\approx \ell$); see~Figure~\ref{fig:surfacesT}. 
This is in sharp contrast with the case of an $n\times n$ square tableau where, for $t\in(0,1)$,
the cell containing the entry $tn^2$ is asymptotically distributed according to the Wigner semicircle law
on its level line; see\cite{PittelRomik07}.
We also refer to Romik~\cite{Romik04} for further discussions on Young tableau landscapes and to Morales, Pak, and Panova~\cite{MoralesPakPanova17} for recent results on skew-shaped tableaux.

\pagebreak

\subsection{From microscopic to macroscopic models: universality of the tails}\label{slope}

\begin{figure}[ht!]
		\begin{center}	
			\includegraphics[width=0.3\textwidth]{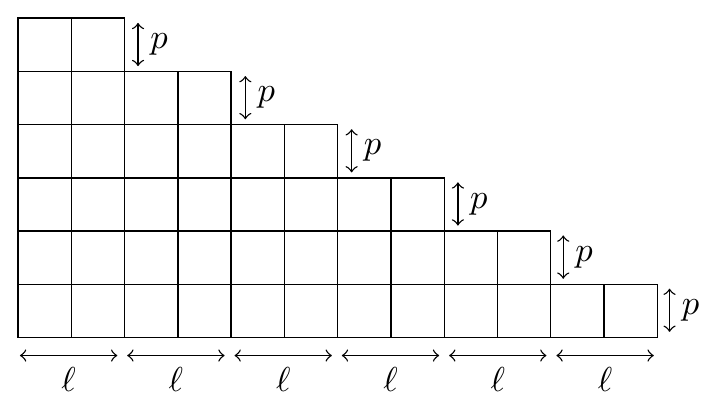} 
			\quad			
			\includegraphics[width=0.3\textwidth]{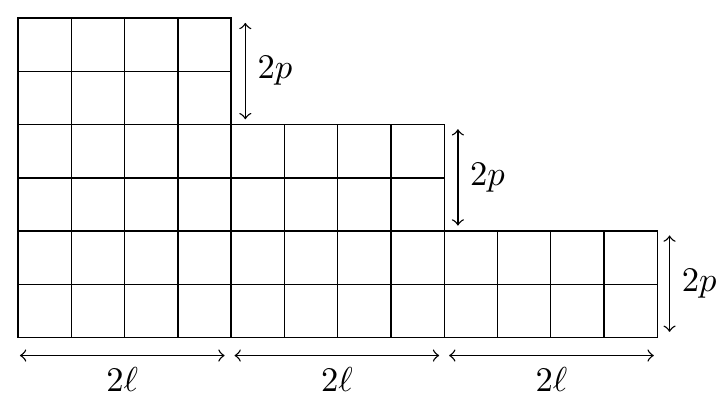} 
			\quad			
			\includegraphics[width=0.3\textwidth]{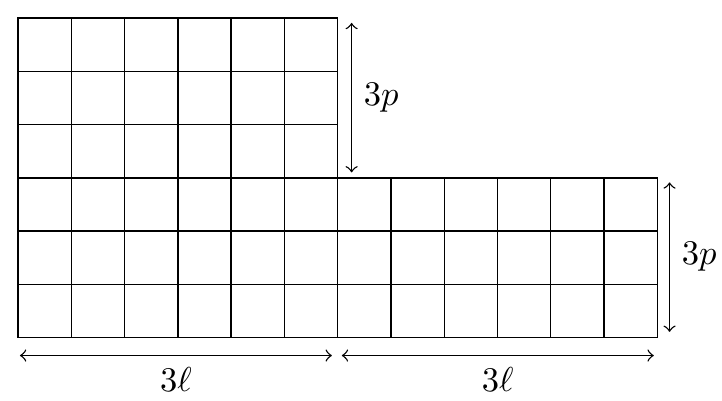} \internallinenumbers
			\caption{Different discrete models converge towards a tableau of slope $-p/\ell$. 
As usual for problems related to urns, many statistics have a sensibility to the initial conditions; 
it is therefore nice that some universality holds: 
the distributions (depending on~$p$,~$\ell$, and the ``zoom factor''~$K$) of our statistics 
have similar tails compared to Mittag-Leffler distributions.} 
		\label{fig:slope} 	
		\end{center}
\end{figure}

As mentioned in the previous section, we can approximate a triangle of slope $-p/\ell$ by a tableau
of parameters $(\ell,p,n)$ but what happens if we approximate it by a tableau of parameters $(K\ell, Kp, n)$ for any ``zoom factor'' $K\in \N$? (See Figure~\ref{fig:slope}.)
In the first case, we obtain as a limit law in the south-east corner $\PGG(p,\ell,p,\ell)$
whereas in the second case, we get the law $\PGG(Kp, K\ell, Kp, K\ell)$ and these two distributions are
different.

In fact, we could even imagine more general periodic patterns as in Theorem~\ref{TheoremCornerGen} 
corresponding to the same macrosopic object. All these models lead to different asymptotic distributions. 
However, we partially have some universal phenomenon in the sense that, although these 
limit distributions are different, they are closely related
by the fact that their tails are similar to the tail of a Mittag-Leffler distribution. 
\endgroup
\begingroup
\begin{definition}[Similar tails] \label{defTails}
One says that two random variables $X$ and $Y$ have similar tails and one writes $X\asymp Y$ if 
\begin{equation}\label{logtail}
 \frac{\log\frac{\E(X^r)}{\E(Y^r)}}{r}\to 0, \text{\qquad as $r\to\infty$}.
\end{equation}
\end{definition}
This definition has the advantage to induce an equivalence relation between random variables
which have moments of all orders: if $X,Y$ are in the same equivalence class, then 
for every $\varepsilon\in(0,1)$, for $r$ large enough, one has
\begin{equation*}
\E(((1-\varepsilon)X)^r)\leq \E(Y^r)\leq \E(((1+\varepsilon)X)^r)\,.
\end{equation*}
In the proof of the following theorem, we give much finer asymptotics than the above bounds.
\begin{theorem}[Similarity with the tail of a Mittag-Leffler distribution] 
 \label{prop:tailsML}
Let $X$ be a random variable distributed as $\PGG([\ell_1,\ldots,\ell_p];b_0,w_0)$ and put 
$\ell = \ell_1+\dots+\ell_p$, $\delta=p/(p+\ell)$. 
Let $Y:=\operatorname{ML}(\delta,\beta)$ 
(where $\operatorname{ML}$ is the Mittag-Leffler distribution defined as in~\eqref{MLmoments} hereafter, with any $\beta >-\delta$).
Then $X$ and $\delta p^{\delta-1}Y$ have similar tails in the sense of Definition~\ref{defTails}.
\end{theorem}
\begin{proof}
First, recall from e.g.~\cite[page~8]{GoldschmidtHaas15} that the Mittag-Leffler distribution $\operatorname{ML}(\alpha,\beta)$ (where $0<\alpha<1$ and $\beta>-\alpha$) is determined
by its moments. Its $r$-th moment has two equally useful closed forms: 
\begin{align}\label{MLmoments}
	m_{\text{ML},r} &= 
\frac{\Gamma(\beta) \Gamma(\beta/\alpha+r)}{\Gamma(\beta/\alpha)\Gamma(\beta+\alpha r)}
= \frac{\Gamma(\beta+1) \Gamma(\beta/\alpha+r+1)}{\Gamma(\beta/\alpha+1)\Gamma(\beta+\alpha r+1)}.
\end{align}
Now, we prove that, for a fixed $\alpha$, the Mittag-Leffler distributions have similar tails.
 From the Stirling's approximation formula, we have
\begin{equation}\label{Stir}
\log \Gamma(\alpha r +\beta) = \alpha r \log(r) +(\alpha \log(\alpha) -\alpha) r + (\beta -\frac{1}{2}) \log(\alpha r)+ \frac{\log(2 \pi)}{2} +O\left(\frac{1}{r}\right).
\end{equation}
Applying this to the moments~\eqref{MLmoments} of the Mittag-Leffler distribution 
$Y=\operatorname{ML}(\alpha,\beta)$, we get
\begin{align}\label{EYr}
\quad \log \E(Y^r) 
= 
(1-\alpha) r \log(r) + \left( -\alpha\,\log(\alpha) +\alpha-1 \right) r
 + (\frac { \beta}{\alpha}-\beta)\log(r)+O(1),
\end{align}
and thus if one compares with another distribution $Y'=\operatorname{ML}(\alpha,\beta')$, this leads to $Y\asymp Y'$.

Next, we prove that $\PGG$ distributions with the same $\delta$ have similar tails.
The moments of $X=\PGG([\ell_1,\ldots,\ell_p];b_0,w_0)$ are given by 
Formula~\eqref{MomentsOfLove}.
Using the approximation~\eqref{Stir}, we get
\begin{align}
\log \E(X^r)
=& (1-\delta) r \log(r) +(1-\delta) \left(\log\left(\frac {\delta}{p}\right)-1\right) \,r\\
&+\left(b_0+s_0 \delta +\frac{(1+\delta) (p-1)}{2}-\frac{\delta}{p} \sum_{j=0}^{p-1} \sum_{k=1}^j \ell_k \right) \log(r) +O(1). \label{EXr}
\end{align}
Here, we see that in fact up to order $O(r)$ only the slope $\delta$ and the period length $p$ play a rôle; 
it is only at order $o(r)$ that $b_0$, $s_0$, and the $\ell_k$ really occur. Thus, 
if we now also consider $X'= \PGG([\ell'_1,\ldots,\ell'_{p'}];b'_0,w'_0)$, we directly deduce $X\asymp\left(\frac{p}{p'}\right)^{\delta-1} X'$.

Finally, we can compare the moments of $X$ (any $\PGG$ distribution associated to a slope $\delta$ and period $p$)
and $Y$ (any Mittag-Leffler distribution with $\alpha:=\delta$) via Formulas~\eqref{EYr} and~\eqref{EXr}, this leads to $X \asymp\delta p^{\delta-1} Y$.
\end{proof}

\begin{remark}
	The tails of this distribution are universal: they 
depend only on the slope~$\delta$ and the period length~$p$. They depend neither on the initial conditions $b_0$ and $w_0$, nor on further details of the geometry of the periodic pattern (the $\ell_i$'s).
\end{remark}

One more universal property which holds for some families of urn distributions is that they possess subgaussian tails,
a notion introduced by Kahane in~\cite{Kahane60} (see also~\cite{KubaSulzbach17} for some urn models exhibiting this behaviour):
\begin{definition}
	A random variable $X$ has subgaussian tails if there exist two constants $c,C > 0$, such that
	\begin{align*}
		\PR(|X| \geq t) \leq C e^{-c t^2}, \qquad t > 0.
	\end{align*}
\end{definition} 

\begin{proposition}
	\label{prop:tailsSubgaussian}
	The $\PGG(p,\ell,b_0,w_0)$ distributions have subgaussian tails if and only if $p \geq \ell$.
\end{proposition}
\begin{proof}
	The $\PGG$ distribution, as defined in Equation~\eqref{ProdGenGamma}, has moments given in Equation~\eqref{m_r}. As derived thereafter, it has moments asymptotically equivalent to 
	\begin{align*}
		(m_r)^{1/r} = ((p+\ell)e)^{(\delta-1)} r^{(1-\delta) } (1+\Landauo(1)).
	\end{align*}
	By \cite[Proposition~9]{Kahane60}, a random variable $X$ has subgaussian tails if and only if there exists a constant $K >0$ such that for all $r>0$ we have $(\E(X^r))^{1/r} \leq K \sqrt{r}$. As $\delta = \frac{p}{p+\ell}$ the claim follows. 
\end{proof}

Another useful notion which helps to gain insight into the limit of Young tableaux is the notion of a \emph{level line}:
let $\mathcal{C}_v$ be the curve separating the cells with an entry bigger than $v$ and the cells with an entry smaller than~$v$
(and to get a continuous curve, one follows the border of the Young tableau if needed; see Figure~\ref{fig:C}).

\begin{figure}[ht!]
		\begin{center}	
			\includegraphics[width=.42\textwidth]{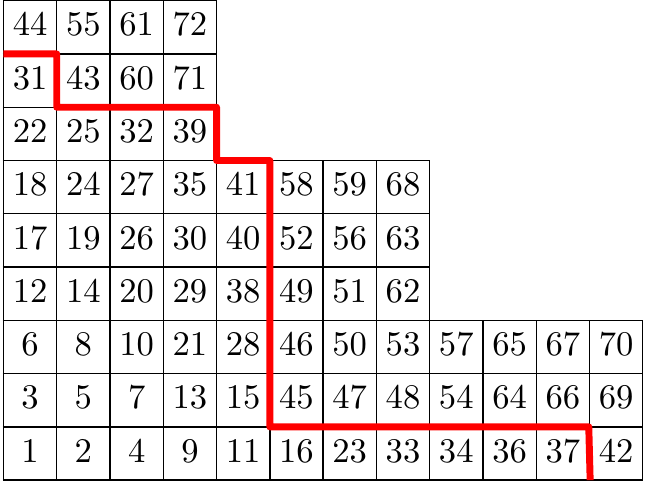} \qquad \qquad 
			\includegraphics[width=.42\textwidth]{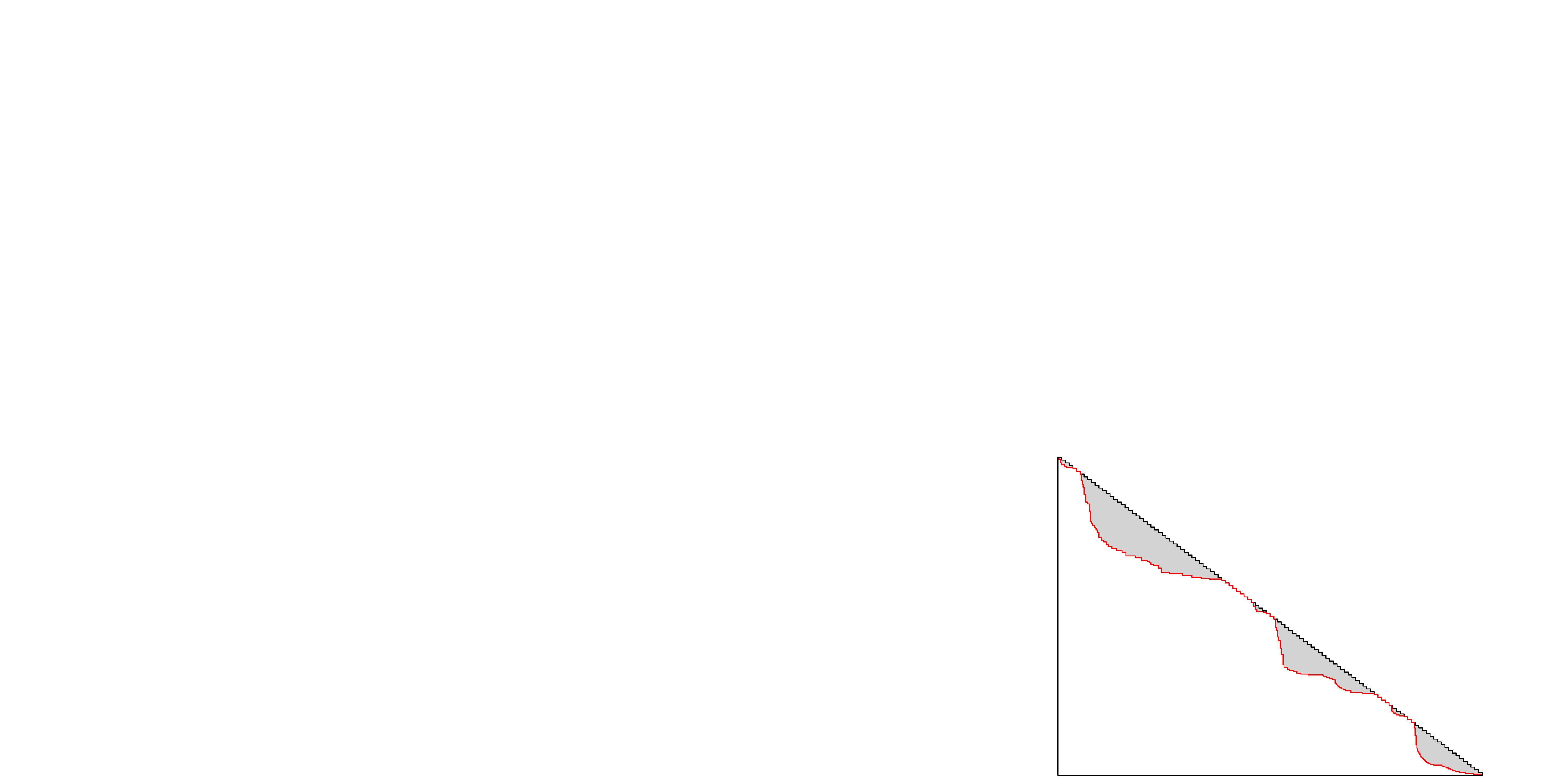} \quad \internallinenumbers
\caption{The level line (in red) of the south-east corner $X_n$: it separates all the entries smaller than $X_n$ from the other ones.
On the left: one example with the level line of $X_n=42$. One the right: the level line of $X_n$, for a very large Young tableau of size $N$ of triangular shape.
The area between this level line and the hypotenuse is the quantity $N-X_n$ analysed in Section~\ref{sec:Tableaux}.}
			\label{fig:C}
		\end{center}
\end{figure}

When $n\to\infty$, one may ask whether 
the level line $\mathcal{C}_{X_n}$ 
converges in distribution to some limiting random curve $\mathcal{C}$. 
If so, the limit laws we computed in Theorem~\ref{TheoremCorner}
would give the (renormalized) area between the macroscopic curve $\mathcal{C}$ and the hypotenuse. 
In particular, the law of $\mathcal{C}$ would depend on the microscopic details of the model, 
since we find for the renormalized area a whole family of distributions $\PGG(p,\ell,b_0,w_0)$ depending on $4$ parameters.
Besides, note that we could imagine even more general microscopic models for the same
macroscopic triangle. For instance, for a slope $-1$, starting from the south-east corner
we could have a periodic pattern 
($1$ step north, $2$ steps west, $2$ steps north, $1$ step west).
All shapes leading to the same slope 
are covered by Theorem~\ref{theo:general} (see also Example~\ref{othermodel}), 
and our method then gives similar, but distinct, limit laws.
Such models thus yield another limit law for the area, and thus another limiting random curve~$\mathcal C$. 

Note that the renormalized area between $\mathcal{C}$ and the hypotenuse does not have the same
distribution as the area below the positive part of a Brownian meander~\cite{Janson07}. 
Funnily, 
Brownian motion theory is cocking a snook at us: 
another one of Janson's papers~\cite{Janson10} studies the area below curves which are related to the Brownian supremum process and, here, 
one observes more similarities with our problem, 
as the moments of the corresponding distribution involve the gamma function. 
However, these moments grow faster than in the limit laws found in Theorem~\ref{TheoremCorner}. 
It is widely open if there is some framework unifying all these points of view.

\subsection{Factorizations of gamma distributions}\label{duality}
\newcommand*{\fnn}{j_m}

With respect to the asymptotic landscape of random Young tableaux, let us add one last result:
our results on the south-east corner directly imply 
similar results on the north-west corner.
In particular, the critical exponent for the upper left corner is $1-\delta$.
In fact, it is a nice surprise that there is even more structure: 
any periodic pattern shape is naturally associated with a family of patterns such 
that the limit laws of the south-east corners of the corresponding Young tableaux are related to each other.

First, let us describe the periodic pattern via a \emph{shape path} $(i_1,j_1; \ldots; i_m, j_m)$: it starts at the north-west corner of the tableau described by the pattern with $i_1$ right steps, followed by $j_1$ down steps, etc.; see Figure~\ref{fig:factorization}.
Then, its cyclic shift is defined by $(j_m, i_1; \ldots; j_{m-1},i_m)$.

Furthermore, this notion is equivalent to Definition~\ref{def:PeriodicYoungTableau} of a periodic tableau via the following formula:
\begin{align*}
	(\ell_1,\ldots,\ell_p) &= (\underbrace{0,\ldots,0,j_m}_{i_m\text{ elements }},\underbrace{0,\ldots,0,j_{m-1}}_{i_{m-1}\text{ elements }},\ldots,\underbrace{0,\ldots,0,j_1}_{i_1\text{ elements }}),
\end{align*} 
Then, the \emph{cyclic shift} is given by
\begin{align*}
	(\ell'_1,\ldots,\ell'_{p'}) &:= (\underbrace{0,\ldots,0,i_m}_{j_{m-1}\text{ elements }},\ldots, \underbrace{0,\ldots,0,i_{2}}_{j_{1}\text{ elements }},\underbrace{0,\ldots,0,i_1}_{j_m\text{ elements }}).
\end{align*} 
In particular we have $p'=\ell$ and $\ell'=p$.

Appending $n$ copies of the shape path $(i_1,j_1; \ldots; i_m, j_m)$ to each other corresponds to $n$ repetitions of the pattern and therefore
gives a periodic tableau.
Note that this new sequence is then equal to the shape of its associated tree similarly to Figure~\ref{fig:tabtreeurn} and in accordance with Definition~\ref{ShapeTree}.

\pagebreak 

\begin{figure}[ht!]
\setlength{\abovecaptionskip}{-1mm} \setlength{\belowcaptionskip}{-5mm}
		\begin{center}	
			\includegraphics[width=0.95\textwidth]{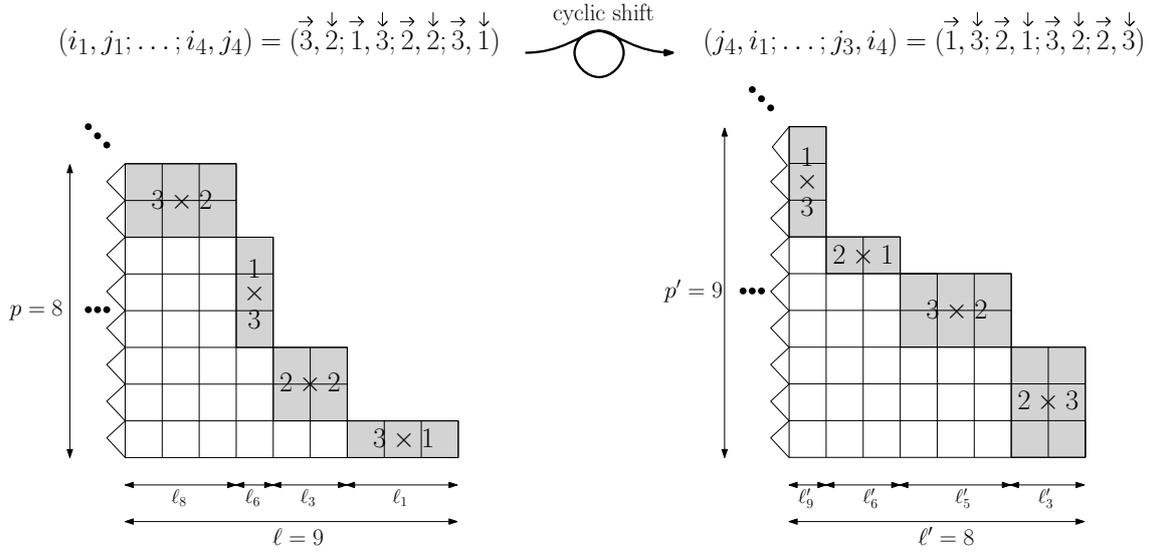} 
			\internallinenumbers
\end{center} \internallinenumbers
			\caption{Example of a cyclic shift on a periodic pattern.
On the left: one sees the shape path (3,2;1,3;2,2;3,1),
it corresponds to the pattern $(\ell_1,\dots,\ell_8)=(3,0,2,0,0,1,0,3)$ (as sequence of consecutive heights, from right to left). 
On the right: one sees its cyclic shift, which corresponds to the pattern 
 $(\ell'_1,\dots,\ell'_9)=(0,0,2,0,3,2,0,0,1)$. In grey we see the size of the sub-rectangles 
described by the shape path, i.e., the $k$-th rectangle has size $i_k \times j_k$.		 
			} 		\label{fig:factorization}
\end{figure}

\begin{proposition}[Factorization of gamma distributions]\label{prop:factorizationOfGamma}
 Let $(\ell_1,\ldots,\ell_p)$ and $(\ell'_1,\ldots,\ell'_{p'})$ be two sequences as defined above 
and let~$\fnn$ be the smallest index such that $\ell_{\fnn} >0$. 
 Let $b_0,w_0$ be two positive integers, and $Y$ and $Y'$ be 
independent random variables with respective distribution 
	$\PGGgen([\ell_1,\ldots,\ell_p];b_0,w_0)$ and $\PGGgen([\ell'_1,\ldots,\ell'_{p'}];b_0+w_0,\fnn)$
	from Theorem~\ref{theo:general}.
Then we have the factorization 
\begin{equation}\label{factor}
YY'\= \frac{1}{p+\ell} {\mathbf \Gamma}(b_0).
\end{equation}
\end{proposition}

\begin{proof}
The equality in distribution is obtained by checking 
the equality of the $r$-th moments and then applying Carleman's theorem:
using Formula~\eqref{m_r} for the moments of $\PGG$
indeed leads (after simplification via the Gauss multiplication formula on the gamma function) to
$\E(Y^r)\E((Y')^r)=\frac{1}{(p+\ell)^r}\E( Z^r)$ where $Z$ is a random variable distributed according to ${\mathbf \Gamma}(b_0)$.
\end{proof}
 
 \begin{remark}[A duality between corners]
One case of special interest is the case of Young tableaux 
having the mirror symmetry
	$(\ell_{\fnn},\ldots,\ell_p) = (\ell_{p},\ldots,\ell_{\fnn})$,
where $\fnn$ is again the smallest index such that $\ell_{\fnn}>0$. Indeed, $Y$ and $Y'$ then correspond to the limit laws for the south-east (respectively north-west) corner of the same 
tableau. In this case, we can think of~\eqref{factor} as expressing a kind of duality between the corners of the tableau.
\end{remark}

Similar factorizations of the exponential law, which is a particular case of the gamma distribution, 
have appeared recently in relation with functionals of L\'evy processes, following~\cite{BertoinYor01}. 
These formulas are also some probabilistic echoes of identities satisfied by the gamma \textit{function}.

\pagebreak

\noindent We can mention one last result in this direction: indeed, 
Theorem~\ref{TheoremCornerGen} used 
for the Young tableau with periodic pattern shape $(\ell_1,\ldots, \ell_p ; 2n)$ 
and the (same) Young tableau 
with periodic pattern shape $(\ell_1,\dots,\ell_p,\ell_1,\dots,\ell_p; n)$ 
leads to two different closed forms of the same limit distribution, and one also gets other closed forms if one repeats $m$ times the pattern $(\ell_1,\ldots, \ell_p)$.
E.g., if one takes all the $\ell_i's$ equal to~$1$, this gives 
\begin{align}
\GG(3,2)&= \sqrt{2} \GG(3,4)\GG(5,4),
\end{align}
and, more generally, 
\begin{align*}
	\GG(s_0+1,2)&= \sqrt{m} \prod_{k=1}^m \GG(s_0+2k-1,2m).
\end{align*}

Using the fact that $\GG(a,1/b)=\bGamma(ab)^{b}$, we can rephrase this identity in terms of powers of $\bGamma$ distributions 
(the notation $\bGamma$, in bold, stands for the distribution, while $\Gamma$ stands for the function; below, we have only occurences in bold):
\begin{align}\bGamma\left(\frac{s_0+1}{2}\right)^{\frac{1}{2}}
&= \sqrt{m} \prod_{k=1}^m \bGamma\left(\frac{s_0+2k-1}{2m}\right)^\frac{1}{2m}.
\end{align}
With $x:=\frac{s_0+1}{2m}$, one gets the following formula equivalent to the Gauss multiplication formula: 
\begin{align}
\bGamma\left(mx\right)^{m}
&= m^m \prod_{k=1}^m \bGamma\left(x+\frac{k-1}{m}\right).
\end{align}

Choosing other values for the $\ell_i$'s leads
to more identities: 
\begin{align*}
 	\prod_{\substack{i=1 \\i\neq\ell_1+\dots+\ell_j+j \\\text{with } 1\leq j \leq p-1}}^{p+\ell-1} \!\!\!\!\! \GG(s_0+i,p+\ell) 
&= m^{1-\delta} \!\!\!\!\! \prod_{\substack{i=1 \\i\neq\ell'_1+\dots+\ell'_j+j \\\text{with } 1\leq j \leq mp-1}}^{m(p+\ell)-1} \!\!\!\!\! \GG(s_0+i,m(p+\ell)).    
\end{align*}
It is pleasant that it is possible to reverse engineer such identities and 
thus obtain a probabilistic proof of the Gauss multiplication formula (see~\cite{Dufresne10}).

\smallskip

This ends our journey in the realm of urns and Young tableaux; in
the next final section, we conclude with a few words 
on possible extensions of the methods used in this article.

\begin{flushright}
\textit{``A method is a trick used twice.''}\\
{\small George P\'olya (1887--1985)}\\
\qquad \\
\textit{``After this the reader who wishes to do so will have no difficulty in developing the theory of \emph{urns}\footnote{The reader is invited to compare with the original citations 
of P\'olya and Young in~\cite[p.~208]{Polya57} and~\cite[p.~366]{GraceYoung03}.}
when they are regarded as differential operators.''}\\
{\small Alfred Young (1873--1940)}
\end{flushright}

\section{Conclusion and further work}\label{Sec6}

In this article, we introduced P\'olya urns with periodic replacements 
and showed that they can be exactly solved with generating function techniques.
The initial partial differential equation encoding their dynamics leads to D-finite moment generating functions,
which we identify as the signature of a generalized gamma product distribution.
It is also pleasant that it finds applications for some 
statistics of Young tableaux.

Many extensions of this work are possible: 
\begin{itemize}
\item The \textbf{density method} which we introduced in~\cite{Marchal18,BanderierMarchalWallner18a} 
can be used to analyse other combinatorial structures, like we did already on permutations, trees, Young tableaux, and Young tableaux with local decreases.
In fact, the idea to use integral representations of order polytope volumes
		in order to enumerate poset structures is quite natural,
		and was used e.g.~in~\cite{Pak01,Elkies03,BaryshnikovRomik10}.
		Our approach, which uses this idea while following at the same time the densities of some parameter,
 allows us to solve both enumeration and random generation. 
We hope that some readers will give it a try on their favourite poset structure!
\item In~\cite{FlajoletGabarroPekari05}, Flajolet et~al.~analyse an urn model which leads 
to a remarkably simple factorization for the history generating function; see Theorem~1 therein and also Theorem~1 in \cite{FlajoletDumasPuyhaubert06}. 
This greatly helps them to perform the asymptotic analysis via \textbf{analytic combinatorics} tools.
Our model does not possess such a factorization; this makes the proofs more involved.
It is nice that our new approach remains generic and can be applied to more general periodic urn models 
(with weights, negative entries, random entries, unbalanced schemes, triangular urns with more colours, multiple drawings, \dots).
It is a full programme to investigate these variants, in order to get a better characterization of the zoo of special functions 
(combination of generalized hypergeometric, etc.) and distributions occurring for the different models.

\item There exists a theory of elimination for partial differential equations, 
chiefly developed in the 1920's by Janet, Riquier, and Thomas (see e.g.~\cite{GerdLangeHegermannRobertz2019,BoulierLemairePoteauxMorenomaza2019} for modern approaches). 
In our case, these approaches however fail to get the linear ordinary differential equations satisfied 
by our generating functions. It is thus an interesting challenge for \textbf{computer algebra} to get an efficient 
algorithm taking as input the PDE and its boundary conditions, and giving as output the D-finite equation (if any). 
Is it possible to extend holonomy theory beyond its apparent linear frontiers? (See the last part of~\cite{PetkovsekWilfZeilberger96}.) 
Also, as an extension of Remark~\ref{remark:FlajoletLafforgue}, it is natural to ask: 
is it possible to extend the work of Flajolet and Lafforgue to the full class of D-finite equations, thus exhibiting new universal limit laws like we did here?

\item Our approach can also be used to analyse the fluctuations of further cells in a random Young tableau.
It remains a challenge to understand the full \textbf{asymptotic landscape of surfaces} associated with \textbf{random Young tableaux}, 
even if it could be globally expected that they behave like a Gaussian free field, 
like many other random surfaces~\cite{Kenyon01}. 
Understanding the fluctuations and the universality of the critical exponent at the 
corner could help to get a more global picture. 
The Arctic circle phenomenon (see~\cite{Romik12}) and 
the study of the level lines $\mathcal{C}$ in random Young tableaux and their possible limits in distribution, as discussed in 
Section~\ref{slope}, seems to be an interesting but very challenging problem.
\end{itemize}

{\bf Acknowledgements.} Let us thank Vadim Gorin, Markus Kuba, C\'ecile Mailler, 
and Henning Sulzbach for kind exchanges on their 
work~\cite{GorinRahman18,KubaSulzbach17,LasmarMaillerSelmi18} and on related questions.
We also thank the referees of the preliminary AofA'2018 version~\cite{BanderierMarchalWallner18}
and of this Annals of Probability version for their careful reading and suggestions which improved the quality of our paper.
We also thank the organizers (Igor Pak, Alejandro Morales, Greta Panova, and Dan Romik) 
of the meeting \textit{Asymptotic Algebraic Combinatorics} (Banff, 11--15 March 2019)
where we got the opportunity to present this work.